\documentclass[11pt]{amsart}
\usepackage{color}
\usepackage{graphicx}
\usepackage[mathcal,mathscr]{euscript}
\usepackage{bbm}
\usepackage{float}
\usepackage{amsmath,amsthm,amssymb,amsfonts,amscd,epsfig,latexsym,graphicx,textcomp}
\usepackage{tikz-cd}
\usepackage{psfrag}
\usepackage[all]{xy}
\usepackage{stackengine}
\usepackage{romannum}
\usepackage{comment}
\usepackage{hhline}
\usepackage{diagbox}
\usepackage{float}
\usepackage{caption}
\usepackage{eufrak}
\usepackage{diagbox}
\usepackage{makecell}
\usepackage{slashed}
\usepackage{MnSymbol}

\restylefloat{figure}

\newtheorem*{theorem*}{Main Theorem}

\usepackage{thmtools}
\usepackage{thm-restate}

\usepackage{hyperref}

\usepackage{cleveref}

\newtheorem{theorem}{Theorem}[section]
\newtheorem{lemma}[theorem]{Lemma}
\newtheorem{corollary}[theorem]{Corollary}
\newtheorem{proposition}[theorem]{Proposition}

\newtheorem{scholium}[theorem]{Scholium}
\newtheorem{question}[theorem]{Question}

\theoremstyle{definition}
\newtheorem{definition}[theorem]{Definition} 
\newtheorem{example}[theorem]{Example} 

\theoremstyle{remark}
\newtheorem{remark}[theorem]{Remark}
\numberwithin{equation}{section}
\newcommand{\ot}{\otimes}
\newcommand{\ra}{\rightarrow}

\newcommand{\BC}{\mathbb{C}}
\newcommand{\BR}{\mathbb{R}}

\DeclareMathOperator{\Ima}{Im}

\def \sign{{\text{sign}}}

\newcommand{\NegPMEdgeDiag}{\raisebox{-0.33\height}{\includegraphics[scale=0.25]{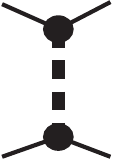}}}
\newcommand{\IIDiag}{\raisebox{-0.33\height}{\includegraphics[scale=0.25]{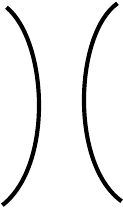}}}

\newcommand{\XDiag}{\raisebox{-0.33\height}{\includegraphics[scale=0.25]{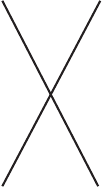}}}
\newcommand{\PMEdgeDiag}{\raisebox{-0.33\height}{\includegraphics[scale=0.25]{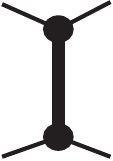}}}

\newcommand{\qdim}{q\!\dim}

\newcommand{\BZ}{\mathbb{Z}}

\newcommand{\ZZ}{{\mathbb Z}}

\newcommand{\CC}{{\mathbb C}}
\newcommand{\NN}{{\mathbb N}}

\newcommand{\QQ}{{\mathbb Q}}

\newcommand{\vertexbracketvertex}{\raisebox{-0.33\height}{\includegraphics[scale=1.2]{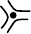}}}
\newcommand{\vertexbracketzero}{\raisebox{-0.33\height}{\includegraphics[scale=1.2]{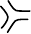}}}
\newcommand{\vertexbracketone}{\raisebox{-0.33\height}{\includegraphics[scale=1.2]{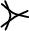}}}

\def\del{\partial}

\def\del{\partial}

\begin{document}
\pagenumbering{arabic}

\title{Quantum state systems that count perfect matchings}

\thanks{}

\author{Scott Baldridge}
\address{Department of Mathematics, Louisiana State University\\
Baton Rouge, LA}
\email{baldridge@math.lsu.edu}

\author{Ben McCarty}
\address{Department of Mathematical Sciences, University of Memphis\\
Memphis, TN}
\email{ben.mccarty@memphis.edu}

\date{}

\begin{abstract}
In this paper we show how to categorify the $n$-color vertex polynomial, which is based upon one of Roger Penrose's formulas for counting the number of $3$-edge colorings of a planar trivalent graph.  Using topological quantum field theory (TQFT), we introduce a quantum state system to build a new bigraded  theory called the bigraded $n$-color vertex homology. The graded Euler characteristic of this homology is  the $n$-color vertex  polynomial.  We then produce a spectral sequence whose $E_\infty$-page is a filtered  theory called filtered $n$-color vertex homology and show that it is generated by certain types of  face colorings of ribbon graphs.  For $n=2$, we show that the filtered $n$-color vertex homology is generated by face colorings that correspond to perfect matchings.  Finally, we introduce and give meaning to what the vertex polynomial counts when $n \geq 2$. This polynomial is a new abstract graph invariant that can be inferred from certain formulas of Penrose.
\end{abstract}

\maketitle

\section{introduction}

In his 1971 paper \cite{Penrose}, Roger Penrose gave several formulas for counting the number of $3$-edge colorings of trivalent graphs.  He described what has come to be known in the literature as the ``Penrose polynomial,'' whose evaluation at $n=3$ counts the number of $3$-edge colorings for planar graphs.  The Penrose polynomial has been studied extensively over the years (cf. \cite{Jaeger,Aigner,EMM,EMMKM,Moffat2013,ColorHomology}). Recently, it was categorified by the authors in \cite{ColorHomology} (see also \cite{BKM2}).  Taking inspiration from Penrose's related system of binors (cf. page 239 of \cite{Penrose}),  we define new polynomials for trivalent graphs: The first is a family of polynomials called the {\em $n$-color vertex polynomial}, $\llangle\Gamma\rrangle_{n}(q)$, which generalizes to $n>2$ what the first author  in \cite{BaldCohomology} called the vertex bracket polynomial for a ribbon graph $\Gamma$ (see the $n=2$ case below). The second, which can be inferred from \cite{Penrose}, is the {\em vertex polynomial}, $V(\Gamma,n)$. This is a polynomial in $n$ that equals the $n$-color vertex polynomial evaluated at one, i.e., $V(\Gamma,n) = \llangle\Gamma\rrangle_{n}(1)$.

For $n=2$, the $n$-color vertex polynomial is given by the following (see \Cref{def:vertexPoly} for the general definition):
\begin{eqnarray}
\bigg\llangle \vertexbracketvertex \bigg\rrangle_{\! 2} &=&  \bigg\llangle \vertexbracketzero \bigg\rrangle_{\! 2} \ - \ q^{3} \bigg\llangle\vertexbracketone \bigg\rrangle_{\! 2}, \mbox{\ and} \\ [.2cm]
\bigg\llangle \bigcirc  \bigg\rrangle_{\! 2} & = & q+1.
\end{eqnarray}
The vertex polynomial, like the Penrose polynomial, has the remarkable property that certain evaluations provide a way to count the number of $3$-edge colorings of the graph (cf. \cite{BaldCohomology,Penrose}).  For example,  $V(\Gamma,2)$ yields a multiple of the number of $3$-edge colorings.  Thus, proving that the $2$-color vertex polynomial is always non-trivial when evaluated at one for planar, bridgeless graphs is equivalent to proving the four-color theorem.  In this paper, we categorify  the $n$-color vertex polynomial for each $n$ to get a family of homology theories.

The categorification of the $n$-color vertex polynomial emerges specifically from tools developed in \cite{ColorHomology}, in which the authors introduced two categorifications of the evaluation of the Penrose polynomial at each positive integer $n\in \NN$, called bigraded $n$-color homology and filtered $n$-color homology. These homology theories, along with the ones defined in this paper, are defined for ribbon graphs, that is, a graph that is the $1$-skeleton of particular types of $2$-dimensional CW complexes of closed smooth surfaces.  They encode structural information about the colorings of the $2$-cells of the surface with $n$ colors (see \Cref{sec:nColor}).  

Briefly, to define the $n$-color vertex homology, start with a vertex ribbon diagram $\Gamma_\bullet$ (cf. \Cref{def:vertex-graph-diagram}) of a trivalent ribbon graph $\Gamma$ of some graph $G(V,E)$. Form a hypercube of states by replacing \vertexbracketvertex \ at each vertex of $\Gamma_\bullet$ with  a vertex $0$-smoothing \vertexbracketzero \  or a vertex $1$-smoothing \vertexbracketone.  The hypercube itself is a $|V|$-regular graph with $2^{|V|}$ vertices. To each vertex of the hypercube, associate a {\em state} that corresponds to an element $\nu=(\nu_1,\ldots,\nu_{|V|})\in \{0,1\}^{|V|}$. This element specifies a set of immersed circles in the plane by whether a vertex $0$- or $1$-smoothing ($\nu_i = 0$ or $\nu_i = 1$) was done at vertex $v_i\in V$.

The hypercube is arranged in columns from the ``all-zero smoothings'' state $(0,0,\ldots,0)\in\{0,1\}^{|V|}$ to the  ``all-one smoothings'' state $(1,1,\ldots,1)\in\{0,1\}^{|V|}$, where the columns consist of the states that have the same number of vertex $1$-smoothings; let $|\nu|$ be that number for each state. Next, form a chain complex by replacing circles in each state by a tensor product of copies of the algebra $V=\mathbbm{k}[x]/(x^n)$ and replace edges of the hypercube with maps between these vector spaces that depend upon what happens to the circles.  These maps gives rise to a differential $\delta$ between columns that turn the hypercube of states into a bigraded chain complex $(C^{i,j}(\Gamma;\mathbbm{k}),\delta)$. This differential preserves the quantum grading (the $j$-grading) and increases the homological grading $i$ by one.  The {\em bigraded $n$-color vertex homology of $\Gamma$}, $VCH^{i,j}_n(\Gamma;\mathbbm{k})$, is then the homology of this complex.  Filtered homology arises from the same hypercube of states in a similar manner but using a different algebra.

This homology categorifies the $n$-color vertex polynomial, which is our first theorem:

\begin{restatable}{Theorem}{gradedEuler}
\label{thm:gradedEuler}
Let $G(V,E)$ be a trivalent graph and $\Gamma$ be a ribbon graph of it.  Let $n\in \NN$ and $\mathbbm{k}$ be a ring in which $\sqrt{n}$ is defined.  The bigraded $n$-color vertex homology of $\Gamma$, $VCH_n^{*,*}(\Gamma;\mathbbm{k})$ is an invariant of the ribbon graph $\Gamma$.  Furthermore, the graded Euler characteristic of it is the $n$-color vertex polynomial:
$$\llangle\Gamma\rrangle_{n} =\chi_q(VCH_n^{*,*}(\Gamma;\mathbbm{k})).$$
\end{restatable}

There is one main issue to proving this theorem:  How to define the maps between vertices in the hypercube in such a way that the diagrams associated to each face commute. One cannot simply use a Frobenius algebra to define the maps as in $n$-color homology. We show that these maps can be defined by embedding the hypercube of vertex states of $\Gamma_\bullet$ in a much larger hypercube (of the ``bubbled blowup'' of $\Gamma$) and taking a composition of maps corresponding to three edges in the larger hypercube to define the map. This allows the power of the TQFT in Section 9 of \cite{ColorHomology} to be used to show that the maps are well defined for all $n$. It also reduces many of the key proofs in this paper to similar lemmas found in Sections 5 and 6 of \cite{ColorHomology}.
In a way, the results of this paper are like a photomosaic: a new image and data emerges by forgetting much of the structure of the larger hypercube. 

To link the bigraded $n$-color vertex homology to face colorings, we define a second differential on the chain complex for the bigraded $n$-color vertex homology. This new differential leads to a spectral sequence with $VCH_n^{*,*}(\Gamma)$ as the $E_1$-page.  The $E_\infty$-page of the spectral sequence is called the filtered $n$-color vertex homology, denoted $\widehat{VCH}_n^{*}(\Gamma)$.  In addition, we prove that the filtered $n$-color vertex homology is generated by coloring the faces of ribbon graphs in which, at every vertex, there are at least two colors represented among the faces incident to the vertex (cf. \Cref{lem:harmonicsToColors}).  Such a coloring is called a partial $n$-face coloring in this paper (cf. \Cref{def:PartialFace}).  When specialized to $n=2$, a partial $2$-face coloring induces a perfect matching, linking quantum state systems to counting perfect matchings.

\begin{restatable}{Theorem}{perfectMatchings}
\label{thm:perfectMatchings}
Let $G(V,E)$ be a connected  planar trivalent graph, $\Gamma$ be a plane graph for $G$, and $\widehat{VCH}_2^{*}(\Gamma)$ its associated filtered $2$-color vertex homology.  Then the rank of the zeroth filtered $2$-color vertex homology group is twice the number of perfect matchings of $\Gamma$, i.e.,
$$\text{dim } \widehat{VCH}_2^{0}(\Gamma) = 2\cdot \#\{\text{perfect matchings of }G\}.$$
\end{restatable}

Theorem~\ref{thm:perfectMatchings} shows that the filtered $2$-color vertex homology is remarkable in that it is the first known homology theory that counts perfect matchings for trivalent graphs.  If we use the entire homology, we can say more.  There is a close relationship between even perfect matchings (recall that a perfect matching $M$ on a trivalent graph $G$ is called even if every cycle of $G\setminus M$ is even length), and $3$-edge colorings.  Thus, by taking the entire filtered $2$-color vertex homology into consideration, we obtain a formula to count the number of 3-edge-colorings of $\Gamma$. 

\begin{restatable}{Theorem}{oddMatchingsCancel}
\label{thm:oddMatchingsCancel}
Let $G(V,E)$ be a connected planar trivalent graph and $\Gamma$ be a plane graph for $G$.  Then 
$$\chi(\widehat{VCH}_2^*(\Gamma)) = 2^{\frac12 |V|} \cdot \# \{ \text{3-edge-colorings of }\Gamma\}.$$
\end{restatable}

Theorem~\ref{thm:oddMatchingsCancel} says that the filtered $n$-color homology may be viewed as a categorification of a famous formula due to Penrose (see page 240 of \cite{Penrose}) and to our knowledge represents the first independent verification of that formula. In fact, our proof explains his formula through the sums of face colorings instead of tensor evaluations (cf. \Cref{sec:VertexLee}).

Finally, we initialize the study of the vertex polynomial. The polynomial is derived by applying a vertex bracket, $V\left( \vertexbracketvertex \right) =  V\left( \vertexbracketzero \right) \ - \ V\left(\vertexbracketone \right)$, inductively to get a formal sum of states and then replacing each immersed circle in each state by a formal factor of $n\in\ZZ$, i.e.,  $V\left( \bigcirc  \right) = n$.  Although the definition of the vertex polynomial appears here for the first time in the literature, the polynomial is an archeological artifact of Roger Penrose's 1971 paper \cite{Penrose}. He must have known about it but did not define it explicitly  in that paper like he did the Penrose polynomial.  In  \Cref{sec:vertexPoly}, we briefly trace through the development of the computations by Penrose from which the polynomial can be inferred. One thing that becomes clear from that discussion is that, while Penrose could make sense of the evaluation of the polynomial at $n=2$ and $n=-2$ in terms of $3$-edge colorings  (Theorem~\ref{thm:oddMatchingsCancel}), the meaning of  the evaluation of the polynomial for $n>2$ was a complete mystery.  This may be one of the reasons why the vertex polynomial did not  ``take off'' in the literature like the Penrose polynomial did. Theorem~\ref{thm:evenOddRibbonGraphs} below provides this meaning by showing that the vertex polynomial is the signed count of  partial $n$-face colorings of oriented surfaces associated to the graph.  
 
Before stating Theorem~\ref{thm:evenOddRibbonGraphs}, we offer a couple of clarifying comments.  First, requiring the automorphism group of $G$ to be trivial is done to make the count specifically about the number of face colorings on ribbon graphs. There is a corresponding statement when it is not trivial (cf. \Cref{theorem:partial-coloring-ribbon-states}), but this requires some knowledge about how the filtered $n$-color vertex homology is defined to compute with it. In this sense, the theorem below can be read and understood without any specialized  knowledge of, say, TQFTs.  Second, the vertex polynomial is an abstract graph invariant.  To  show this, we prove that the vertex polynomial of all oriented ribbon  graphs of a connected trivalent graph are equal up to sign (cf. \Cref{thm:VInvt-of-oriented-ribbon-graph}).  The invariant can then be defined by choosing a ``nonnegative'' oriented ribbon graph (see \Cref{definition:abstract-vertex-poly}), which is an oriented ribbon graph $\Gamma$ such that $V(\Gamma,n)\geq 0$ for large $n\in\NN$. Note that all plane graphs are nonnegative.

\begin{restatable}{Theorem}{evenOddRibbonGraphs}
\label{thm:evenOddRibbonGraphs} 
Let $\Gamma$ be a nonnegative oriented ribbon graph of a connected trivalent graph $G$ with trivial automorphism group. The vertex polynomial of $G$, evaluated at $n\in\NN$, is
$$V(G,n) =  2\left(\# \left\{\mbox{\parbox{1.9in}{partial $n$-face colorings of all \\ distinct oriented {\bf\em even} ribbon graphs with respect to $\Gamma$}}\right\} \ - \ \# \left\{\mbox{\parbox{1.86in}{partial $n$-face colorings  of all distinct oriented {\bf\em odd} ribbon graphs with respect to $\Gamma$}}\right\}\right)$$
where a ribbon graph is even or odd with respect to $\Gamma$ based upon the parity of the number of half-twists that need to be inserted into its bands to make it equivalent to $\Gamma$.
\end{restatable}

It is interesting to compare this theorem to Theorem~E of \cite{ColorHomology} on the Penrose polynomial. The face colorings  counted here are partial $n$-face colorings (see \Cref{def:PartialFace}) instead of proper $n$-face colorings counted by the Penrose polynomial. Thus, while faces with three distinct colors must be present at every vertex to be in the count of Theorem~E, this theorem allows for same-color faces to also be present at a vertex. Thus, the two polynomials are very different even though both provide information about $3$-edge colorings at $n=-2$  (cf. \Cref{eq:equalities-of-vertex-and-Penrose}).  For more details regarding the history of the vertex polynomial and some of its properties, see \Cref{sec:vertexPoly}.

\subsection{Organization of the paper}
The remainder of the paper is organized as follows.  \Cref{Sec:perfect-matching-graphs} outlines the necessary background on ribbon graphs.  \Cref{sec:polynomials} introduces the vertex polynomials under consideration.  \Cref{sec:nColor} reviews the construction of bigraded $n$-color homology.  \Cref{sec:vertexHomology} defines bigraded $n$-color vertex homology, and \Cref{sec:VertexLee} defines the filtered version, the spectral sequence, and derives various results using those tools.  \Cref{sec:TotalMatching} provides a discussion the vertex polynomial as well as a new polynomial invariant of a graph, which is the Poincar\'{e} polynomial of the filtered $n$-color vertex homology.  \Cref{sec:4valent} gives a brief discussion on how to define the homologies for $4$-regular graphs.  

Many of the tools used in this paper are derived directly from results in \cite{ColorHomology}.  To help the reader with the proofs in this paper, we often point to the relevant theorems and proofs there.  We recommend that the reader have that paper available for easy reference while reading this one.

\section{Perfect Matching Graphs}\label{Sec:perfect-matching-graphs}
In this section we recall the definition of a perfect matching graph, which is an equivalence class of decorated trivalent ribbon graphs. A {\em plane graph} $\Gamma$ is an embedding, $i:G\ra S^2$, of a connected, planar graph $G$ into the sphere. Ribbon graphs generalize plane graphs in that they allow for both non-planar embeddings of plane graphs into higher genus surfaces and for embeddings of nonplanar graphs into higher genus surfaces while retaining a key aspect of plane graphs: $S^2 \setminus i(G)$ is a set of disks.  Full details of the relationship between ribbon graphs and diagrams can be found in \cite{BKR}.

An abstract graph $G(V,E)$ is equivalent to a 1-dimensional CW complex where vertices of $V$ are identified with points. Each edge $\{v_i, v_j\}$ is identified with a unit interval $I=[0,1]$, and edges are glued together at coincident vertices.  For the sake of this paper, all graphs are {\em multigraphs}, which may have loops (edges with a single incident vertex) and multiple edges incident to the same two distinct vertices. Unless otherwise stated in a theorem, all graphs are assumed to be connected. Finally, ``vertex-less'' edges are allowed, i.e., edges that can be represented in the plane as (possibly immersed) circles and are also considered loops.

\subsection{Ribbon graphs}\label{Subsec:ribbongraphs}
A perfect matching graph is an equivalence class of trivalent graphs with extra structure. One of these structures is that of a ribbon graph. For a detailed introduction to ribbon graphs see \cite[Section 1.1.4]{Moffat2013}. 
\begin{definition}\label{Def:ribbongraph}
A {\em ribbon graph of a graph $G$} is an embedding $i:G\ra \Gamma$ where $G$ is thought of as a $1$-dimensional CW complex and $\Gamma$ is a surface with boundary where $\Gamma$ deformation retracts onto $i(G)$.  We say that \( G \) is the \emph{underlying graph} of \( \Gamma \), and that \( \Gamma \) is \emph{the surface associated to} the ribbon graph. 
\end{definition}

We will often refer to the ribbon graph simply by $\Gamma$ and think of $\Gamma$ as a surface with an embedded graph $G$. An orientation of a ribbon graph, if one exists, is an orientation of the surface. For defining the vertex homology we will generally begin with an oriented ribbon graph, but will also make use of nonorientable ribbon graphs in defining some of the complexes. Let \( \overline{\Gamma} \) denote the closed smooth surface obtained by attaching discs to the boundary of $\Gamma$. The embedding of $G$ into the surface \( \overline{\Gamma} \) is known as a {\em 2-cell embedding}.\footnote{Such an embedding is also known as a {\em cellular embedding} or {\em cellular map}.}

Let \( \Gamma_1 \) and \( \Gamma_2  \) be ribbon graphs.  We say that \( \Gamma_1\) and \( \Gamma_2 \) are \emph{equivalent} ribbon graphs if there is a homeomorphism \( f : \overline{\Gamma}_1 \rightarrow \overline{\Gamma}_2 \)  that induces an isomorphism from  $G_1$  to $G_2$. Thus, one can define the {\em genus of a ribbon graph $\Gamma$} to be the genus of the associated closed smooth surface $\overline{\Gamma}$. 

Ribbon graphs get their name from the topological construction of attaching bands (ribbons) to disks.  Given a graph \( G \), a ribbon graph $\Gamma$ is determined by a cyclic ordering of the edges at every vertex and a sign for each edge: The ribbon graph is obtained by taking a disk for every vertex of \( G \),  gluing bands as prescribed by the edges and their cyclic ordering and inserting a half twist for each negative edge. Thus, the vertices (edges) of \( G \) are in bijection with the discs (bands) of \( \Gamma \), and we shall not distinguish between them, referring to them as the \emph{vertices} and \emph{edges} of \( \Gamma \). The {\em faces} of $\Gamma$ are the disks given by the complement of the graph in $\overline{\Gamma}$.

\Cref{Fig:ribbons} shows that two distinct ribbon graphs may have the same underlying abstract graph. These ribbon graphs are distinguished by the number of boundary components of their associated surfaces.

\begin{figure}
\includegraphics[scale=0.75]{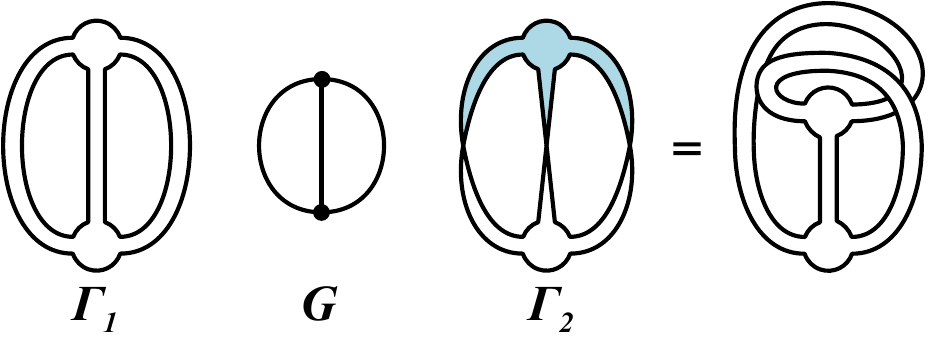}
\caption{Distinct ribbon graphs \( \Gamma_1 \) and \( \Gamma_2 \) with the same underlying graph \( G \).}
\label{Fig:ribbons}
\end{figure}

In this paper,  ribbon graphs are represented by the following diagrams.
\begin{definition}[Ribbon diagram, cf. \cite{BKR}]\label{Def:ribbondiagram}
A \emph{ribbon diagram} is a graph drawn in the plane (with possible intersections between its edges), with vertices decorated by circular regions, \raisebox{-6pt}{\includegraphics{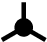}}, to distinguish between vertices and edge intersections. A cyclic ordering of the edges at a vertex is given implicitly by such a diagram, i.e.,\ it is given by their ordering in the plane. 
\end{definition} 

An example of a ribbon diagram is given in \Cref{Fig:ribbondiagram}.  A ribbon diagram can be used to construct an oriented ribbon graph using disks for vertices and bands for edges.  If \( \Gamma \) is obtained from a ribbon diagram, \( D \), in this manner we say that \( D \) \emph{represents} \( \Gamma \).  One could also define a ribbon graph as an equivalence class of ribbon diagrams, up to a set of ribbon moves (cf. \cite{BKR}).  In this sense, we may think of ribbon graphs and ribbon diagrams (modulo equivalence via the ribbon moves) as the same thing and use the terms interchangeably.

If desired, a {\em sign} of plus or minus one can be associated to each edge. For each negative edge, a half twist is introduced into the band that is glued to the disks of its incident vertices. A positive edge is denoted by a solid line and a negative edge denoted by a dotted line. Such graphs are called {\em signed ribbon diagrams} or {\em signed rotation systems} in the literature (cf. Section 6.7 of \cite{ColorHomology}, see also \cite{Moffat2013} or \cite{MT}). Since theorems in this paper involve applying local operations in going from $0$-smoothings to $1$-smoothings, which do not effect crossing data coming from the twisted bands, all theorems apply to signed ribbon diagrams. Thus, we only refer to ribbon graphs in the theorems, which may require ribbon diagrams with negative edges to represent them.

\begin{figure}
\includegraphics{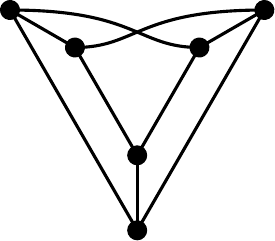}
\caption{A ribbon diagram of a $K_{3,3}$ ribbon graph. Note that the closed oriented surface associated to the ribbon graph is a torus.}
\label{Fig:ribbondiagram}
\end{figure}

\subsection{Perfect matching graphs}\label{Subsec:pm-graphs}
In \cite{BaldCohomology}, a plane graph with a perfect matching was called a perfect matching graph.  The notion of a perfect matching graph  can be generalized to any ribbon graph by decorating the ribbon diagram with a perfect matching (cf. ``matched diagram'' of \cite{BKR} for an example of a perfect matching graph with further decorations). A perfect matching is a set of edges that ``match'' every vertex to exactly one other vertex:

\begin{definition}\label{Def:pm}
A \emph{perfect matching} of an abstract graph $G(V,E)$ is a subset of non-loop edges of the graph, $M\subset E$, such that each vertex is incident to exactly one edge in the subset. 
\end{definition}
The term {\em matching} is used in graph theory for any subset of non-loop edges of the graph where each vertex is incident to either zero and one edge in the subset; the term \emph{perfect} here refers to the fact that every vertex is incident to exactly one matched edge. 

The objects of the main theorems of this paper are equivalence classes of ribbon graphs.  Perfect matching graphs are used to define invariants for them.  An ordered pair of a ribbon graph with a perfect matching, $(\Gamma, M)$, is called a perfect matching graph, and represented using a ribbon diagram.  Examples of perfect matching graphs are given in \Cref{Fig:matcheddiagrams}.   

\begin{definition}\label{Def:pm-graph}
A \emph{perfect matching graph}, denoted $\Gamma_M$, is a ribbon graph, $i:G\ra \Gamma$, together with a perfect matching $M$ of the graph $G$. We represent the perfect matching in a ribbon diagram of $\Gamma$ using thickened edges. 
\end{definition}

\begin{figure}[H]
\includegraphics[scale=0.75]{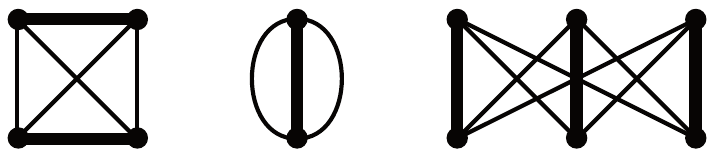}
\caption{Perfect matching graphs.}
\label{Fig:matcheddiagrams}
\end{figure}

\subsection{The blowup of a graph}\label{Subsec:blowup} While perfect matching graphs play a central role in many of the constructions in this paper, we will generally work with the {\em blowup of the graph}, which has a canonically defined perfect matching.   

\begin{definition}
Let $G(V,E)$ be a trivalent graph and $\Gamma$ be a ribbon graph of $G$ represented by a ribbon diagram.  Define the {\em blowup of $\Gamma$}, denoted $\Gamma^\flat$, to be the ribbon diagram given by replacing every vertex of $\Gamma$ with a cycle as in the following picture:
\begin{center}
\includegraphics[scale=.55]{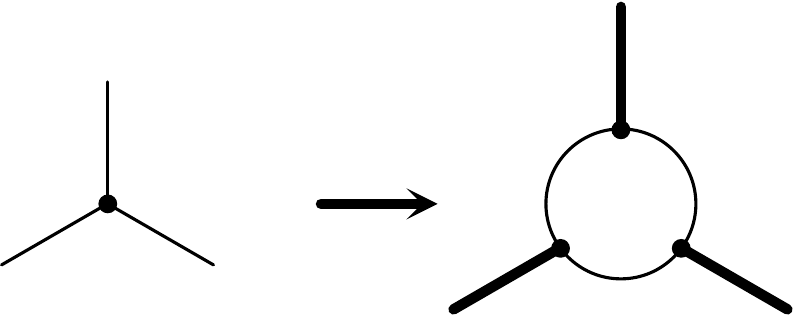}
\end{center}
The blowup has a canonical perfect matching given by the original edges of $\Gamma$. Thus, the resulting perfect matching graph is $\Gamma_E^\flat$.  \label{def:blowup-of-a-graph}
\end{definition}

In \Cref{sec:vertexHomology}, we will construct the vertex homology.  The following definition will be of use in that context.

\begin{figure}[H]
\includegraphics[scale=.85]{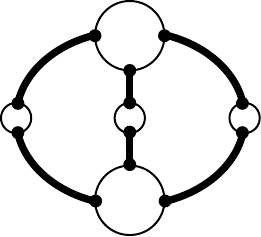}
\caption{The bubbled blowup of the theta graph.}
\label{fig:bubbledBlowUpTheta}
\end{figure}

\begin{definition}
Let $G(V,E)$ be a trivalent graph and $\Gamma$ be a ribbon graph of $G$ represented by a ribbon diagram.  Define the {\em bubbled blowup of $\Gamma$}, denoted $\Gamma^B$, to be the ribbon diagram given by replacing every perfect matching edge of $\Gamma_E^\flat$ with two perfect matching edges and a $2$-cycle as in the following picture:
\begin{center}
\includegraphics[scale=.5]{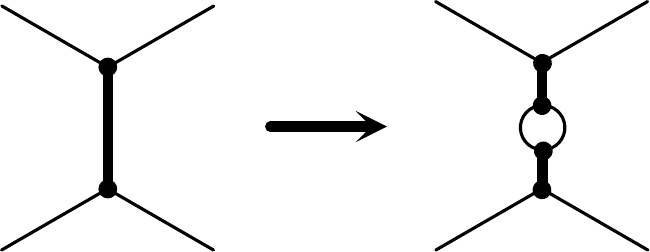}
\end{center}
If $E$ is the set of edges of $G$, then $\Gamma^B$ has a canonical perfect matching $E'$ which consists of two edges for every edge as shown above (the two ``half-edges'' of $e$).  We will refer to the perfect matching graph $\Gamma_{E'}^B$ as simply $\Gamma^B$.   \label{def:bubbled-blowup-of-a-graph}
\end{definition}

Most of the invariants in this paper are defined using a perfect matching graph. There are often many different perfect matchings for the same graph.  However, the canonical perfect matchings given by the blowup and bubbled blowup lead to invariants that depend only on the ribbon structure.  When the graph is also trivalent, Theorem 6.17 of \cite{ColorHomology}, combined with these canonical perfect matchings, allows the construction of abstract graph invariants (cf. \cite{BKM2}). \\

\section{The $n$-Color Vertex Polynomial of a Ribbon Graph}\label{sec:polynomials}
We now define the $n$-color vertex polynomial of a connected trivalent ribbon graph $\Gamma$ of an abstract graph $G(V,E)$. The construction involves resolutions that occur at the vertices of the ribbon graph. All definitions of invariants in this paper are based upon ribbon diagrams, which are then shown to be independent of the ribbon diagram chosen.  For example, the vertex polynomials and vertex homology theories are defined from the vertex ribbon diagram, which we define next.  

\begin{definition} \label{def:all-zero-state}
Let $G(V,E)$ be an abstract trivalent graph with a ribbon graph given by a ribbon diagram $\Gamma$.  Let $\Gamma^\flat_E$ be the blowup of $\Gamma$ with its canonical perfect matching $E$.  The {\em all-zero state} is the result of replacing each positive edge $\PMEdgeDiag$ in $\Gamma^\flat_E$ with a  $\IIDiag$ and replacing each negative edge   $\NegPMEdgeDiag$ in $\Gamma^\flat_E$  with a $\XDiag$. 
\end{definition}

Note that the circles in the all-zero state correspond to the faces of the ribbon graph $\Gamma$.

\begin{definition} \label{def:vertex-graph-diagram}
Let $G(V,E)$ be an abstract trivalent graph with a ribbon graph given by a ribbon diagram $\Gamma$. Let $\Gamma^\flat_E$ be the blowup of $\Gamma$ with its canonical perfect matching $E$.  The {\em vertex ribbon diagram}, denoted $\Gamma_\bullet$, is the all-zero state of $\Gamma^\flat_E$ together with dots  placed where each vertex of the graph was before taking the all-zero resolution, i.e., a $\bullet$ is placed in the region where the three face(s) are incident to the vertex (see the right-hand side  of \Cref{fig:blowup-and-vertex}). 
\end{definition}

\begin{definition}\label{def:vertexPoly}
Let $\Gamma$ be a ribbon diagram for a connected trivalent graph $G(V,E)$ and $n\in\NN$. Let $m= \frac{n}{2}$ if $n$ is even and $\frac{n-1}{2}$ if $n$ is odd.  The {\em $n$-color vertex polynomial}, denoted $\llangle \Gamma \rrangle_{n}$, is characterized by applying the following rules to the vertex ribbon diagram $\Gamma_\bullet$:

\begin{eqnarray}
\bigg\llangle \vertexbracketvertex \bigg\rrangle_{n} &=&  \bigg\llangle \vertexbracketzero \bigg\rrangle_{n} \ - \ q^{3m} \bigg\llangle\vertexbracketone \bigg\rrangle_{n} \label{eq:vertex-bracket}\\ [.2cm]
\label{eq:immersed_circle}
\bigg\llangle \bigcirc  \bigg\rrangle_{n} & = &\left\{ \begin{array}{ll} q^m+\cdots+q+1+q^{-1}+\cdots+q^{-m+1} \ \ \ \ & \mbox{ if $n$ even} \\[.3cm]
q^m+\cdots+q+1+q^{-1}+\cdots+q^{-m+1}+q^{-m} & \mbox{ if $n$ odd}\\
\end{array}\right. \\[.2cm]
\bigg\llangle \Gamma_1 \sqcup \Gamma_2 \bigg\rrangle_{n}&=& \bigg\llangle \Gamma_1 \bigg\rrangle_{n} \cdot \bigg\llangle \Gamma_2 \bigg\rrangle_{n}\label{eq:vertex-disjoint-union}
\end{eqnarray}

\end{definition}

In Equation~\eqref{eq:vertex-bracket}, the \vertexbracketzero \  will be referred to as the {\em vertex 0-smoothing} while the \vertexbracketone will be referred to as the {\em vertex 1-smoothing}.  When inductively applying \Cref{eq:vertex-bracket}, any immersed circle that appears will be referred to as a \emph{loop}.  The expression in Equation~\eqref{eq:immersed_circle} will be referred to as the \emph{loop polynomial}, and the evaluation of the loop polynomial at $q=1$ will be referred to as the \emph{loop value}. In fact, the loop polynomials in Equation~\ref{eq:immersed_circle} are just Chebyshev polynomials, $\Delta_{n-1}$, defined inductively by $\Delta_0=1$, $\Delta_1=d$, and $\Delta_{n+1}=d\Delta_n - \Delta_{n-1}$,  for the choice of $d=q^\frac12+q^{-\frac12}$, and then shifted by $q^{\frac12}$ when $n$ is even: 
\begin{equation}
\bigg\llangle \bigcirc  \bigg\rrangle_{n} = \left\{ \begin{array}{ll} q^{\frac12}\Delta_{n-1} \ \ \ \ & \mbox{ if $n$ even,} \\[.3cm]
\Delta_{n-1} & \mbox{ if $n$ odd.} \end{array}\right.
\end{equation}
Hence, the loop polynomials for $n$ odd are just the usual quantum integers as defined in \cite{MOY}, i.e., $\llangle \bigcirc\rrangle_{n}=[n]$.  However, when $n$ is even, where many of the interesting face coloring results occur for graphs on surfaces, the loop polynomials must be shifted, i.e., $\llangle \bigcirc\rrangle_{n}=q^{\frac12} [n]$, for the polynomials to be categorified into the homology theories of this paper. This hints at how the quantum invariants of surfaces differ from those in knot theory. 

If desired, one could use even gradings for our homology theories, like in \cite{KhoRoz1} and \cite{KhoRoz2}, by starting with $d=q+q^{-1}$ and setting $\llangle \bigcirc\rrangle_{n}=\Delta_{n-1}$ for $n$ odd and  $\llangle \bigcirc\rrangle_{n}= q \, \Delta_{n-1}$ for $n$ even. Everything in this paper will go through with only superficial modifications (see Remark 9.5 in \cite{ColorHomology}). Note that the loop polynomial will still need to be shifted by $q$ when $n$ is even to get valid quantum invariants of surfaces. 

It is instructive to calculate the $2$-color vertex polynomial (and hypercube generated by it) for the theta graph $\theta$. First, the blowup of $\theta$ and the  vertex ribbon diagram $\Gamma_\bullet$ are shown in \Cref{fig:blowup-and-vertex}.

\begin{figure}[H]
\psfragscanon
\psfrag{b}{$\flat$}\psfrag{=}{$=$}
\psfrag{G}{$\Gamma_\bullet  \ = $}
\includegraphics[scale=.7]{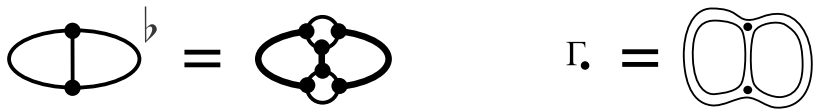}
\caption{The blowup and vertex ribbon diagram of the theta graph.} \label{fig:blowup-and-vertex}
\end{figure}

\noindent Using Equation~\eqref{eq:vertex-bracket} on $\Gamma_\bullet$ in \Cref{fig:blowup-and-vertex} gives four states, which can be arranged into a hypercube as shown in \Cref{fig:vertex-state-of-theta}. Later, it will be useful to identify each of these states with a corresponding state in the hypercube of states for the bubbled blowup (cf. \Cref{subsec:cube-of-vertex-res}).  
\begin{figure}[H]
\psfragscanon
\psfrag{b}{$\flat$}\psfrag{=}{$=$}
\psfrag{G}{$\Gamma_\bullet  \ = $}
\includegraphics[scale=.17]{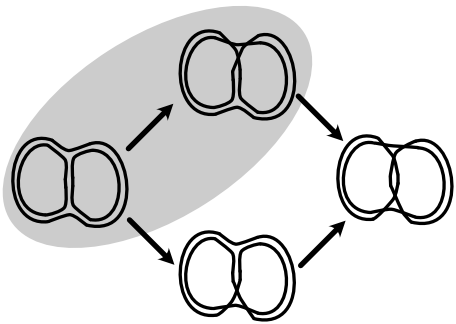}
\caption{The hypercube of vertex states for the  theta graph.}\label{fig:vertex-state-of-theta}
\end{figure}
The $2$-color vertex polynomial can be calculated from the hypercube to get, \begin{equation} \llangle \Gamma \rrangle_{2} (q)=(q+1)^3 - q^3(q+1) - q^3(q+1)+q^6(q+1)^3,\label{eq:vertex-poly-of-theta}
\end{equation}
or $\llangle \Gamma \rrangle_{2}(q) = 1+3q + 3q^2 - q^3 - 2q^4 + q^6 + 3q^7 + 3q^8+ q^9$.

As we will see,  Theorem~\ref{thm:gradedEuler}  and Theorem~\ref{thm:oddMatchingsCancel} together imply that $ \llangle \Gamma \rrangle_{2}(1) =  12$. Evaluating \Cref{eq:vertex-poly-of-theta} at $q=1$ confirms this.  \Cref{app:computations} includes  Mathematica code for computing the  $n$-color  vertex polynomials for $n=2,3,$ and $4$. In particular, \Cref{ex:theta} provides the VPD code for this example.

We conclude this section by observing that, by evaluating the $n$-color vertex polynomial at $q=1$, one obtains a related polynomial that we call, simply, the \emph{vertex polynomial}.\footnote{The term ``vertex polynomial'' shows up a couple of times in the literature, but it appears to be more of an accounting device for recording a tuple than a polynomial that has meaning when evaluated.}

\begin{definition}\label{defn:VertexPenrose}
Let $G(V,E)$ be an abstract trivalent graph with ribbon diagram $\Gamma$.  Let $\Gamma_\bullet$ be the vertex ribbon diagram of $\Gamma$.  The {\em vertex polynomial}, $V(\Gamma,n)$, is characterized by applying the following rules to the vertex ribbon diagram $\Gamma_\bullet$ for $n\in\ZZ$:
\begin{eqnarray}
V\left( \vertexbracketvertex \right) &=&  V\left( \vertexbracketzero \right) \ - \ V\left(\vertexbracketone \right) \label{eq:vertexP-bracket}\\ [.2cm]
\label{eq:VPimmersed_circle}
V\left( \bigcirc  \right)& = & n\\[.2cm]
V(\Gamma_1 \sqcup \Gamma_2)&=& V(\Gamma_1) \cdot V( \Gamma_2)\label{eq:vertexP-disjoint-union}
\end{eqnarray}
\end{definition}

Several properties of the vertex polynomial will be discussed in \Cref{sec:vertexPoly}. Also, we show how to upgrade this definition to an abstract graph invariant with vertices of any valence.

\begin{remark}\label{rem:definingMaps}
Note the edge in the shaded region of \Cref{fig:vertex-state-of-theta}.  Following the main ideas in \cite{BaldCohomology} and \cite{ColorHomology}, to define a homology theory (see also \cite{Kho}), one associates a vector space $V$ to each circle in a state and defines a map between tensor products, which in this case would be $V^{\otimes 3}\rightarrow V$.  Since this map involves all three circles, it is not clear how to define this map using a Frobenius algebra with multiplication and comultiplication maps. In this paper, we embed the hypercube of vertex states into a hypercube of states used in the bigraded $n$-color homology, which we describe next.
\end{remark}

\section{Bigraded $n$-Color Homology}
\label{sec:nColor}

We will need the bigraded $n$-color homology to define the $n$-color vertex homology.  We recall the basic construction of the bigraded $n$-color homology here and refer the reader to \cite{ColorHomology} for more details.  In addition to the homological grading, there is also a quantum grading.  The {\em graded (or quantum) dimension}, $\qdim$, of a graded vector space $V=\oplus_i V^i$ is the polynomial in $q$ defined by $$\qdim(V)=\sum_i q^i\dim(V^i).$$ 

\noindent For a graded vector space $V$, we can shift the grading by $\ell$ to get a new graded vector space, $V\{\ell\}$, defined by $$(V\{\ell\})^m=V^{m-\ell}.$$
Clearly, $\qdim(V\{\ell\}) = q^\ell\cdot\qdim(V)$.  We allow  $\qdim$ to be a polynomial in integer powers by shifting by integer amounts.\\

The complex is based on the $n$-dimensional vector space $V=\mathbbm{k}[x]/ (x^n) = \langle 1, \ldots, x^{n-1} \rangle$ where $\mathbbm{k}$ will usually be taken to be $\CC$.  With some exceptions, $\mathbbm{k}$ cannot be $\ZZ$ or $\QQ$ since the ring must contain $\sqrt{n}$.  The grading is given in \Cref{table:quantum-gradings} depending on the parity of $n$ ($m=n/2$ if $n$ is even, $m=(n-1)/2$ if $n$ is odd).
\medskip
\begin{table}[H]
\begin{tabular}{|l|l|}
\hline
{\bf $n$ even, $m=n/2$} & {\bf $n$ odd, $m=(n-1)/2$ }\\ \hline
$\deg 1 =m$ & $\deg 1 = m$ \\ 
$\deg x = m-1$ & $\deg x = m-1$\\
\ \ $\vdots$ & \ \ $\vdots$\\ 
 $\deg x^m = 0$ &  $\deg x^m = 0$\\
\ \ $\vdots$ & \ \ $\vdots$\\ 
$\deg x^{n-1} = 1-m$ & $\deg x^{n-2} = 1-m$\\
\mbox{} & $\deg x^{n-1} = -m$\\
\hline
\end{tabular}
\caption{Quantum gradings of $V$ when $n$ is even or odd.}
\label{table:quantum-gradings}
\end{table}

\noindent Note that the quantum grading $\qdim V^{\otimes k}$ is the $k^{th}$ power of the loop polynomial (see Equation~\eqref{eq:immersed_circle}).

\subsection{Smoothings, states, and hypercubes}
\label{section:smoothings-states-hypercubes}
Let $G(V,E)$ be a planar trivalent graph with perfect matching $M$. Let $\Gamma_M$ be any perfect matching graph of it.
Index the edges of $M$ from one to $\ell$: $M=\{e_1,\dots,e_\ell\}$.  Let $\Gamma_\alpha$ be a state indexed by $\alpha\in\{0,1\}^\ell$, where each $\alpha_i$ in $\alpha=(\alpha_1,\dots, \alpha_\ell)$ represents doing either a $0$-smoothing  $\IIDiag$ or $1$-smoothing  $\XDiag$ at $e_i$.\\

It is useful to conceptualize the set of states as a hypercube in which each state is a vertex of the cube with edges between states  determined as follows: A directed edge from $\Gamma_\alpha$ to $\Gamma_{\alpha'}$ in the hypercube occurs when $\alpha_i = \alpha_i'$ for all $i$ except one edge $e_k\in M$ where $\alpha_k =0$ and $\alpha_k'=1$.  Label this edge by a tuple of $0$'s and $1$'s given by the $\alpha_i$'s and $\alpha_i'$'s that are the same and a ``$\star$'' for the $k$th position where $\alpha_k = \alpha_k' -1$.  For example, \Cref{fig:P3L} shows the hypercube of states for the blowup of the $\theta$ graph. In general, we call this conceptualization the {\em hypercube of states of $\Gamma$}.

\begin{figure}[h]
\includegraphics[scale=.8]{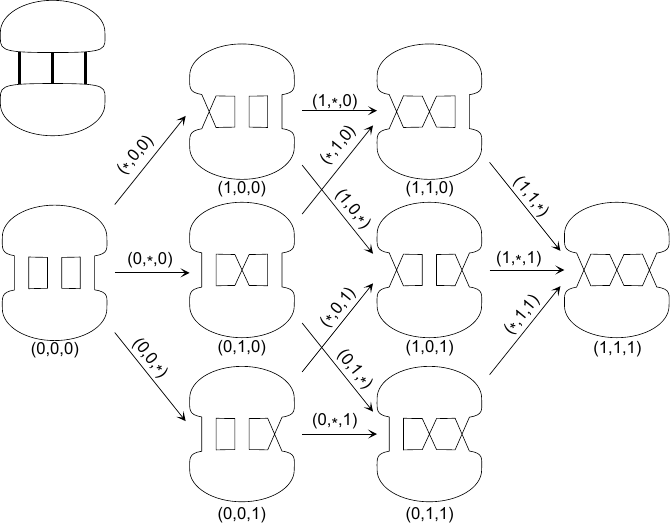}
\caption{Hypercube of states for the blowup of the $\theta$ graph.}
\label{fig:P3L}
\end{figure}

For $\alpha\in\{0,1\}^\ell$, $k_\alpha$ is the number of circles in the state $\Gamma_\alpha$, and $|\alpha|$ is the sum of the $\alpha_i$'s.  For example, in \Cref{fig:P3L}, for $\alpha  = (1,0,0)$, $k_\alpha =2$, and $|\alpha| = 1$.

\subsection{The differential chain complex for $\Gamma$}  
\label{subsection:differential}

To build the chain complex for the $n$-color homology,  associate a graded vector space to the states of a perfect matching graph $\Gamma_M$.  For each $\alpha\in\{0,1\}^\ell$, the associated graded vector space is $$\large V_\alpha = \large V^{\otimes k_\alpha}\!\!\left\{m |\alpha|\right\}.$$
Define the complex $C^{*,*}(\Gamma_M)$ by

$$C^{i,*}(\Gamma_M)=\bigoplus_{\substack{\alpha\in\{0,1\}^\ell \\ i=|\alpha|}}V_\alpha.$$
The homological and internal gradings ($q$-grading) are both integer valued, with the latter being defined by the grading of the elements in $V_\alpha$.  For a monomial element $v\in V_\alpha\subset C^{*,*}(\Gamma_M)$, the homological grading $i$ and the $q$-grading $j$ satisfy:
\begin{eqnarray*}
i(v) & =& |\alpha|,\\
j(v) &=& \deg(v)+ m |\alpha|,
\end{eqnarray*}
where $\deg(v)$ is the degree of $v$ as an element of  $V^{\otimes k_\alpha}$ of $V_\alpha$ prior to shifting the grading by $m |\alpha|$. The complex is trivial outside of $i=0, \dots, \ell$. 

To define the differential for the bigraded $n$-color homology, $\del: C^{i,*}(\Gamma_M) \rightarrow C^{i+1,*}(\Gamma_M)$, consider each edge $\Gamma_\alpha \rightarrow \Gamma_{\alpha'}$ in the hypercube and define a map, $\del_{\alpha\alpha'}:V_\alpha \ra V_{\alpha'}$ between $V_\alpha \subset C^{i,j}(\Gamma_M)$ and $V_\alpha' \subset C^{i+1,j}(\Gamma_M)$.  As in the case of  Khovanov homology (cf. \cite{Kho}), this map is determined by the change in the number of circles between $\Gamma_\alpha$ and $\Gamma_{\alpha'}$:  multiplication, $m$, if two circles in $\Gamma_\alpha$ are merged into one and comultiplication, $\Delta$, if one circle splits into two.  Unlike Khovanov homology, there is a third map, $\eta$, where the number of circles is unchanged (cf. \cite{BaldKauffMc}).  The differential can then be succinctly written using the local maps:
\begin{eqnarray} m(x^i \ot x^j) &=&  \left\{ \begin{array}{ll} \ \ x^{i+j}  \ \ \ \ & \text{if } i+j < n \\[.3cm]
0 & \text{otherwise,}
\end{array}\right.\label{eq:wide-hat-differential-m}\\
\Delta(x^k) &=&  \sum_{\substack{0 \leq i,j < n \\ i+j = k + 2m}} x^i \ot x^j, \mbox{\ and} \nonumber \label{eq:wide-hat-differential-delta}\\
\eta(x^k)&=&  \left\{ \begin{array}{ll} \ \ \sqrt n x^{k+m} \ \ \ \ & \text{if } k+m < n \\[.3cm]
0 & \text{otherwise,}
\end{array}\right.\nonumber \label{eq:wide-hat-differential-eta}
\end{eqnarray}
In the formula, $m$ is the integer defined above depending on the parity of $n$. 

One may check that the maps as described make diagrams corresponding to faces of the hypercube commute, hence signs may be chosen for each map to make them anticommute (see Section 4.2 of \cite{ColorHomology}).  The map $\del_{\alpha \alpha'}:V_\alpha\rightarrow V_{\alpha'}$ can now be defined as the identity on the vector spaces associated with circles that do not change and, up to sign, either $m, \Delta$ or $\eta$ on the vector space(s) associated with circles that are affected by the change from a $0$-smoothing in $\Gamma_\alpha$ to a $1$-smoothing $\Gamma_{\alpha'}$.  Thus, we obtain a chain complex.\\

\begin{theorem}[cf. Section 4.2 of \cite{ColorHomology}] The sequence $(C^{i,*}(\Gamma), \del^i)$ is a differential chain complex, that is, $\del^{i+1}\circ \del^i = 0$ with bigrading $(1,0)$. \label{thm:A-Chain-Complex}
\end{theorem}

The $n$-color homology is the homology of this chain complex:

\begin{definition}
Let $G(V,E)$ be a trivalent graph, $M\subset E$ a perfect matching, and $\Gamma_M$ be a perfect matching graph for $(G,M)$.  The {\em $n$-color homology of $\Gamma_M$}  is
$$CH_n^{i,j}(\Gamma_M;\mathbbm{k}) = \frac{\ker \del^i:C^{i,j}(\Gamma_M) \ra C^{i+1,j}(\Gamma_M)}{\Ima \del^{i-1}:C^{i-1,j}(\Gamma_M) \ra C^{i,j}(\Gamma_M)}.$$
\label{def:2ColorHomology}
\end{definition}

As proven in \cite{ColorHomology}, this homology is invariant under the ribbon moves on perfect matching diagrams, and hence is an invariant of the perfect matching graph itself. Thus, we obtain the following.

\begin{theorem}[cf. Theorem B in \cite{ColorHomology}]
The $n$-color homology is a perfect matching graph invariant.
\end{theorem}

In this paper, we will apply the complex of the $n$-color bigraded theory to the bubbled blowup $\Gamma^B$ of $\Gamma$ to define the local differentials in the hypercube of vertex states.

\section{Vertex Homology}
\label{sec:vertexHomology}

The vertex homology of a trivalent ribbon graph follows a similar procedure as the $n$-color homology but is defined using a state system based on the vertices of a graph instead of perfect matching edges.  

\subsection{The hypercube of vertex states}\label{subsec:cube-of-vertex-res}
Let $\Gamma$ be a ribbon diagram for a connected trivalent graph $G(V,E)$.  Order the vertices:  $V=\{v_1,\ldots, v_{|V|}\}$.   Let $\Gamma_\nu$ be a state indexed by $\nu \in \{0,1\}^{|V|}$,  where each $\nu_i$ in $\nu = (\nu_1,\ldots, \nu_{|V|})$ represents doing a vertex 0-smoothing \vertexbracketzero \ if $\nu_i=0$ or a vertex 1-smoothing \vertexbracketone \ if $\nu_i=1$ in $\Gamma_\bullet$ (cf. \Cref{def:vertex-graph-diagram}).  As in \Cref{section:smoothings-states-hypercubes}, we conceptualize the set of states together with directed edges (defined next) as the {\em hypercube of vertex states of $\Gamma_\bullet$} (cf. \Cref{fig:vertex-state-of-theta}).  

For each pair of states, $\Gamma_\nu$ and $\Gamma_{\nu'}$, that differ in a single smoothing (that is, $\nu_i = \nu_i'$ for all $i$ except one vertex $v_k \in V$, where $\nu_k = 0$ and $\nu_k' = 1$), form a directed edge $\rho_{\nu\nu'}$ from $\Gamma_\nu$ to $\Gamma_{\nu'}$.  Label this edge by a tuple of $0$'s and $1$'s given by the $\nu_i$'s and $\nu_i'$'s where they are the same and a ``$\star$'' for the $k$th position where $\nu_k=\nu_k'-1$.  The edge is directed by requiring the tail to be  $*=0$ and the head  $*=1$. (See \Cref{fig:P3L} for examples of this type of labeling.) 

Observe that every state of the hypercube of vertex states of $\Gamma_\bullet$ appears as a state in the hypercube of states of the bubbled blowup $\Gamma^B$.  Moreover, since each state $\Gamma_\nu$ contains only vertex smoothings, which introduce crossings three at a time at a vertex, $\Gamma_\nu$ corresponds to a state $\Gamma_\alpha$ of $\Gamma^B$ such that $|\alpha| = 3|\nu|$, where $|\nu |$ is the sum of the $\nu_i$'s.

\subsection{The differential chain complex for $\Gamma$}
\label{subsec:VertexDiff}

Using the same quantum gradings as in \Cref{sec:nColor},  associate a graded vector space to each vertex state of $\Gamma_\bullet$.  For $\nu \in \{0,1\}^{|V|}$, let $k_\nu$ be the number of circles in the state $\Gamma_\nu$.  Then, for each $\nu \in \{0,1\}^{|V|}$, associate the graded vector space
\begin{equation}V_\nu = V^{\otimes k_\nu} \{3m |\nu |\}\label{eq:grading-shift-on-V}\end{equation}
where $V=\mathbbm{k}[x]/(x^n)$. Define the complex $C_n^{*,*}(\Gamma)$ by 
\begin{equation} C_n^{i,*}(\Gamma) = \bigoplus_{\substack{\nu\in\{0,1\}^{|V|} \\ i=|\nu|}}V_\nu. \label{eq:chain-complex-vertex} \end{equation}
As before, the internal grading ($q$-grading) is defined by the grading of the elements in $V_\nu$ just as it is in the $n$-color homology.  As for the homological grading, we have $i = |\nu|$.  The complex is trivial outside of $i = 0,\ldots , |V|$.  

Note that $V_\nu$ has the same quantum grading as $V_\alpha$ where $|\alpha| = 3|\nu|$, which means $V_\nu$ can be thought of as $V_\alpha$ in the chain complex of $\Gamma^B$ (see \Cref{fig:bubbledBlowUpHyper}).   This correspondence then allows one to think of each state, $\Gamma_\nu$, as well as the graded vector space associated to it, either as a state in the hypercube of vertex states of $\Gamma_\bullet$ or as a state $\Gamma_\alpha$ in the hypercube of states for the bubbled blowup $\Gamma^B$. This latter perspective is particularly useful for defining the differential for the $n$-color vertex homology (see \Cref{rem:definingMaps}).

The map for each edge in the hypercube of vertex states can now be defined.  Suppose $\Gamma_\nu$ and $\Gamma_{\nu'}$ are joined by an edge.  Consider $\Gamma_\nu$ and $\Gamma_{\nu'}$ as states in the hypercube of states for $\Gamma^B$ and observe that they are joined by at least one path of three edges.  Since each edge of the hypercube of states for $\Gamma^B$ has an associated linear map (either an $m$, $\Delta$ or $\eta$ as defined in \Cref{subsection:differential}) define 
$$\delta_{\nu\nu'} : V_\nu \rightarrow V_{\nu'}$$
by the composition of the relevant maps joining $\Gamma_\nu$ to $\Gamma_{\nu'}$ in the hypercube of states of $\Gamma^B$.  However, this 3-edge path is not unique as there are in fact six such paths (cf. \Cref{fig:bubbledBlowUpHyper}).  Therefore, for this composition to lead to a well-defined map, all six of the relevant compositions must define the same map.  An immediate consequence of the TQFT defined in Section 9 of \cite{ColorHomology} is that they do for all $n\in \NN$.  In that paper a functor was constructed from geometric complexes to graded $\CC$-modules.  This functor allows one to interpret each of the statements about compositions of maps in terms of cobordisms, where properties like commutativity of diagrams are clear, and therefore show that all six paths define the same map.  Nevertheless, it is instructive to consider an example of how the local differentials are well-defined. This is done next.

\begin{figure}[H]
\includegraphics[scale=.27]{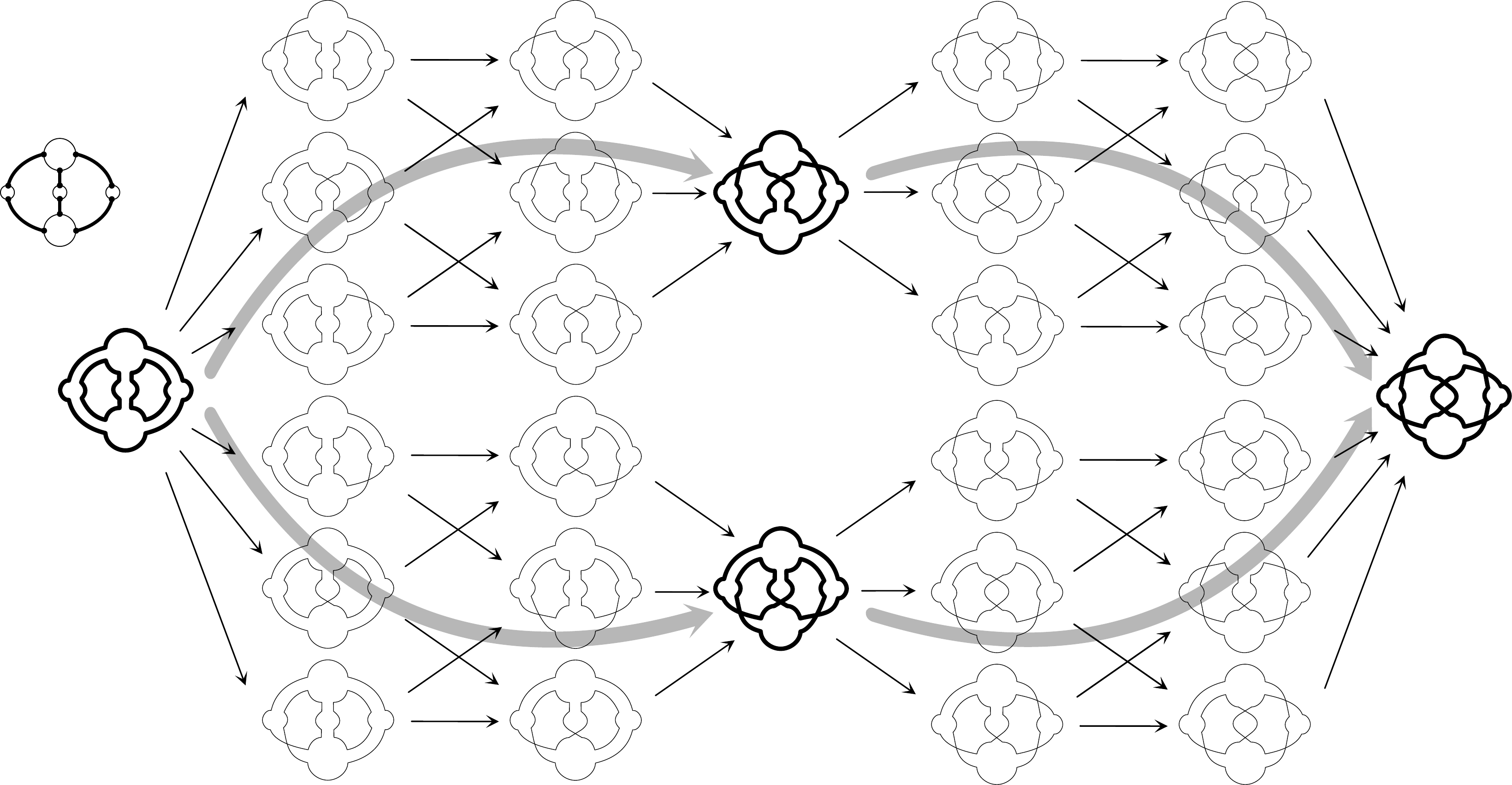}
\caption{This is a portion of the hypercube of states for the bubbled blowup of the $\theta$ graph. It shows the four states used in the construction of vertex homology (cf. \Cref{fig:vertex-state-of-theta}), as well as the intermediate states used to define maps between them.  The four grey arrows show the edges of the hypercube of vertex states.}\label{fig:bubbledBlowUpHyper}
\end{figure}

For an edge joining $\Gamma_\nu$ to $\Gamma_{\nu'}$ in the hypercube of states of $\Gamma_\bullet$, the required maps depend only on the local picture at a vertex smoothing, and how the arcs interact with it.  Therefore, up to symmetry it is sufficient to analyze the seven configurations in \Cref{fig:VertexCases} and show that the composition of three maps described above is well-defined for each. We analyze Configuration 3 for $n=3$.  The remaining cases are pictured in \Cref{app:differentials} and may be handled similarly. 

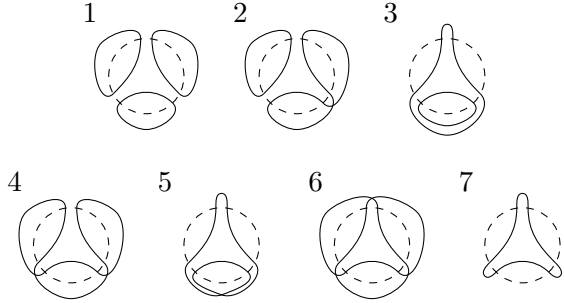
\begin{figure}[h]
$$\begin{tikzpicture}[scale = 0.5]

\begin{scope}[xshift = 2 cm]
\node at (-1.5,1.5) {1};
\draw[dashed] (0,-0.18) circle (1 cm);
\draw (0.15,0.866-0.05) to [out = -90, in = 135] (0.8 + 0.08,-0.8 + 0.15);
\draw (-0.15,0.866-0.05) to [out = -90, in = 45] (-0.8 - 0.08,-0.8 + 0.15);
\draw (-0.8 + 0.095,-0.8 - 0.1) to [out = 45, in = 135] (0.8 - 0.095,-0.8 - 0.1);
\draw (0.15,0.866-0.05) to [out = 90, in = 135] (1.2,0.7) to [out = -45, in = -45] (0.8 + 0.08,-0.8 + 0.15);
\draw (-0.15,0.866-0.05) to [out = 90, in = 45] (-1.20, 0.7) to [out = 225, in = 225] (-0.8 - 0.08,-0.8 + 0.15);
\draw (0.8 - 0.095,-0.8 - 0.1) to [out = -45, in = 0] (0, -1.6) to [out = 180, in = 225]  (-0.8 + 0.095,-0.8 - 0.1);
\end{scope}

\begin{scope}[xshift = 6 cm]
\node at (-1.5,1.5) {2};
\draw[dashed] (0,-0.18) circle (1 cm);
\draw (0.15,0.866-0.05) to [out = -90, in = 135] (0.8 + 0.08,-0.8 + 0.15);
\draw (-0.15,0.866-0.05) to [out = -90, in = 45] (-0.8 - 0.08,-0.8 + 0.15);
\draw (-0.8 + 0.095,-0.8 - 0.1) to [out = 45, in = 135] (0.8 - 0.095,-0.8 - 0.1);
\draw (0.15,0.866-0.05) to [out = 90, in = 135] (1.2,0.7) to [out = -45, in = -45] (0.8 - 0.095,-0.8 - 0.1);
\draw (-0.15,0.866-0.05) to [out = 90, in = 45] (-1.20, 0.7) to [out = 225, in = 225] (-0.8 - 0.08,-0.8 + 0.15);
\draw (0.8 + 0.08,-0.8 + 0.15) to [out = -45, in = 0] (0, -1.6) to [out = 180, in = 225]  (-0.8 + 0.095,-0.8 - 0.1);
\end{scope}

\begin{scope}[xshift = 10 cm]
\node at (-1.5,1.5) {3};
\draw[dashed] (0,-0.18) circle (1 cm);
\draw (0.15,0.866-0.05) to [out = -90, in = 135] (0.8 + 0.08,-0.8 + 0.15);
\draw (-0.15,0.866-0.05) to [out = -90, in = 45] (-0.8 - 0.08,-0.8 + 0.15);
\draw (-0.8 + 0.095,-0.8 - 0.1) to [out = 45, in = 135] (0.8 - 0.095,-0.8 - 0.1);
\draw (0.15,0.866-0.05) to [out = 90, in = 0] (0, 1.2) to [out = 180, in = 90] (-0.15,0.866-0.05);
\draw (0.8 + 0.08,-0.8 + 0.15) to [out = -45, in = 0] (0, -1.75) to [out = 180, in = 225] (-0.8 - 0.08,-0.8 + 0.15) ;
\draw (0.8 - 0.095,-0.8 - 0.1) to [out = -45, in = 0] (0, -1.5) to [out = 180, in = 225] (-0.8 + 0.095,-0.8 - 0.1);
\end{scope}

\begin{scope}[yshift = -4.5 cm]

\node at (-1.5,1.5) {4};
\draw[dashed] (0,-0.18) circle (1 cm);
\draw (0.15,0.866-0.05) to [out = -90, in = 135] (0.8 + 0.08,-0.8 + 0.15);
\draw (-0.15,0.866-0.05) to [out = -90, in = 45] (-0.8 - 0.08,-0.8 + 0.15);
\draw (-0.8 + 0.095,-0.8 - 0.1) to [out = 45, in = 135] (0.8 - 0.095,-0.8 - 0.1);
\draw (0.15,0.866-0.05) to [out = 90, in = 135] (1.2,0.7) to [out = -45, in = -45] (0.8 - 0.095,-0.8 - 0.1);
\draw (-0.15,0.866-0.05) to [out = 90, in = 45] (-1.20, 0.7) to [out = 225, in = 225]  (-0.8 + 0.095,-0.8 - 0.1);
\draw (0.8 + 0.08,-0.8 + 0.15) to [out = -45, in = 0] (0, -1.6) to [out = 180, in = 225] (-0.8 - 0.08,-0.8 + 0.15);

\begin{scope}[xshift = 4 cm]

\node at (-1.5,1.5) {5};
\draw[dashed] (0,-0.18) circle (1 cm);
\draw (0.15,0.866-0.05) to [out = -90, in = 135] (0.8 + 0.08,-0.8 + 0.15);
\draw (-0.15,0.866-0.05) to [out = -90, in = 45] (-0.8 - 0.08,-0.8 + 0.15);
\draw (-0.8 + 0.095,-0.8 - 0.1) to [out = 45, in = 135] (0.8 - 0.095,-0.8 - 0.1);
\draw (0.15,0.866-0.05) to [out = 90, in = 0] (0, 1.2) to [out = 180, in = 90] (-0.15,0.866-0.05);
\draw (0.8 + 0.08,-0.8 + 0.15) to [out = -45, in = -20] (0, -1.5) to [out = 160, in = 225] (-0.8 + 0.095,-0.8 - 0.1);
\draw (0.8 - 0.095,-0.8 - 0.1) to [out = -45, in = 20] (0, -1.5) to [out = 200, in = 225]  (-0.8 - 0.08,-0.8 + 0.15) ;
\end{scope}

\begin{scope}[xshift = 8 cm]
\node at (-1.5,1.5) {6};
\draw[dashed] (0,-0.18) circle (1 cm);
\draw (0.15,0.866-0.05) to [out = -90, in = 135] (0.8 + 0.08,-0.8 + 0.15);
\draw (-0.15,0.866-0.05) to [out = -90, in = 45] (-0.8 - 0.08,-0.8 + 0.15);
\draw (-0.8 + 0.095,-0.8 - 0.1) to [out = 45, in = 135] (0.8 - 0.095,-0.8 - 0.1);
\draw (-0.15,0.866-0.05) to [out = 90, in = 135] (1.2,0.7) to [out = -45, in = -45] (0.8 - 0.095,-0.8 - 0.1);
\draw (0.15,0.866-0.05) to [out = 90, in = 45] (-1.20, 0.7) to [out = 225, in = 225]  (-0.8 + 0.095,-0.8 - 0.1);
\draw (0.8 + 0.08,-0.8 + 0.15) to [out = -45, in = 0] (0, -1.6) to [out = 180, in = 225] (-0.8 - 0.08,-0.8 + 0.15);
\end{scope}

\begin{scope}[xshift =12 cm]
\node at (-1.5,1.5) {7};
\draw[dashed] (0,-0.18) circle (1 cm);
\draw (0.15,0.866-0.05) to [out = -90, in = 135] (0.8 + 0.08,-0.8 + 0.15);
\draw (-0.15,0.866-0.05) to [out = -90, in = 45] (-0.8 - 0.08,-0.8 + 0.15);
\draw (-0.8 + 0.095,-0.8 - 0.1) to [out = 45, in = 135] (0.8 - 0.095,-0.8 - 0.1);
\draw (0.15,0.866-0.05) to [out = 90, in = 0] (0, 1.2) to [out = 180, in = 90] (-0.15,0.866-0.05);
\draw (0.8 + 0.08,-0.8 + 0.15) to [out = -45, in = 45]  (0.8 + 0.2,-0.8 - 0.2) to [out = 225, in = -45]  (0.8 - 0.095,-0.8 - 0.1);
\draw (-0.8 - 0.08,-0.8 + 0.15) to [out = 225, in = 135] (-0.8 - 0.2,-0.8 - 0.2) to [out = -45, in = 225] (-0.8 + 0.095,-0.8 - 0.1);
\end{scope}

\end{scope}

\end{tikzpicture}$$
\caption{The possible configurations of circles at a $0$-smoothing vertex. Each configuration shows a local neighborhood of  vertex (enclosed by the dotted circle) and the possible ways that the three arcs at a vertex can be joined to form circles (up to symmetry).}
\label{fig:VertexCases}
\end{figure}

There are three compositions (up to symmetry) that need to be analyzed in \Cref{fig:Case3}:  $\Delta \circ m \circ \eta$, $\Delta \circ  \eta \circ m$, and $\eta \circ \Delta \circ m$.  
For $\Delta \circ m \circ \eta$,  follow the upper path of three edges in \Cref{fig:Case3}.  In the initial state, label the outer circle as Circle 1, and the inner circle as Circle 2.  For $\eta \circ \Delta \circ m$, follow the bottom path of three edges in \Cref{fig:Case3} and in the penultimate state, label the upper circle as Circle 1, and the lower circle as Circle 2.  Calculating, each give the same answer, albeit in different ways:
$$\Delta \circ m \circ (\eta \otimes Id)(1\otimes 1) = \Delta \circ m (\sqrt{3} x \otimes 1) = \Delta (\sqrt{3} x) = \sqrt{3}(x^2 \otimes x+x\otimes x^2)$$
$$\Delta \circ  \eta \circ m(1\otimes 1) = \Delta \circ \eta (1) = \Delta (\sqrt{3} x) =  \sqrt{3}(x^2 \otimes x+x\otimes x^2)$$
$$(\eta\otimes Id) \circ \Delta \circ m(1\otimes 1) = (\eta\otimes Id) \circ \Delta(1) = (\eta\otimes Id) (x^2\otimes 1+1 \otimes x^2 +x\otimes x) =  \sqrt{3}(x\otimes x^2+x^2 \otimes x)$$

\noindent In each case, the result does not depend on the 3-edge path chosen.  Moreover, this continues to be true for the other generators and the remaining cases shown in \Cref{app:differentials}.  

Thus, define the differential $\delta^i : C_n^{i,*} (\Gamma) \rightarrow C_n^{i+1,*}(\Gamma)$ to be the sum of appropriate $\delta_{\nu\nu'}$'s.  For $w\in V_\nu \subset C_{n}^{i,*}(\Gamma)$, 
\begin{equation}\delta^i(w) = \sum_{\substack{\rho \ \mbox{\tiny such that }\\ \mbox{\tiny Tail}(\rho) = \nu}} \sign(\rho)\delta_{\nu\nu'}(w).\label{definition:vertex-delta}\end{equation}
where the sign of each edge is equal to $-1$ to the number of $1$'s to the left of $\ast$. 

\begin{figure}[h]
$$\begin{tikzpicture}[scale = 0.5]

\begin{scope}
\draw[dashed] (0,-0.18) circle (1 cm);
\draw (0.15,0.866-0.05) to [out = -90, in = 135] (0.8 + 0.08,-0.8 + 0.15);
\draw (-0.15,0.866-0.05) to [out = -90, in = 45] (-0.8 - 0.08,-0.8 + 0.15);
\draw (-0.8 + 0.095,-0.8 - 0.1) to [out = 45, in = 135] (0.8 - 0.095,-0.8 - 0.1);

\draw (0.15,0.866-0.05) to [out = 90, in = 0] (0, 1.2) to [out = 180, in = 90] (-0.15,0.866-0.05);
\draw (0.8 + 0.08,-0.8 + 0.15) to [out = -45, in = 0] (0, -1.75) to [out = 180, in = 225] (-0.8 - 0.08,-0.8 + 0.15) ;
\draw (0.8 - 0.095,-0.8 - 0.1) to [out = -45, in = 0] (0, -1.5) to [out = 180, in = 225] (-0.8 + 0.095,-0.8 - 0.1);
\end{scope}

\begin{scope}[xshift = 1 cm, yshift = -0.2 cm]
\node at (1,1.35) {\tiny $\eta$};
\node at (1.25,0.25) {\tiny $m$};
\node at (1.25,-0.75) {\tiny $m$};
\draw[->] (0.5,0.5) --(2,2);
\draw[->] (0.5,0) --(2,0);
\draw[->] (0.5,-0.5) --(2,-2);
\end{scope}

\begin{scope}[xshift = 5.5 cm, yshift = -0.2 cm]
\node at (1.25,3.25) {\tiny $m$};
\node at (1.25,2.25) {\tiny $m$};
\node at (.75,1.25) {\tiny $\eta$};

\node at (1.25,-.75) {\tiny $\Delta$};
\node at (.75,-1.75) {\tiny $\eta$};
\node at (1.25,-2.75) {\tiny $\Delta$};

\draw[->] (0.5,0.5) --(2,2);
\draw[->] (0.5,-0.5) --(2,-2);
\draw[->] (0.5,3) --(2,3);
\draw[->] (0.5,-3) --(2,-3);
\draw[->] (0.5,-2.5) --(2,-1);
\draw[->] (0.5,2.5) --(2,1);
\end{scope}

\begin{scope}[xshift = 8.5 cm, yshift = -0.2 cm]
\node at (3,1.5) {\tiny $\Delta$};
\node at (2.5,0.25) {\tiny $\Delta$};
\node at (2.5,-1) {\tiny $\eta$};
\draw[->] (2,2) -- (3.5,0.5);
\draw[->] (2,0) -- (3.5,0);
\draw[->] (2,-2) -- (3.5,-0.5);
\end{scope}

\begin{scope}[xshift = 4.5 cm, yshift = 3 cm]
\draw[dashed] (0,-0.18) circle (1 cm);
\draw (-0.15,0.866-0.05) to [out = -90, in = 135] (0.8 + 0.08,-0.8 + 0.15);
\draw (0.15,0.866-0.05) to [out = -90, in = 45] (-0.8 - 0.08,-0.8 + 0.15);
\draw (-0.8 + 0.095,-0.8 - 0.1) to [out = 45, in = 135] (0.8 - 0.095,-0.8 - 0.1);

\draw (0.15,0.866-0.05) to [out = 90, in = 0] (0, 1.2) to [out = 180, in = 90] (-0.15,0.866-0.05);
\draw (0.8 + 0.08,-0.8 + 0.15) to [out = -45, in = 0] (0, -1.75) to [out = 180, in = 225] (-0.8 - 0.08,-0.8 + 0.15) ;
\draw (0.8 - 0.095,-0.8 - 0.1) to [out = -45, in = 0] (0, -1.5) to [out = 180, in = 225] (-0.8 + 0.095,-0.8 - 0.1);
\end{scope}

\begin{scope}[xshift = 4.5 cm, yshift = 0 cm]
\draw[dashed] (0,-0.18) circle (1 cm);
\draw (0.15,0.866-0.05) to [out = -90, in = 135] (0.8 + 0.08,-0.8 + 0.15);
\draw (-0.15,0.866-0.05) to [out = -90, in = 45] (-0.8 + 0.095,-0.8 - 0.1);
\draw (-0.8 - 0.08,-0.8 + 0.15) to [out = 45, in = 135] (0.8 - 0.095,-0.8 - 0.1);

\draw (0.15,0.866-0.05) to [out = 90, in = 0] (0, 1.2) to [out = 180, in = 90] (-0.15,0.866-0.05);
\draw (0.8 + 0.08,-0.8 + 0.15) to [out = -45, in = 0] (0, -1.75) to [out = 180, in = 225] (-0.8 - 0.08,-0.8 + 0.15) ;
\draw (0.8 - 0.095,-0.8 - 0.1) to [out = -45, in = 0] (0, -1.5) to [out = 180, in = 225] (-0.8 + 0.095,-0.8 - 0.1);
\end{scope}

\begin{scope}[xshift = 4.5 cm, yshift = -3 cm]
\draw[dashed] (0,-0.18) circle (1 cm);
\draw (0.15,0.866-0.05) to [out = -90, in = 135] (0.8 - 0.095,-0.8 - 0.1);
\draw (-0.15,0.866-0.05) to [out = -90, in = 45] (-0.8 - 0.08,-0.8 + 0.15);
\draw (-0.8 + 0.095,-0.8 - 0.1) to [out = 45, in = 135] (0.8 + 0.08,-0.8 + 0.15);

\draw (0.15,0.866-0.05) to [out = 90, in = 0] (0, 1.2) to [out = 180, in = 90] (-0.15,0.866-0.05);
\draw (0.8 + 0.08,-0.8 + 0.15) to [out = -45, in = 0] (0, -1.75) to [out = 180, in = 225] (-0.8 - 0.08,-0.8 + 0.15) ;
\draw (0.8 - 0.095,-0.8 - 0.1) to [out = -45, in = 0] (0, -1.5) to [out = 180, in = 225] (-0.8 + 0.095,-0.8 - 0.1);
\end{scope}

\begin{scope}[xshift = 4.5 cm]
\begin{scope}[xshift = 4.5 cm, yshift = 3 cm]
\draw[dashed] (0,-0.18) circle (1 cm);
\draw (-0.15,0.866-0.05) to [out = -90, in = 135] (0.8 + 0.08,-0.8 + 0.15);
\draw (0.15,0.866-0.05) to [out = -90,  in = 60] (-0.2, 0) to [out = 240, in = 45] (-0.8 + 0.095,-0.8 - 0.1);
\draw (-0.8 - 0.08,-0.8 + 0.15) to [out = 45, in = 135] (0.8 - 0.095,-0.8 - 0.1);

\draw (0.15,0.866-0.05) to [out = 90, in = 0] (0, 1.2) to [out = 180, in = 90] (-0.15,0.866-0.05);
\draw (0.8 + 0.08,-0.8 + 0.15) to [out = -45, in = 0] (0, -1.75) to [out = 180, in = 225] (-0.8 - 0.08,-0.8 + 0.15) ;
\draw (0.8 - 0.095,-0.8 - 0.1) to [out = -45, in = 0] (0, -1.5) to [out = 180, in = 225] (-0.8 + 0.095,-0.8 - 0.1);
\end{scope}

\begin{scope}[xshift = 4.5 cm, yshift = 0 cm]
\draw[dashed] (0,-0.18) circle (1 cm);
\draw (-0.15,0.866-0.05) to [out = -90, in = 120] (0.2, 0) to [out = -60, in = 135] (0.8 - 0.095,-0.8 - 0.1);
\draw (0.15,0.866-0.05) to [out = -90, in = 45] (-0.8 - 0.08,-0.8 + 0.15);
\draw (-0.8 + 0.095,-0.8 - 0.1) to [out = 45, in = 135] (0.8 + 0.08,-0.8 + 0.15);

\draw (0.15,0.866-0.05) to [out = 90, in = 0] (0, 1.2) to [out = 180, in = 90] (-0.15,0.866-0.05);
\draw (0.8 + 0.08,-0.8 + 0.15) to [out = -45, in = 0] (0, -1.75) to [out = 180, in = 225] (-0.8 - 0.08,-0.8 + 0.15) ;
\draw (0.8 - 0.095,-0.8 - 0.1) to [out = -45, in = 0] (0, -1.5) to [out = 180, in = 225] (-0.8 + 0.095,-0.8 - 0.1);
\end{scope}

\begin{scope}[xshift = 4.5 cm, yshift = -3 cm]
\draw[dashed] (0,-0.18) circle (1 cm);
\draw (0.15,0.866-0.05) to [out = -90, in = 135] (0.8 - 0.095,-0.8 - 0.1);
\draw (-0.15,0.866-0.05) to [out = -90, in = 45] (-0.8 + 0.095,-0.8 - 0.1);
\draw (-0.8 - 0.08,-0.8 + 0.15) to [out = 45, in = 180] (0, -0.5) to [out = 0, in = 135] (0.8 + 0.08,-0.8 + 0.15);

\draw (0.15,0.866-0.05) to [out = 90, in = 0] (0, 1.2) to [out = 180, in = 90] (-0.15,0.866-0.05);
\draw (0.8 + 0.08,-0.8 + 0.15) to [out = -45, in = 0] (0, -1.75) to [out = 180, in = 225] (-0.8 - 0.08,-0.8 + 0.15) ;
\draw (0.8 - 0.095,-0.8 - 0.1) to [out = -45, in = 0] (0, -1.5) to [out = 180, in = 225] (-0.8 + 0.095,-0.8 - 0.1);
\end{scope}

\end{scope}

\begin{scope}[xshift = 13.5 cm]
\begin{scope}
\draw[dashed] (0,-0.18) circle (1 cm);
\draw (-0.15,0.866-0.05) to [out = -90, in = 120] (0.2, 0) to [out = -60, in = 135] (0.8 - 0.095,-0.8 - 0.1);
\draw (0.15,0.866-0.05) to [out = -90,  in = 60] (-0.2, 0) to [out = 240, in = 45] (-0.8 + 0.095,-0.8 - 0.1);
\draw (-0.8 - 0.08,-0.8 + 0.15) to [out = 45, in = 180] (0, -0.5) to [out = 0, in = 135] (0.8 + 0.08,-0.8 + 0.15);

\draw (0.15,0.866-0.05) to [out = 90, in = 0] (0, 1.2) to [out = 180, in = 90] (-0.15,0.866-0.05);
\draw (0.8 + 0.08,-0.8 + 0.15) to [out = -45, in = 0] (0, -1.75) to [out = 180, in = 225] (-0.8 - 0.08,-0.8 + 0.15) ;
\draw (0.8 - 0.095,-0.8 - 0.1) to [out = -45, in = 0] (0, -1.5) to [out = 180, in = 225] (-0.8 + 0.095,-0.8 - 0.1);
\end{scope}

\end{scope}

\end{tikzpicture}$$
\caption{Configuration 3:  the six $3$-edge paths from the two circles of $\Gamma_\nu$ to the two circles of $\Gamma_{\nu'}$.}
\label{fig:Case3}
\end{figure}
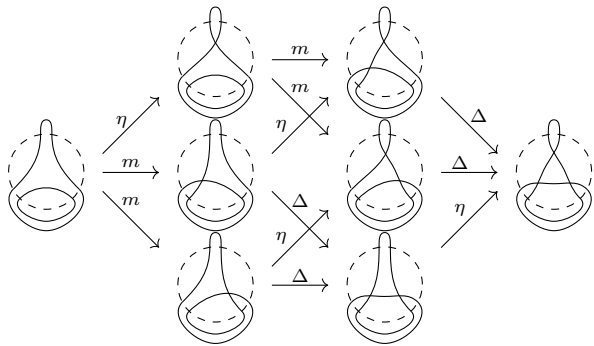

We now have enough to state the following theorem and define the vertex homology.

\begin{theorem}
$(C_n^{i,*}(\Gamma),\delta^i)$ is a chain complex with differential that increases the homological grading by one and preserves the quantum grading, i.e., it has bigrading $(1,0)$.
\label{thm:2ColorVertexHomology}
\end{theorem}

\begin{proof}
All that remains to show is that $\delta^{i+1}\circ \delta^i=0$.  This follows from the fact that all diagrams of faces of the hypercube of vertex states for $\Gamma_\bullet$ commute because, in the bubbled blowup $\Gamma^B$, 
\begin{enumerate}
\item each map $\delta_{\nu\nu'}$ is independent of the 3-edge path chosen in $\Gamma^B$, and more generally,
\item any diagram given by a face in the hypercube of states of $\Gamma^B$ commutes, which means that faces in $\Gamma_\bullet$ that correspond to diagrams made from four 3-edge paths in the hypercube of $\Gamma^B$ must also commute.
\end{enumerate}
Thus, with the chosen sign convention for the definition of $\delta^i$ in \Cref{definition:vertex-delta}, each diagram corresponding to a face in the hypercube of vertex states anticommute.  
\end{proof}

\begin{definition}
Let $\Gamma$ be a ribbon diagram of a trivalent graph of $G(V,E)$.  The {\em bigraded $n$-color vertex homology of $\Gamma$}  is
$$VCH_n^{i,j}(\Gamma;\mathbbm{k}) = \frac{\ker \delta^i:C_n^{i,j}(\Gamma) \ra C_n^{i+1,j}(\Gamma)}{\Ima \delta^{i-1}:C_n^{i-1,j}(\Gamma) \ra C_n^{i,j}(\Gamma)}.$$
\label{def:2ColorHomology}
\end{definition}

There now is enough background to prove the first theorem.

\begin{proof}[Proof of Theorem~\ref{thm:gradedEuler}]
The homology is invariant under the ribbon moves described in \cite{BKR}, hence is independent of the ribbon diagram used to define it. Also, the homology is  independent of how the vertices were indexed ($V=\{v_1,\ldots, v_{|V|}\}$ and $V=\{v'_1,\ldots,v'_{|V|}\}$ produce the same homology). Thus, the homology is a ribbon graph invariant.  

For $V=\mathbbm{k}[x]/ (x^n)$, the quantum dimension of $V$,  $q\dim V$, is equal to the expression given in Equation~\eqref{eq:immersed_circle} based on whether $n$ is even or odd. Since the chain groups, $C^{i,*}_n(\Gamma)$, are direct sums of shifted versions of $V$ (cf. \Cref{eq:chain-complex-vertex}), the graded Euler characteristic of the vertex $n$-color homology is equal to the $n$-color vertex polynomial. 
\end{proof}

It should be noted that $VCH^{i,j}_n(\Gamma;\mathbbm{k}) \not\cong CH^{3i,j}_n(\Gamma^B;\mathbbm{k})$ for all positive integers $n$. They are different theories, even though the same complex is used to define them. The difference is, by analogy, as if someone defined a new ``exterior differential''  ${d_3:\Omega^p(M;\BR) \ra \Omega^{p+3}(M;\BR)}$ on $p$-forms of a manifold $M$ such that $d_3\circ d_3 = 0$. The strange  differential with its jump in homological grading,  $\delta:C^{i,j}_n(\Gamma^B) \ra C^{i+3,j}_n(\Gamma^B)$, makes one wonder if there are exotic differentials on ``blowups'' of knots in Khovanov homology that lead to new knot invariants.

\begin{remark} \label{rem:generalize} The bigraded $n$-color vertex homology can be generalized to $r$-regular graphs with $r>3$, not just trivalent ones.  Beginning with any $r$-regular ribbon graph, construct the blowup to obtain a trivalent graph with canonical perfect matching $E$, and the bubbled blowup construction works as well.  The grading shift in \Cref{eq:grading-shift-on-V} of $3m|\nu|$ that accompanies each state $\Gamma_\nu$ of the hypercube of states of $\Gamma_\bullet$ is changed to $rm|\nu|$. (Similarly, there is a  shift from $q^{3m}$ in Equation~\eqref{eq:vertex-bracket} to $q^{rm}$ for the equivalent polynomial.)  The commutativity of the diagrams  for faces in the hypercube of states for the bubbled blowup will ensure a well-defined differential, and ultimately a homology theory.  

Our choice to restrict to trivalent graphs for this paper is due to the relative importance of trivalent graphs (blowup of any graph is trivalent), together with the fact that this restriction makes an analysis of the relevant maps (cf. \Cref{fig:VertexCases}) tractable.  In the general case of an $r$-regular graph, the configurations that would need to be studied grows with $r$.  Nevertheless, the TQFT defined in Section 9 of \cite{ColorHomology} works for regular graphs of any valence and thus, implies that such constructions will lead to well-defined homology theories (see \Cref{sec:4valent} for $r=4$).  In fact, one could even eliminate regularity.  The commentary above on the differentials still applies, however it is not clear how to define a quantum grading that leads to a bigraded theory. But a singularly graded homology still exists.
\end{remark}

\subsection{Examples}
In this section, we present several calculations of the bigraded $2$-color homology, as well as the $2$-color vertex polynomial.  Mathematica code for computation of the $n$-color vertex polynomials for $n=2,3,4$ is provided in \Cref{app:computations}.

\begin{example}
Consider a ribbon diagram, $\Gamma_\theta$, of the $\theta$ graph and its associated hypercube of vertex states, as depicted in \Cref{fig:bubbledBlowUpHyper}.  The hypercube of vertex states consists of the all-zero state with three circles, two states with one circle, and the all-one state with three circles.  The only nontrivial map takes $1\otimes 1\otimes 1 \in C_2^{0,3}(\Gamma_\theta)$ to $(x,x) \in C_2^{1,3}(\Gamma_\theta)$.  Thus we obtain the homology shown in \Cref{fig:ThetaHomology}.

\begin{center}
\begin{figure}[H]
\begin{tabular}{ | c | c | c | c | }
 \hline
 9 &   &  & $\langle 1\otimes 1\otimes 1\rangle$ \\ 
 \hline
8  &   &  & $\langle 1\otimes 1\otimes x, 1\otimes x\otimes 1,x\otimes 1\otimes 1\rangle$ \\
 \hline
7  &   &  & $\langle 1\otimes x\otimes x, x\otimes 1\otimes x,x\otimes x\otimes 1\rangle$ \\
 \hline
 6 &   &  & $\langle x\otimes x\otimes x\rangle$ \\   
 \hline
 5 &   &  &  \\  
 \hline
4  &   &  $\langle (1,0),(0,1) \rangle$& \\
 \hline
 3  &   &  $\langle (x,0) \rangle$& \\
 \hline
2  &   $\langle 1\otimes 1\otimes x, 1\otimes x\otimes 1,x\otimes 1\otimes 1\rangle$ &  & \\
 \hline
1  &   $\langle 1\otimes x\otimes x, x\otimes 1\otimes x,x\otimes x\otimes 1\rangle$ &  &\\
 \hline
0  &   $\langle x\otimes x\otimes x\rangle$ &  &\\  
 \hline
\diagbox[dir=SW]{$j$}{$i$} & 0 & 1 & 2 \\
  \hline
 \end{tabular}
\caption{The vertex homology of the $\theta$ graph, $VCH_2^{i,j}(\Gamma_\theta)$.}
\label{fig:ThetaHomology}
 \end{figure}
\end{center}

A few observations are in order here.  The first is that the graded Euler characteristic of this homology is the $n$-color vertex polynomial (cf. \Cref{eq:vertex-poly-of-theta}).   Second, almost all elements of the chain groups survive in this homology.  This is true in general: most of the maps used to define the differential are zero (though larger graphs, and a choice of $n>2$ will tend to produce more nonzero differentials than this example).  
\end{example}

\begin{example}
\label{ex:3lolly}
For the 3-lollipop ribbon graph, $\Gamma_{L3}$, shown in \Cref{fig:3Lolly}, observe that the hypercube of vertex states contains only maps taking one circle to one circle, or two circles to two circles (cf. \Cref{fig:Case3App,fig:Case7}).  In both cases, the differentials are identically zero, and hence $VCH_2^{i,j}(\Gamma_{L3};\mathbbm{k}) \cong C_n^{i,j}(\Gamma_{L3})$ for all $i,j$.  In this case, the homology contains no new information.  However, later it will be shown that  the filtered $n$-color vertex homology of this graph is identically zero, which indicates that it has no perfect matchings (cf. \Cref{cor:TFAE-perfect-matching}).  

\begin{figure}[h]
\includegraphics[scale=.45]{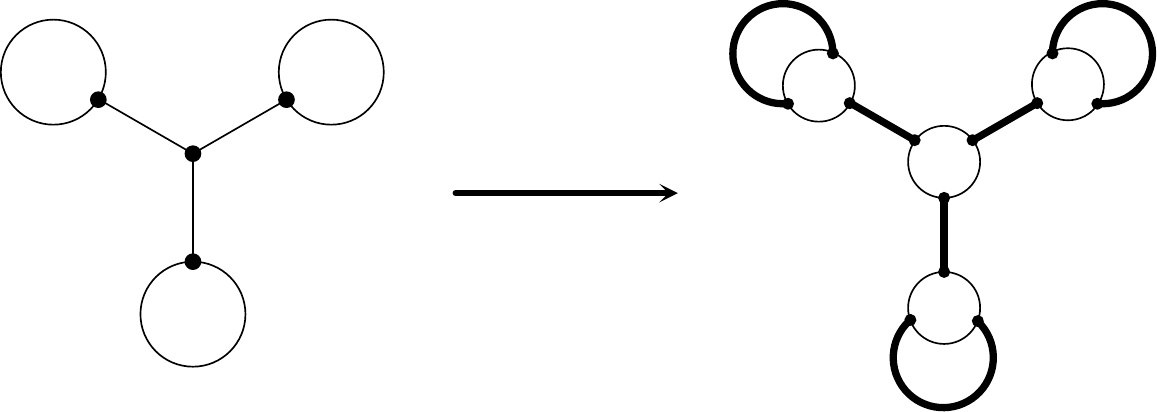}
\caption{The 3-lollipop ribbon graph $\Gamma_{L3}$ has no perfect matchings, but the blowup still has a canonical perfect matching consisting of the edges of $\Gamma_{L3}$.}\label{fig:3Lolly}
\end{figure}
\end{example}

\begin{example}\label{ex:P3Calc}
Let $\Gamma_{P3}$ be a ribbon graph for the $3$-prism in \Cref{fig:P3BlownUp}.  Observe that the 2-color vertex polynomial is given by
\begin{multline*}
\llangle \Gamma_{P3} \rrangle_2 = 1 + 5 q + 10 q^2 + 4 q^3 - 13 q^4 - 17 q^5 + 9 q^6 + 33 q^7 + 
 27 q^8 - 11 q^9 - 36 q^{10} - 24 q^{11} + 7 q^{12} + 33 q^{13} + 27 q^{14} + \\
 3 q^{15} - 18 q^{16} - 18 q^{17} - 5 q^{18} + 5 q^{19} + 10 q^{20} + 10 q^{21} + 
 5 q^{22} + q^{23},
 \end{multline*}
which can be calculated using the Mathematica code provided in \Cref{app:computations}. See the introduction of the appendix for code for $P3$ or \Cref{ex:P3}.
 
Like the difference between the Khovanov homology and the Jones polynomial, the $n$-color homology is a stronger invariant than the $n$-color vertex polynomial.  To see this, focus on the $9q^6$ term in the polynomial above; the quantum grading in grading six of the $2$-color vertex homology contains more information than the coefficient of this term.  In homological grading one, we have six states, each consisting of three circles.  In particular, observe that $\delta: C_2^{1,j} (\Gamma_{P3})\rightarrow C_2^{2,j}(\Gamma_{P3})$ is nonzero only when $j=6$.  Moreover, the nonzero maps occur where $1\otimes1\otimes1$ is mapped to $\sqrt{2} x$ on a single-cycle state in homological grading two.  \Cref{fig:P3BlownUp} shows the relevant states where these nonzero maps occur. Observe that these maps occur in pairs. Consequently,  there exists a linear combination of terms of the form $1\otimes1\otimes1$ that maps to 0.  Notice that none of these terms can be in the image of the previous differential.  Thus, $VCH_2^{1,6}(\Gamma_{P3};\mathbbm{k}) \cong \mathbbm{k}$.  

In homological grading two,  all fifteen generators of $C_2^{2,6}(\Gamma_{P3})$ are in the kernel of the differential, but the image of the previous differential has rank five.  Hence, $VCH_2^{2,6}(\Gamma_{P3};\mathbbm{k})\cong \mathbbm{k}^{10}$.  Thus, the graded Euler characteristic for quantum grading $j=6$ is $10-1=9$, which gives the $9q^6$ term. Therefore, the homology has more information in it than the polynomial.

\begin{theorem}
The $n$-color vertex homology is a stronger invariant than the $n$-color vertex polynomial.
\end{theorem}

\begin{figure}[h]
\includegraphics[scale=.35]{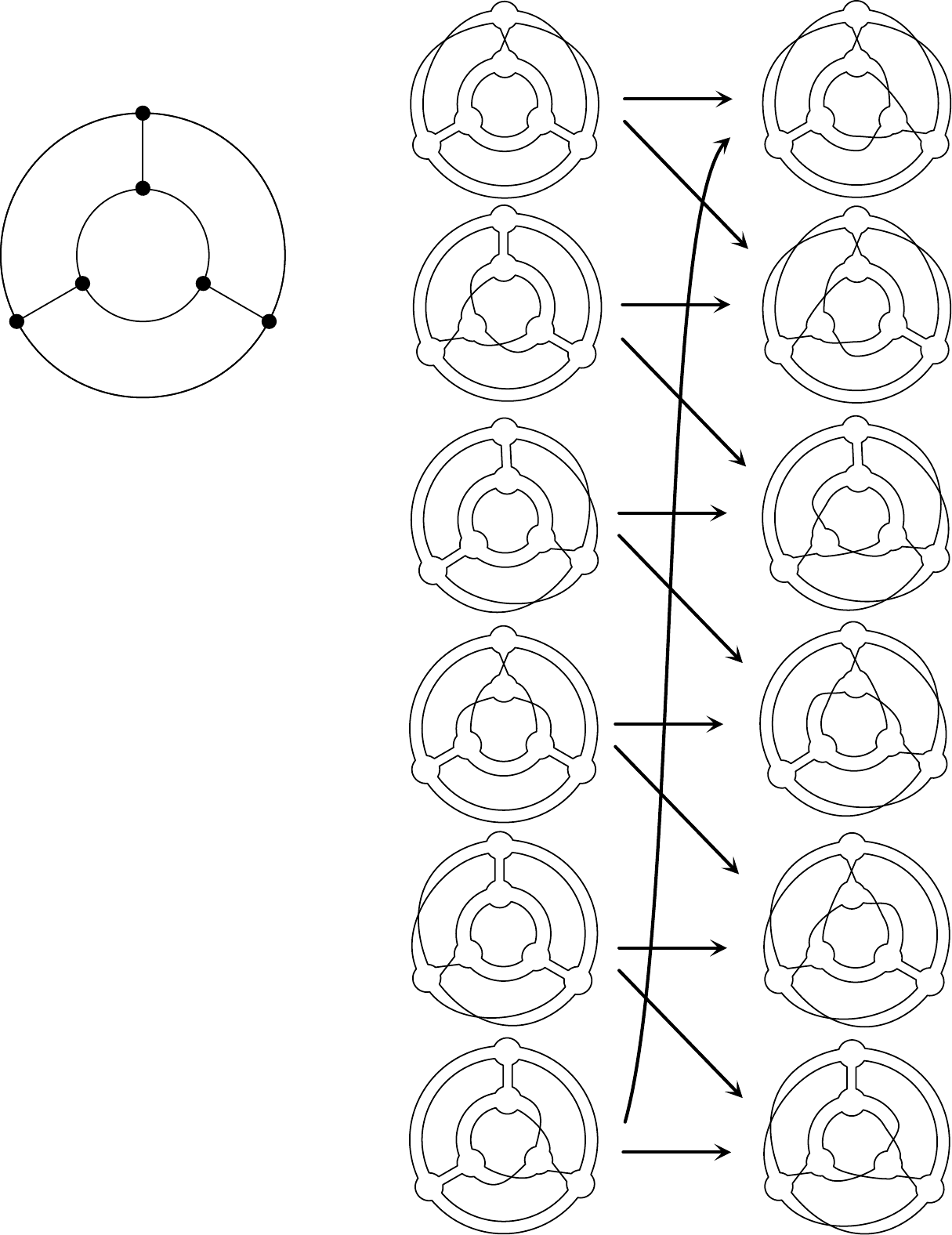}
\caption{A portion of the hypercube of vertex states for the $3$-prism.}\label{fig:P3BlownUp}
\end{figure}

\end{example}

\section{Filtered $n$-color vertex homology}
\label{sec:VertexLee}

In \cite{ColorHomology}, a differential $\widetilde{\del}$ was defined on the chain complex $C^{*,*}(\Gamma_M)$ that gave rise to a spectral sequence. The $E_1$-page of this sequence was the bigraded $n$-color homology and the $E_\infty$-page was the filtered $n$-color homology.  Moreover, the total differential in that case, $\widehat{\del} = \del  + \widetilde{\del}$, produced a homology theory that counted the $n$-face colorings of the ribbon graphs of the graph.  The goal of this section is to take advantage of this construction to build a spectral sequence for $n$-color vertex homology theory.  Just as in \Cref{subsec:VertexDiff}, compositions of multiplication, comultiplication, and the single-circle maps along 3-edge paths are used to define a vertex version of $\widehat{\del}$.   The twist in this section is that the maps in the composition are made up of {\em both} $\del_{\alpha\alpha'}$ and $\widetilde{\del}_{\alpha\alpha'}$ maps. Using these compositions leads to a new map, $\widetilde{\delta}$,  based upon three new maps, $\delta_n, \delta_{2n},$ and $\delta_{3n}$, each distinguished by their jump in quantum grading. (The map from  \Cref{subsec:VertexDiff} is $\delta_0 :=\delta$ in this context.)

\subsection{A new differential}\label{sec:fourDiffs}
The ideas behind $\widetilde{\del}: C^{i,*}(\Gamma_M) \rightarrow C^{i+1,*}(\Gamma_M)$ is sketched out here. See Section 5 of \cite{ColorHomology} for complete details. 

On the chain level, $\widetilde{\del}$ is defined in the same way as $\del$.  Let $\rho_{\alpha\alpha'}$ be an edge in the hypercube of states from $\Gamma_\alpha $ to $\Gamma_{\alpha'}$.  This gives rise to a linear map, $\widetilde{\del}_{\alpha\alpha'}:V_\alpha \rightarrow V_{\alpha'}$, which is defined based upon the circles in $\Gamma_\alpha$:

\begin{enumerate}
\item If $\rho_{\alpha\alpha'}$ represents fusing two circles in $\Gamma_\alpha$ into one circle in $\Gamma_{\alpha'}$, define the relevant map of $\widetilde{\del}_{\alpha\alpha'}$ by multiplication,
$$\widetilde{m}(x^i\ot x^j) = x^{i+j-n} \mbox{\ \  if $i+j\geq n$},$$
and zero if $i+j<n$.
\item If $\rho_{\alpha\alpha'}$ represents the splitting of a circle  in $\Gamma_\alpha$ into two circles in $\Gamma_{\alpha'}$,  define   the relevant map of  $\widetilde{\del}_{\alpha\alpha'}$ by comultiplication, $$\widetilde{\Delta}(x^k)= \sum_{\substack{i+j=k+2m-n\\0\leq i,j < n}} x^i \ot x^j \mbox{\ \  if $k+2m \geq n$},$$
and zero if $i+2m<n$.
\item If $\rho_{\alpha\alpha'}$ represents introducing a double point in a circle in $\Gamma_\alpha$ to get a circle in $\Gamma_{\alpha'}$, define  the relevant map of  $\widetilde{\del}_{\alpha\alpha'}$ by, $$\widetilde{\eta}(x^k) = \sqrt{n}\cdot x^{k+m-n} \mbox{\ \  if $k+m\geq n$},$$
and zero if $k+m<n$.
\end{enumerate}
Like $\del_{\alpha\alpha'}$, the map $\widetilde{\del}_{\alpha\alpha'}:V_\alpha\rightarrow V_{\alpha'}$ is defined on basis elements of $V_\alpha$ as a tensor product of maps. This product is given by the identity on the vector spaces associated with circles that do not change from $\Gamma_\alpha$ to $\Gamma_{\alpha'}$, and either $\widetilde{m}$, $\widetilde{\Delta}$, or $\widetilde{\eta}$ on the vector space(s) associated to circles that are modified by the change from a $0$-smoothing in $\Gamma_\alpha$ to a $1$-smoothing $\Gamma_{\alpha'}$. Extend this map linearly.

In particular, for $n=2$, the relevant maps of $\widetilde{\del}_{\alpha\alpha'}$ are defined by:

\hskip1in\begin{minipage}[t]{2in}\begin{eqnarray*}
1 \otimes 1, 1 \otimes x, x \otimes 1 
&\buildrel \widetilde{m} \over \longmapsto& 0 \\
x \otimes x &\buildrel \widetilde{m} \over \longmapsto& 1 \\
1 &\buildrel \widetilde{\eta} \over \longmapsto& 0 \\
x &\buildrel \widetilde{\eta} \over \longmapsto& \sqrt{2} \cdot 1
\end{eqnarray*}
\end{minipage} \hskip.5in \begin{minipage}[t]{2in}\begin{eqnarray*}
1 &\buildrel \widetilde{\Delta} \over \longmapsto& 1 \otimes 1 \\
x &\buildrel \widetilde{\Delta} \over \longmapsto& 1 \otimes x + x \otimes 1\\
\end{eqnarray*}
\end{minipage}

\begin{remark}
These maps are different than the ones used in both Khovanov homology \cite{Kho} and Lee homology \cite{LeeHomo}!
\end{remark}

As in the previous section, each edge $\rho_{\nu\nu'}$ of the hypercube of vertex states $\Gamma_\bullet$ joins two states, $\Gamma_\nu$ and $\Gamma_{\nu'}$, that differ by a single vertex smoothing.  Consider $\Gamma_\nu$ and $\Gamma_{\nu'}$ as states in the hypercube of states of the bubbled blowup $\Gamma^B$.  When thought of in this way, they are joined by a path of three edges, which we denote $\rho_1, \rho_2,$ and $\rho_3$.  Thus,  the $\widehat{\delta} = \delta + \widetilde{\delta}$ map from $\Gamma_\nu$ to $\Gamma_{\nu'}$ in $\Gamma_\bullet$ is defined using a $3$-fold composition of choices of maps $\del_{\alpha\alpha'}$ or $\widetilde{\del}_{\alpha\alpha'}$ for each of the three edges in $\Gamma^B$:  
\begin{enumerate}
\item The linear maps $\del_{\rho_i}$ are defined by $m$, $\Delta$, or $\eta$.
\item The linear maps $\widetilde{\del}_{\rho_i}$ are defined by $\widetilde{m}$, $\widetilde{\Delta}$, or $\widetilde{\eta}$.
\end{enumerate}
Choosing such a map for each of the three edges above gives $2^3$ possible compositions that are needed to define $\widehat{\delta}$ for the map that corresponds to the edge from $\Gamma_\nu$ to $\Gamma_{\nu'}$. These 8 possibilities can be grouped by their effect on the $q$-grading, as follows:  The composition, $\del_{\rho_3} \circ \del_{\rho_2} \circ \del_{\rho_1}$, coincides with $\delta$ and hence, preserves the quantum grading (see \Cref{thm:2ColorVertexHomology}).  Three 3-edge compositions involve a single $\widetilde{\del}_{\rho_i}$ map and therefore increase the quantum grading by $n$, since $\widetilde{m}$, $\widetilde{\Delta}$ and $\widetilde{\eta}$ each increase the quantum grading by $n$.  Three 3-edge compositions involve exactly two $\widetilde{\del}_{\rho_i}$ maps and therefore increase the quantum grading by $2n$.  Finally, a single 3-edge composition involves three $\widetilde{\del}_{\rho_i}$ maps and increases the quantum grading by $3n$.  This information is summarized in \Cref{table:QuantumJump}:
\begin {table}[ht]
\begin{center}
  \setlength\extrarowheight{6pt}
\begin{tabular}{| c | c |}
\hline
Jump in $q$-grading & Composition \\
\hline 
$0$ & $\del_{\rho_3} \circ \del_{\rho_2} \circ \del_{\rho_1}$ \\
\hline
$n$ & $\del_{\rho_3} \circ \widetilde{\del}_{\rho_2} \circ \del_{\rho_1}$, $\widetilde{\del}_{\rho_3} \circ \del_{\rho_2} \circ \del_{\rho_1}$, $\del_{\rho_3} \circ \del_{\rho_2} \circ \widetilde{\del}_{\rho_1}$ \\
\hline
$2n$ & $\widetilde{\del}_{\rho_3} \circ \del_{\rho_2} \circ \widetilde{\del}_{\rho_1}$, $\del_{\rho_3} \circ \widetilde{\del}_{\rho_2} \circ \widetilde{\del}_{\rho_1}$, $\widetilde{\del}_{\rho_3} \circ \widetilde{\del}_{\rho_2} \circ \del_{\rho_1}$ \\
\hline
$3n$ & $\widetilde{\del}_{\rho_3} \circ \widetilde{\del}_{\rho_2} \circ \widetilde{\del}_{\rho_1}$\\
\hline
\end{tabular}
\end{center}
\caption{The effect of each composition on the quantum grading.}
\label{table:QuantumJump}
\end{table}

Define the map $(\delta_{\nu\nu'})_{k}$ based upon the path of edges $\rho_1, \rho_2, \rho_3$ and the jump in the quantum grading $k\in\{0,n,2n,3n\}$ by taking the sum of compositions in \Cref{table:QuantumJump} for $k$. For example, for Configuration 1 in \Cref{fig:VertexCases}, if $\Gamma_\nu$ to $\Gamma_{\nu'}$ is the top $3$-edge path in \Cref{fig:Case1}, then for $k=2n$, 
$$(\delta_{\nu\nu'})_{2n} = \widetilde{\eta}\circ m \circ \widetilde{m} + \eta \circ \widetilde{m} \circ \widetilde{m} + \widetilde{\eta}\circ \widetilde{m} \circ m.$$
Next, define the four different graded maps, denoted $\delta_0, \delta_n, \delta_{2n}$ and $\delta_{3n}$, as the sum of all maps $(\delta_{\nu\nu'})_{k}$ for each edge from $\Gamma_\nu$ to some other state where $k$ is the change in quantum degree. Therefore, for $k\in\{0,n,2n,3n\}$ and $w\in V_\nu \subset C_n^{i,*}(\Gamma)$, 
\begin{equation}
\delta_{k}(w) = \sum_{\substack{\rho \ \mbox{\tiny such that }\\ \mbox{\tiny Tail}(\rho) = \nu}} \sign(\rho)(\delta_{\nu\nu'})_k(w),
\label{eq:delta-k-defintion}
\end{equation}
where $\sign(\rho) = (-1)^{\#\{\mbox{$1$'s to the left of $\ast$ in $\rho$}\}}$.  Extend this map linearly to all $C_n^{i,*}(\Gamma)$.

Below is an example of these maps, which will be discussed in more detail later in this section.

\begin{example}
\label{ex:ThetaLee}
Set $n=2$ and consider the homology of the $\theta$ graph shown in \Cref{fig:ThetaHomology}.  The first column (homological degree zero) has the following nontrivial maps:
$$\delta_2(1\otimes 1\otimes x) = \delta_2(1\otimes x\otimes 1) = \delta_2(x\otimes 1\otimes 1) = \sqrt{2}((1,0) +(0,1)),$$
$$\delta_2(1\otimes x\otimes x) = \delta_2(x\otimes 1\otimes x) = \delta_2(x\otimes x\otimes 1) = \sqrt{2}((x,0) +(0,x)),$$
$$\delta_4(x\otimes x\otimes x) = \sqrt{2}((1,0) +(0,1)).$$
The second column has the following nontrivial maps:
$$\delta_4(1,0) = -\delta_4(0,1) = \sqrt{2}(1\otimes 1\otimes x + 1\otimes x \otimes 1 + x \otimes 1 \otimes 1),$$
$$\delta_2(1,0) = -\delta_2(0,1) = \sqrt{2}(x\otimes x\otimes x),$$
$$\delta_4(x,0) = -\delta_4(0,x) = \sqrt{2}(1\otimes x\otimes x + x\otimes 1 \otimes x + x \otimes x \otimes 1),$$
$$\delta_6(x,0) = -\delta_6(0,x) = \sqrt{2}(1\otimes 1\otimes 1).$$
All other maps are zero.
\end{example}

Using the quantum grading, compositions of two of the $\delta_k$'s can be split into sums that have the same grading. For example, $\delta_0\delta_{3n}+\delta_{n}\delta_{2n}+\delta_{2n}\delta_{n}+\delta_{3n}\delta_0$ has grading $(1,3n)$.  When this is done, their sums are always equal to zero: 

\begin{proposition} For each $k\in \{0,1, \dots, 6\}$, sums of compositions that have the same grading are zero:
$$\sum_{\stackrel{0\leq p,q, \leq 3}{k=p+q}}\delta_{pn}\circ\delta_{qn} = 0,$$
and the sum of compositions has bigrading $(1,k)$.
\label{thm:LeeDiffs}\label{thm:diffsAntiCommute}
\end{proposition}

That is, $\delta_0\delta_0 = 0$, $\delta_{n}\delta_{0} + \delta_{0}\delta_n = 0$, $\delta_{2n}\delta_{0} +\delta_{n}\delta_{n}+ \delta_{0}\delta_{2n}=0$, and so on up to $\delta_{3n}\delta_{3n}=0$.  

Embedding the $\Gamma_\bullet$ hypercube in the bubbled blowup hypercube once again provides the necessary tools for working with these maps. It reduces showing the behavior of complicated-to-express diagrams that correspond to faces of the  hypercube of vertex states  to a sequence of  calculations corresponding to faces of the hypercube of states of the bubbled blowup. As will become evident, without the TQFT tools of \cite{ColorHomology}, proving this proposition for $n\gg 0$ would  be intractable. 

\begin{proof}
Consider a face joining two states $\Gamma_\nu$ and $\Gamma_{\nu''}$ in the $\Gamma_\bullet$ hypercube. Each 2-edge path of the face joining $\Gamma_\nu$ to $\Gamma_{\nu''}$ corresponds to a path of six edges in the hypercube of states of the bubbled blowup $\Gamma^B$ (see \Cref{fig:poly-faces}). Due to the quantum gradings,  compositions of maps with the same grading that take $w\in V_\nu\subset C^{i,j}_n(\Gamma)$ to an element of $C^{i,j+k}_n(\Gamma)$ for $k\in\{0,1,\dots, 6\}$ can be analyzed separately. Also, the claim about the bigrading is clear from the construction. 

\begin{figure}[h]
\includegraphics[scale=.28]{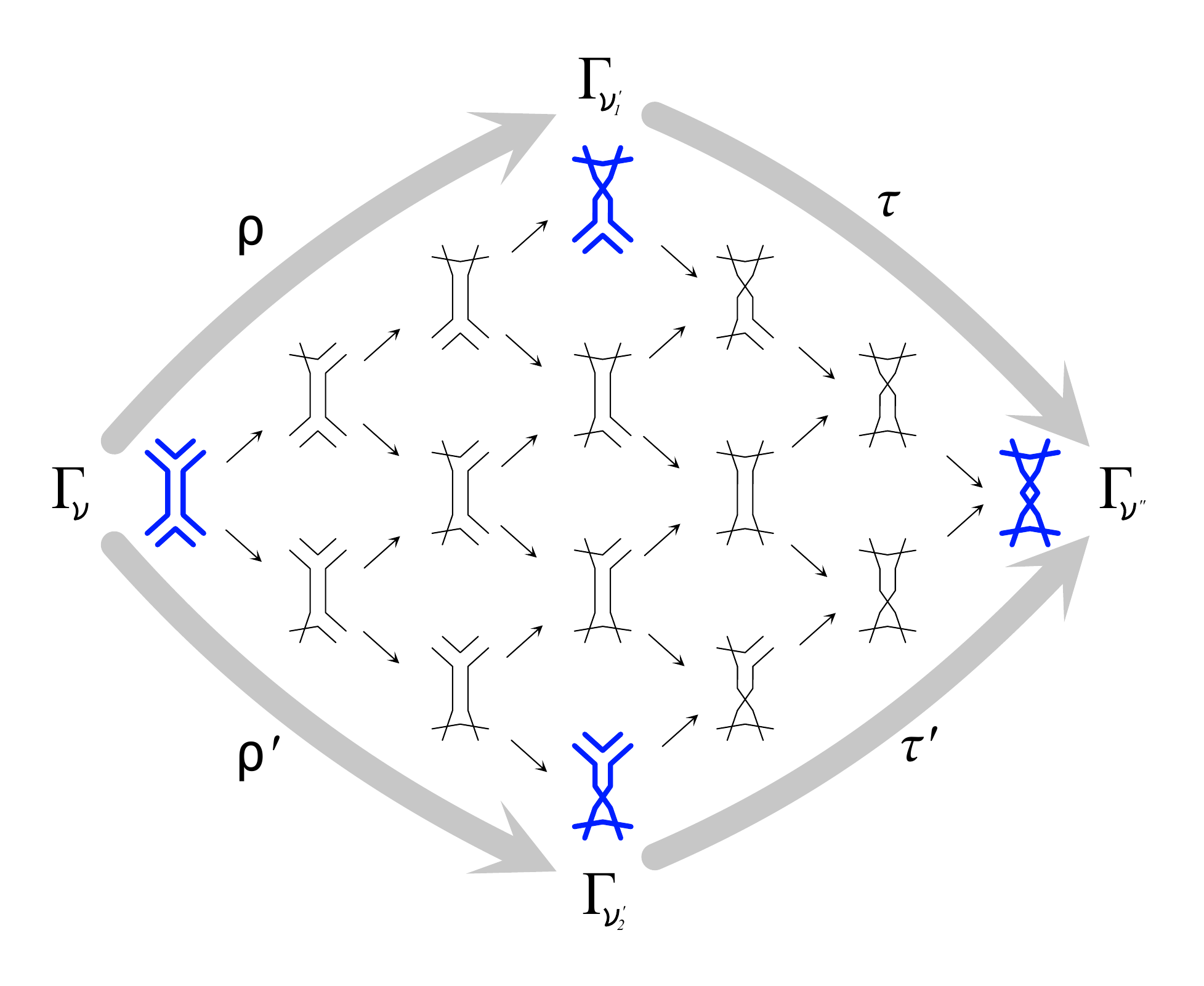}
\caption{A polygon of faces in the hypercube of states of $\Gamma^B$.}\label{fig:poly-faces}
\end{figure}

Each face in the hypercube of vertex states corresponds to a polygon of faces in the hypercube of states of the bubbled blowup.  \Cref{fig:poly-faces} shows an example of such a polygon where two vertices ``interact'' across an edge (most do not).  Every face in the $\Gamma_\bullet$ hypercube corresponds to a polygon in the $\Gamma^B$ hypercube like the one in the figure. 

To prove the proposition, a statement like the following needs to be shown for each face in the $\Gamma_\bullet$ hypercube: 
\begin{equation}
(\delta_\tau)_{2n}(\delta_\rho)_{0} + (\delta_\tau)_{n}(\delta_\rho)_{n} + (\delta_\tau)_{0}(\delta_\rho)_{2n} = (\delta_{\tau'})_{2n}(\delta_{\rho'})_{0} + (\delta_{\tau'})_{n}(\delta_{\rho'})_{n} + (\delta_{\tau'})_{0}(\delta_{\rho'})_{2n} \label{eq:6-edge-paths}
\end{equation}

\noindent This is the statement for $k=2$ in the proposition; the other statements are similar. The proof steps face-by-face through the polygon of the $\Gamma^B$ hypercube from the top $2$-edge path $\Gamma_\nu \ra \Gamma_{\nu'_1} \ra \Gamma_{\nu''}$ to the bottom $2$-edge path $\Gamma_\nu \ra \Gamma_{\nu'_2} \ra \Gamma_{\nu''}$ . Using the example of \Cref{fig:poly-faces}, the first step amounts to showing 
\begin{equation}
(\delta_\tau)_{2n}(\delta_\rho)_{0} + (\delta_\tau)_{n}(\delta_\rho)_{n} + (\delta_\tau)_{0}(\delta_\rho)_{2n} = (\delta_{\tau_1})_{2n}(\delta_{\rho_1})_{0} + (\delta_{\tau_1})_{n}(\delta_{\rho_1})_{n} + (\delta_{\tau_1})_{0}(\delta_{\rho_1})_{2n} \label{eq:first-step-6-edge-paths}
\end{equation}
for the top face shown in \Cref{fig:top-face}.

\begin{figure}[h]
\includegraphics[scale=.28]{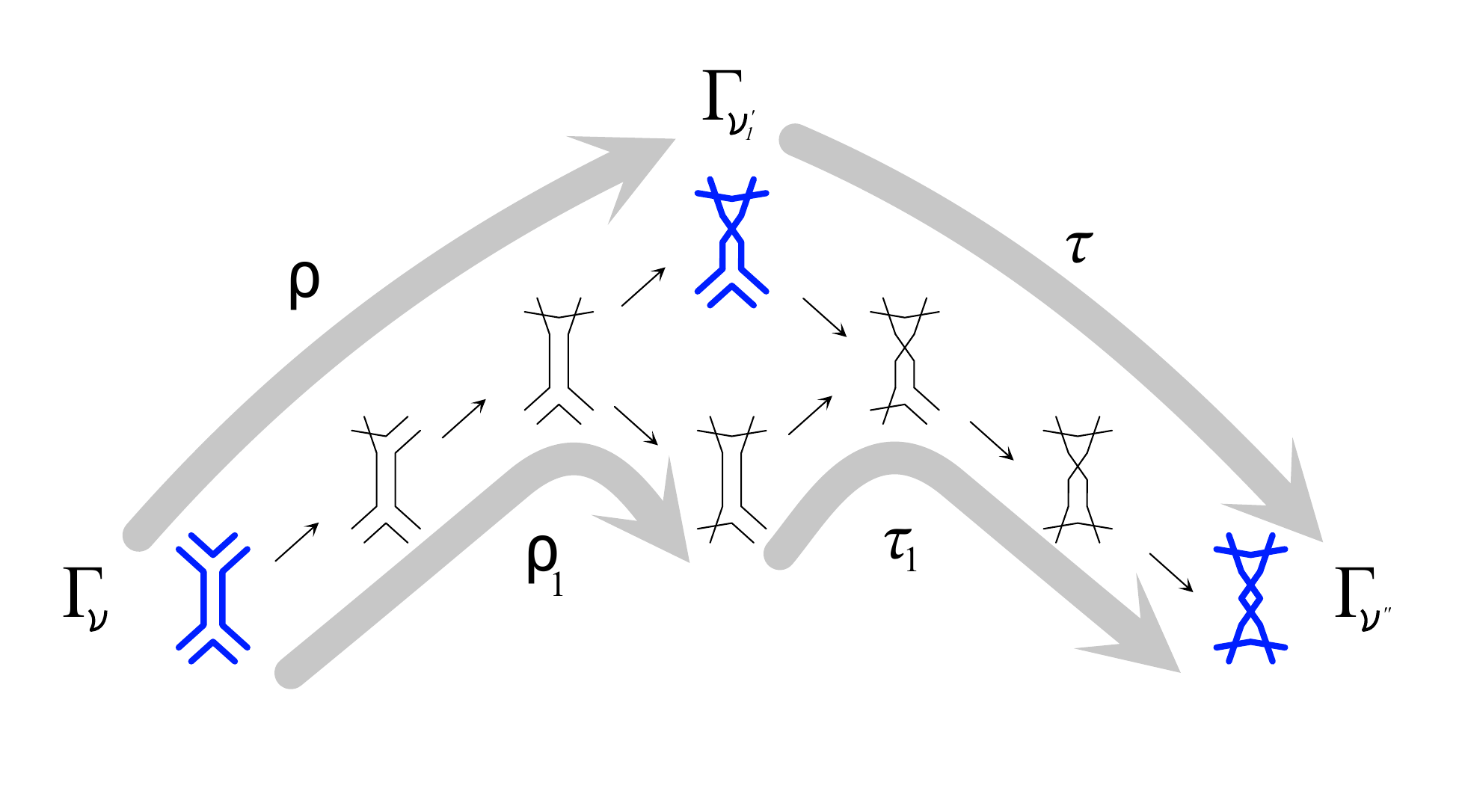}
\caption{Top face of the polygon in \Cref{fig:poly-faces}.}
\label{fig:top-face}
\end{figure}

Writing out the composition of maps (corresponding to the $\partial_{\rho_i}$'s  and $\widetilde{\partial}_{\rho_i}$'s from  \Cref{table:QuantumJump}) for each $3$-edge path in \Cref{eq:first-step-6-edge-paths}, the calculation reduces to proving the following about diagrams that correspond to faces in the $\Gamma^B$ hypercube:

\begin{enumerate}
\item  the commutativity of diagrams involving only $m$, $\Delta$, and $\eta$,
\item  the commutativity of diagrams involving only $\widetilde{m}$, $\widetilde{\Delta}$, and $\widetilde{\eta}$, and
\item  the equality of pairs of sums that correspond to $\partial \widetilde{\partial} + \widetilde{\partial}\partial$ for the two $2$-edge paths along the top and bottom, i.e., diagrams that ``mix''  $m$, $\Delta$, and $\eta$ maps with $\widetilde{m}$, $\widetilde{\Delta}$, and $\widetilde{\eta}$ maps.
\end{enumerate}
The first two statements are already known to commute. The last statement needs more explanation: Suppose the top face looked like the face in \Cref{fig:example-of-top-face}. In this case, the third statement translates into showing the following equality:
\begin{equation}
m\circ(\widetilde{\eta}\ot Id) +\widetilde{m}\circ(\eta\ot Id) = \eta\circ\widetilde{m} + \widetilde{\eta}\circ m.\label{eq:sum-of-compositions}
\end{equation} 
The lefthand side of the equation corresponds to the top $2$-edge path  and the righthand side corresponds to the bottom path.

\begin{figure}[H]
\includegraphics[scale=.60]{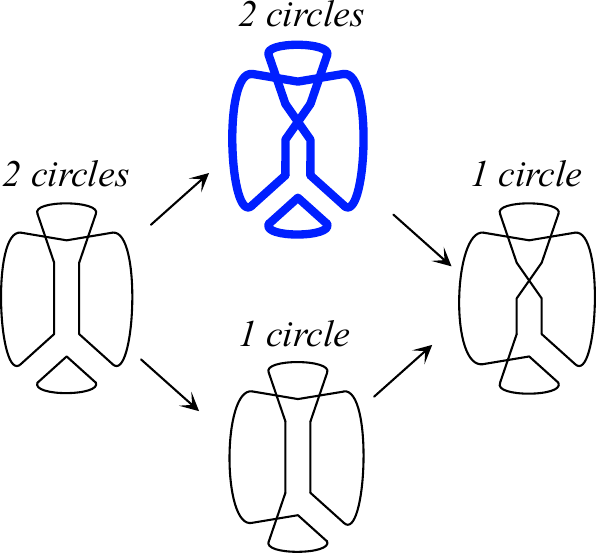}
\caption{An example of the top face of the polygon in \Cref{fig:top-face}.}\label{fig:example-of-top-face}
\end{figure}

Proving equalities like the one shown in \Cref{eq:sum-of-compositions} is the content of Theorem 5.3 in \cite{ColorHomology}. For example,  \Cref{eq:sum-of-compositions} is Equation (5) in the proof of that theorem. Ultimately, the reason why {\em all} pairs of sums of such maps for {\em all} $n$ are equal for {\em all} diagrams corresponding to faces in the $\Gamma^B$ hypercube rests on the category theory explained in Section 9 of \cite{ColorHomology}. 

Since the terms of each side of \Cref{eq:sum-of-compositions} can be paired up into directly commuting diagrams as in Statements (1) or (2) above, or sums of pairs that are equal as in Statement (3) above, the equality holds. The choice of restricting to each $k\in \{0,1,\dots, 6\}$ guarantees that the correct pairs appear on both sides. Moving face-by-face through the polygon in a similar fashion shows that \Cref{eq:6-edge-paths} holds in general.

To finish the proof, note that the chosen sign convention in the definition of the $\delta_k$'s (see \Cref{eq:delta-k-defintion}) guarantees that one of the maps $\delta_\rho$, $\delta_\tau$, $\delta_{\rho'}$, or $\delta_{\tau'}$ will have a minus sign, or three will come with a minus sign and the other is positive. Thus, moving the righthand side of  \Cref{eq:6-edge-paths} to the lefthand side  for all relevant faces in the $\Gamma_\bullet$ hypercube gives the equation:
$$2(\delta_{2n}\delta_{0} +\delta_{n}\delta_{n}+ \delta_{0}\delta_{2n})=0,$$
which proves the theorem for this case. All other cases for $k\in \{0,1,\dots,6\}$ follow the same reasoning.
\end{proof}

Notice that the maps $\delta_{k}$ for $k\in\{0,n,2n,3n\}$ are also invariant under the ribbon moves (cf. \cite{BKR}) for the same reasons that $\delta$ is invariant.  In order to obtain a differential $\widehat{\delta}$ that mimics the properties of the Lee differential (cf. \cite{ColorHomology,LeeHomo}), first take the sum of the maps:
\begin{equation}\widetilde{\delta} = \delta_n + \delta_{2n} + \delta_{3n}.\label{eq:delta-hat}\end{equation}
While this map does not preserve quantum grading, it still  respects a filtration described in the next section.  When combined with $\delta=\delta_0$,  \Cref{thm:diffsAntiCommute}  implies
$$(\delta + \widetilde{\delta})^2 = (\delta_0 + \delta_{n} +\delta_{2n}+\delta_{3n})^2 = 0.$$ 
Forgetting the quantum grading (for the moment), $\widehat{\delta} := \delta + \widetilde{\delta}$ is a new differential, called the {\em filtered differential}.

At this point, it will  become advantageous to work with the Frobenius algebra $V_t=\mathbbm{k}[x]/(x^n-t)$ for some $t\in [0,1]$ instead of just $V_0=\mathbbm{k}[x]/(x^n)$. All theorems above for $t=0$ also hold  for $t>0$. Hence:

\begin{proposition} Let $t\in [0,1]$ and build  graded chain groups $C^{i,j}_n(\Gamma)$ as in the definitions and theorems above using $V_t=\mathbbm{k}[x]/(x^n-t)$.  Then both $(C^{i,j}_n(\Gamma), \delta)$ and $(C^{i,*}_n(\Gamma), \widehat{\delta})$ are chain complexes. 
\end{proposition}

The homology of the $\delta$ complex defined for any $t\in [0,1]$ is isomorphic to the bigraded $n$-color vertex homology, which was originally defined for $t=0$. For any two values of $t\in [0,1]$, the homology theories for the second complex based upon those values are isomorphic as well. However, when $t=1$, the algebra $V_1$ has excellent properties for describing colorings of  faces of ribbon graphs. Therefore, in the definition below, the chain complex $(C^{i,*}_n(\Gamma), \widehat{\delta})$ is defined for $t=1$. To help distinguish it from the chain complex for $V=V_0$, and to simplify notation in what follows, the notation $\widehat{V}$ will be used for $V_1$, and $(\widehat{C}^i_n(\Gamma),\widehat{\delta})$ will be used for the chain complex that is based upon $\widehat{V}$. 

\begin{definition}
Let $G(V,E)$ be a connected trivalent graph and $\Gamma$ be a ribbon graph of $G$.  The {\em filtered $n$-color vertex homology of $\Gamma$}  is 
\begin{equation*}
\widehat{VCH}_n^*(\Gamma,\mathbbm{k}):=H(\widehat{C}_n^{*}(\Gamma), \widehat{\delta}).
\end{equation*}
\label{def:FilteredColorVertexHomology}
\end{definition}

\subsection{A spectral sequence for filtered $n$-color vertex homology} In this subsection, we describe the spectral sequence for the filtered $n$-color vertex homology.  For $v\in C^{i,j}_n(\Gamma) \subset \widehat{C}^i_n(\Gamma)$, define a grading function for the $q$-grading so that $q(v)=j$. If $v\in \widehat{C}^i_n(\Gamma)$ is written as a sum of monomials $v=v_1+v_2+\dots +v_\ell$, then set $q(v) = \min\{q(v_i) \ | \ i=1,\ldots,\ell\}$.  Define a finite length decreasing filtration on $\widehat{C}^{i}_n(\Gamma)$ by
$$\mathcal{F}^p(\widehat{C}_n^*(\Gamma)) = \{ v\in \widehat{C}^{*}_n(\Gamma) \ | \ q(v)\geq p\}.$$

\bigskip

All three maps $\delta=\delta_0$, $\tilde{\delta}$, and $\widehat{\delta}$ respect the filtration, e.g., $\widehat{\delta}(\mathcal{F}^p(\widehat{C}_n^i(\Gamma)))\subset \mathcal{F}^p(\widehat{C}_n^{i+1}(\Gamma))$, and $\widetilde{\delta}$ is a filtered map of bigrading $(1,n)$. Therefore, the filtration induces a spectral sequence: The $E_0$-page is the original bigraded complex $C_n^{*,*}(\Gamma)$ with the bigraded differential $d_0:=\delta$. The first page, $E_1$, is  the bigraded $n$-color vertex homology.  The higher differentials $d_r$ on $E_r(\Gamma)$ are zero except when $r$ is a multiple of $n$. When $r$ is a multiple of $n$, then $d_r$ has bigrading $(1,r)$.  The differentials $d_r$ for $r>0$ are based upon $\widetilde{\delta}$.

A standard theorem (cf. \cite{McCleary}) of spectral sequences then implies:

\begin{theorem} \label{theorem:spectral-sequence-converges} Let $G(V,E)$ be a trivalent graph and $\Gamma$ be a ribbon graph of it. Let $n\in\NN$ and $\mathbbm{k}$ be a ring in which $\sqrt{n}$ is defined. There exists a spectral sequence that has the form
$$E^{i,j}_1 = VCH_n^{i,j}(\Gamma;\mathbbm{k}) \ \implies \ \widehat{VCH}^i_n(\Gamma;\mathbbm{k})$$
when expressed in terms of the gradings of $C^{i,j}_n(\Gamma)$. Thus, the $E_\infty$-page of this spectral sequence is isomorphic to $\widehat{VCH}^*_n(\Gamma;\mathbbm{k})$.\label{thm:spectral-sequence-for-filtered-homology}
\end{theorem}

\begin{example}
\label{ex:Theta2}
Returning to the $\theta$ graph and the $n=2$ case,  observe from the maps in \Cref{ex:ThetaLee} imply that $\widehat{VCH}_2^0(\Gamma_\theta;\mathbbm{k})$ and $\widehat{VCH}_2^2(\Gamma_\theta;\mathbbm{k})$ both have rank six, and $\widehat{VCH}_2^1(\Gamma_\theta;\mathbbm{k})$ has rank zero.  Note that the Euler characteristic of the filtered $2$-color vertex homology is twelve, which is twice the number of 3-edge colorings. This example supports the conclusion of Theorem~\ref{thm:oddMatchingsCancel}. 
\end{example}

At this point, the reason why nontrivial classes in \Cref{ex:Theta2} are counting $3$-edge colorings appears to be mysterious. We show how to interpret these classes as types of face colorings and perfect matchings of the ribbon graph $\Gamma$ next.

\subsection{A color basis.}
The new differential  $\widehat{\delta} = \delta + \widetilde{\delta}$ produces a homology theory whose properties are perhaps best understood after a change of basis.   Following Section 5.3 of \cite{ColorHomology}, this basis can be thought of as $n$ colors, $c_0,\ldots, c_{n-1}$. In order to use this basis, we take $\mathbbm{k} =\CC$ and  $\widehat{V} = \CC[x]/(x^n-1)$ as the algebra. From now on in this paper, when using $\mathbbm{k}=\CC$, the ring will be suppressed from the notation, i.e., $\widehat{VCH}^i_n(\Gamma)$ will stand for $\widehat{VCH}^i_n(\Gamma;\CC)$.

\begin{definition} Let $n$ be a positive integer and let $\lambda = e^{\frac{2\pi \mathrm{i}}{n}}\in \CC$ be an $n$th root of unity. The {\em color basis} of $\widehat{V}=\BC[x]/(x^n-1)$ consists of the elements,
$$c_i := \frac{1}{n}\left(1+ \lambda^i x +\lambda^{2i}x^2+ \lambda^{3i}x^3+\cdots+\lambda^{(n-1)i}x^{n-1}\right),$$
for $0\leq i \leq n-1$. \label{def:colorbasis}
\end{definition}
 
In the color basis, the vector space $\widehat{V}_\nu$ is generated by tensor products of colors $c_i$.  That is, if there are $k$ circles in a state $\Gamma_\nu$, then for  ${I\in \{0,\ldots, n-1\}^{k}}$, then $c_I$ is a multi-index that stands for $c_{i_1}\ot\cdots \ot c_{i_{k}}$. Thus, $\widehat{V}_\nu = \langle c_I\rangle.$
Following \cite{ColorHomology}, call each basis element $c_I$ a {\em coloring} of the state and call a linear combination of colorings a {\em mixture}.

The color basis is well-behaved with regard to multiplication and the local differentials in $\widehat{V}$.

\begin{lemma}[Lemma 5.9 of \cite{ColorHomology}]\label{lem:widehat-maps}
In the color basis, the following equations hold:
\begin{enumerate}
\item $c_i \cdot c_j =  \delta^{ij} c_j,$ hence $\widehat{m}(c_i \ot c_j) =  \delta^{ij} c_j$,
\item $\widehat{\Delta}(c_i) =  n \lambda^{-2mi} c_i \ot c_i,$ and
\item $\widehat{\eta}(c_i) = \sqrt n \lambda^{-mi} c_i$.
\end{enumerate}
\end{lemma}

Suppose that $\Gamma_\nu \rightarrow \Gamma_{\nu'}$ is an edge in the hypercube of vertex states and $\widehat{\delta}_{\nu\nu'}$ is the local filtered differential associated to it.  Working in the color basis $\langle c_I\rangle$, and applying \Cref{lem:widehat-maps}, we see that $\widehat{\delta}_{\nu\nu'}(c_I)$ is zero if and only if, at the vertex being changed from a $0$-smoothing to a $1$-smoothing, the faces incident to the vertex are colored with exactly three different colors or exactly two different colors (cf. the diagrams on the right of \Cref{fig:possibleColors}). It is nonzero if the face(s) are colored with a single color.  The face colorings can be seen by applying the maps in  \Cref{lem:widehat-maps} to Configurations 1-7 in \Cref{fig:VertexCases} (cf. \Cref{fig:Case1} to \Cref{fig:Case7}).

This takes a particularly nice form when calculating the filtered $2$-color vertex homology of a  ribbon graph $\Gamma$ in degree zero.  Recall that the all-zero state, $\Gamma_{\vec{0}}$, is a set of circles that corresponds to the faces of $\Gamma$. Consider $\widehat{\delta}:\widehat{C}^{0}_2(\Gamma) \ra \widehat{C}^1_2(\Gamma)$. A coloring $c_I$ of the circles in the all-zero state satisfies $\widehat{\delta}(c_I) = 0$ when the circles incident to each vertex are colored by exactly two colors. This can be summarized as: 

\begin{proposition}\label{prop:degree-zero-harmonic-colorings}
Let $G(V,E)$ be a trivalent graph and $\Gamma$ be a ribbon graph of it.  Then a color basis for $\widehat{VCH}^0_2(\Gamma)$ is the set of all  colorings of the faces of $\Gamma$ where the faces incident to each vertex of $\Gamma$ are colored with exactly two colors.
\end{proposition}
 
This proposition is the first step towards proving Theorem~\ref{thm:perfectMatchings}. (Compare this proposition to Theorem D.3 in \cite{ColorHomology}, which has a similar flavor.)

\subsection{A metric and Hodge decomposition} In Section 6 of \cite{ColorHomology}, a metric on the chain complex was used to define an adjoint operator and prove a Hodge decomposition theorem.  A similar construction goes through in this context as well.  We recall the basic definitions and theorems here and refer the reader to Section 6 of \cite{ColorHomology} for more details.

\begin{definition}[Section 6.1 of \cite{ColorHomology}] 
For two colorings $c_I, c_J \in \widehat{V}_\nu$ and two complex numbers $a,b\in\BC$, define a Hermitian metric $\langle \ , \ \rangle: \widehat{V}_\nu \ot \widehat{V}_\nu \ra \BC$
by $$\langle ac_I, bc_J \rangle = a\bar{b}\delta_{IJ},$$ where ${\delta_{IJ}=1}$ if $I=J$ and is zero otherwise. Extend this metric to all of $\widehat{C}_n^*(\Gamma)$.
\end{definition}

This metric defines an adjoint operator to $\widehat{\delta}$ just as in \cite{ColorHomology}: the adjoint $\widehat{\delta}^*:\widehat{C}_n^i(\Gamma) \ra \widehat{C}_n^{i-1}(\Gamma)$ is the operator that satisfies
$$\langle \widehat{\delta}(c),d \rangle = \langle c , \widehat{\delta}^*(d)\rangle.$$

The adjoint operator $\widehat{\delta}^*:\widehat{C}_n^i(\Gamma)\ra \widehat{C}_n^{i-1}(\Gamma)$ is then a differential composed from $3$-edge paths in the hypercube of states of $\Gamma^B$ of adjoints maps.  These compositions can be written down explicitly using the following maps:

\begin{lemma}[See Lemma 6.2 in \cite{ColorHomology}]\label{lemma:adjoint-hat-operators}
For an edge in the hypercube of states of $\Gamma^B$ given by $\widehat{\del}_{\alpha\alpha'}:\widehat{V}_{\alpha} \ra \widehat{V}_{\alpha'}$ with $|\alpha| = |\alpha'| -1$, the adjoint of $\widehat{\del}_{\alpha\alpha'}$, denoted $\widehat{\del}^*_{\alpha\alpha'}:\widehat{V}_{\alpha'} \ra \widehat{V}_{\alpha}$, is given by the following adjoints of $\widehat{m}$, $\widehat{\Delta}$, and $\widehat{\eta}$ on colors $c_i, c_j$:
\begin{eqnarray}\label{eqn:adjoint-of-m-delta-eta}
\widehat{m}^*(c_i) &=& c_i\ot c_i,\\ \nonumber
\widehat{\Delta}^*(c_i\ot c_j) &=& n \lambda^{2mi}\delta^{ij}c_i, \mbox{ and}\\ \nonumber
\widehat{\eta}^*(c_i) &=& \sqrt{n} \lambda^{mi}c_i.\nonumber
\end{eqnarray}
\end{lemma}

Define a Laplace operator by $\slashed{\Delta} := (\widehat{\delta}+\widehat{\delta}^*)^2$. Denote the space of  {\em harmonic mixtures}, i.e., $c\in \widehat{C}_n^i(\Gamma)$ such that $\slashed{\Delta}(c)=0$, by $\widehat{\mathcal{VCH}}_n^i(\Gamma)$.  Standard arguments using the metric show that $\slashed{\Delta}(c)=0$ if and only if $\widehat{\delta}(c)=0$ and $\widehat{\delta}^*(c)=0$. Furthermore, again by standard arguments, the space $\widehat{C}_n^i(\Gamma)$ can be decomposed into subspaces using a Hodge-like theorem for $\widehat{\delta}$:

\begin{lemma}[Hodge decomposition theorem for $\widehat{\delta}$] \label{lem:hodge-decomposition-theorem} The space $\widehat{C}_n^i(\Gamma)$ can be decomposed as
$$\widehat{C}_n^i(\Gamma) = \widehat{\mathcal{VCH}}_n^i(\Gamma)\oplus \widehat{\delta}\widehat{C}_n^{i-1}(\Gamma)\oplus \widehat{\delta}^*\widehat{C}_n^{i+1}(\Gamma).$$
In particular, each element $c\in\widehat{C}_n^i(\Gamma)$ can be uniquely decomposed into the sum $c=c_h +\widehat{\delta}c_- + \widehat{\delta}^*c_+$.
\end{lemma}

In Lemma 6.4 of \cite{ColorHomology}, it was shown that along any edge of the hypercube of states,  from $\Gamma_\alpha$ to $\Gamma_{\alpha'}$, the differential takes each element of the color basis for $V_\alpha$ to a multiple of a single element of the color basis for $V_{\alpha'}$.  The same is true for filtered $n$-color vertex homology using $3$-edge paths in $\Gamma^B$.  In fact, the following lemma is the key proposition of this section.

\begin{lemma}
\label{lem:Colors2Colors}
Let $G(V,E)$ be a connected trivalent graph and $\Gamma$ a ribbon graph of $G$ represented by a ribbon diagram.  Let $\widehat{\delta}_{\nu\nu'}: \widehat{V}_\nu\ra \widehat{V}_{\nu'}$ be the map corresponding to an edge in the hypercube of vertex states for $\Gamma_\bullet$.  If $\langle c_I\rangle$ is a color basis for $\widehat{V}_\nu = \widehat{V}^{\otimes k_\nu}$ and $\langle c'_{I}\rangle$ is a color basis for $\widehat{V}_{\nu'} = \widehat{V}^{\otimes k'_{\nu}}$, then either 
\begin{enumerate}
\item $\widehat{\delta}_{\nu\nu'} (c_I) = 0$, or 
\item $\widehat{\delta}_{\nu\nu'} (c_I)$ is a nonzero multiple of exactly one color from $\langle c'_{I}\rangle$.
\end{enumerate} 
In particular, let $a=\sum_I a^I c_I \in \widehat{V}_\nu$ such that $a^I\in \BC$ for all $I$. If $\widehat{\delta}_{\nu\nu'}(a)=0$, then for each $I$,  either $\widehat{\delta}_{\nu\nu'}(c_I)=0$ or, if $\widehat{\delta}_{\nu\nu'}(c_I)\not=0$, then $a^I=0$.
 All statements hold for $\widehat{\delta}^*_{\nu\nu'}: \widehat{V}_{\nu'} \ra \widehat{V}_{\nu}$ as well. 
\end{lemma}

\begin{proof}
The lemma follows from the same argument as Lemma 6.4 of \cite{ColorHomology} together with the observation that the map $\widehat{\delta}_{\nu\nu'}$ is nonzero precisely when the face(s) incident to the  vertex smoothing site are colored with exactly one color (see  \Cref{lem:widehat-maps} and the discussion directly after it).  Note that all seven configurations in \Cref{fig:VertexCases} need to be checked using the diagrams shown in \Cref{app:differentials}.
\end{proof}

Following the same program as Section 6 of \cite{ColorHomology}, we define the space of \emph{harmonic mixtures of a state}, which will be denoted in this paper by $\widehat{\mathcal{VCH}}_n(\Gamma_\nu)$.  That is, given a state $\Gamma_\nu$,  define $\widehat{\delta}_\nu:\widehat{V}_\nu \ra \widehat{C}_n^{i+1}(\Gamma)$ by taking the sum of all  nontrivial differentials from  $\widehat{V}_\nu$ to  degree  $(|\nu|+1)$-states. Similarly, $\widehat{\delta}_\nu^*: \widehat{V}_\nu \ra \widehat{C}_n^{i-1}(\Gamma)$ is the sum of all nontrivial adjoint maps from $\widehat{V}_\nu$ to  degree $(|\nu|-1)$-states.  Define

$$\widehat{\mathcal{VCH}}_n(\Gamma_\nu) := ker \widehat{\delta}_\nu \bigcap ker \widehat{\delta}^*_\nu.$$

The set of harmonic colorings of each state form a basis for the filtered $n$-color vertex homology.  We omit the proof here and refer the reader to Sections~6.2 to 6.5 of \cite{ColorHomology} for the details, which are essentially the same due to the key \Cref{lem:Colors2Colors}.

\begin{proposition}[cf. Proposition 6.12 of \cite{ColorHomology}] Let $\Gamma$ be a ribbon graph of a connected trivalent graph $G(V,E)$.  Using the color basis, $$\widehat{\mathcal{VCH}}^i_n(\Gamma) = \bigoplus_{|\nu|=i} \widehat{\mathcal{VCH}}_n(\Gamma_\nu).$$ 
\label{prop:direct-sum-equals-all-harmonics}
\end{proposition}

Of course,  by the Hodge decomposition, the space of harmonic mixtures is isomorphic to the filtered $n$-color vertex homology, i.e., $$ \widehat{VCH}^i_n(\Gamma) \cong \widehat{\mathcal{VCH}}^i_n(\Gamma)$$
for all $i=0,\ldots, |V|$.  This implies that, while the $E_1$-page of the spectral sequence involves many different states in the computation of $VCH_n^{i,j}(\Gamma)$, by the $E_\infty$-page, only colorings on specific states matter, i.e., the colorings of states appear out of the ``quantum fuzz'' of earlier pages.

\subsection{Harmonic colorings of a vertex state} In light of \Cref{prop:direct-sum-equals-all-harmonics}, from now on we focus on the harmonic colorings of a vertex state, i.e., colorings $c_I$ such that $\widehat{\delta}_\nu(c_I)=0$ and $\widehat{\delta}^*_\nu(c_I)=0$. More can be said about them. In fact, \Cref{prop:degree-zero-harmonic-colorings} can now be generalized to all  states in the hypercube of vertex states of $\Gamma_\bullet$:

\begin{proposition}[compare to Proposition 6.15 of \cite{ColorHomology}] Let $\Gamma$ be a ribbon graph of a connected trivalent graph $G(V,E)$.  Each harmonic coloring $c_I\in \widehat{\mathcal{VCH}}_n(\Gamma_\nu)$ of a vertex state $\Gamma_\nu$ of $\Gamma$ has the property that  there are at least two colors present at every vertex smoothing site (cf. the top and bottom pictures of  \Cref{fig:possibleColors}).
\label{prop:harmonic-vertex-smoothing-sites}
\end{proposition}

\begin{proof} This proposition follows from \Cref{lem:widehat-maps}, the discussion following it,  the definition of the adjoint maps, and \Cref{lemma:adjoint-hat-operators}.  More explicitly, apply $\widehat{\delta}_\nu$ and $\widehat{\delta}^*_\nu$ to Configurations 1-7 in \Cref{fig:VertexCases} using $3$-edge paths in \Cref{fig:Case1} to \Cref{fig:Case7}. The harmonic colorings are precisely the ones that have at least two colors present at every vertex smoothing site. \end{proof}

To translate these conditions into more descriptive language, we introduce some new types of graph colorings that will be useful for interpreting the meaning of the harmonic classes.  In fact, \Cref{prop:harmonic-vertex-smoothing-sites} motivates \Cref{def:proper-and-partial-face-colorings-at-a-vertex} below.

To get started, recall the notion of state graphs from Section~6.6 of \cite{ColorHomology}. There is a correspondence between the circles in a vertex state $\Gamma_\nu$ and the faces of the ribbon graph determined by $\Gamma_\nu$, which is called a {\em state graph}. Hence, a coloring of the faces of a vertex state $\Gamma_\nu$ means a choice of colors $\{c_0,c_1,\ldots, c_{n-1}\}$ for each circle/face of the state graph $\Gamma_\nu$, where the $c_i$'s can be thought of as blue, red, purple, etc.

\begin{definition}\label{def:proper-and-partial-face-colorings-at-a-vertex}
Let $G(V,E)$ be a trivalent graph and $\Gamma$ a ribbon graph of it. Given a  coloring of the faces of some state graph $\Gamma_\nu$ and some vertex $v\in V$,  define 

\begin{enumerate}
\item a \emph{proper $n$-face coloring at $v$} to be a coloring of the faces incident to $v$ such that there are three distinct colors among the faces (top picture in \Cref{fig:possibleColors}), and 

\item a \emph{partial $n$-face coloring at $v$} to be a  coloring of the faces incident to $v$ such that there are exactly two colors  among the faces (bottom picture in \Cref{fig:possibleColors}) .
\end{enumerate}
\end{definition}

A third notion, an improper $n$-face coloring at $v$, is a coloring where all face(s) are the same color at $v$. It will not be needed in this paper.

\begin{figure}[H]
\includegraphics[scale=.45]{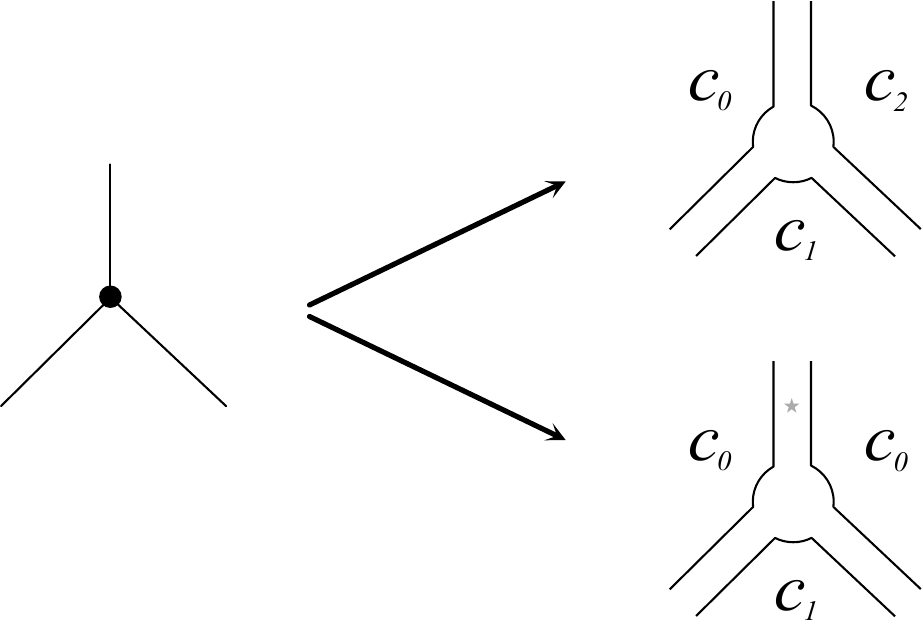}
\caption{Two possibilities for the harmonic coloring at a vertex, up to symmetry.}\label{fig:possibleColors}
\end{figure}

\begin{remark} Proper and partial face colorings were defined here so that they could be generalized to $r$-regular graphs in future research, see \Cref{fig:4valentCases} in \Cref{sec:4valent} for example. 
\end{remark}

\begin{definition}\label{def:PartialFace}
Let $G(V,E)$ be a trivalent graph and $\Gamma$ a ribbon graph of $G$.  A  coloring of the faces of a vertex state $\Gamma_\nu$ of $\Gamma$ is called a \emph{partial $n$-face coloring} if the coloring is proper or partial at each vertex of $G$. 
\end{definition}

If a partial $n$-face coloring of $\Gamma_\nu$ is proper at each vertex $v\in V$, then this coloring is just the usual notion of a {\em proper $n$-face coloring}, or simply, an {\em $n$-face coloring}. The other extreme, i.e., a coloring where all vertices are partially colored, is also interesting  because they can be identified with perfect matchings of $G$. First, all partial $n$-face colorings generate matching sets:

\begin{definition}
Let $G(V,E)$ be a trivalent graph and $\Gamma$ a ribbon graph of it.  A partial $n$-face coloring of state graph $\Gamma_\nu$ of $\Gamma$ generates an {\em induced matching set $M$}, $M\subset E$, on $G$ by selecting all edges that are incident to face(s) that share the same color (cf. the starred edge in the bottom picture of \Cref{fig:possibleColors}). 
\end{definition}

For example, for any partial $n$-face coloring, a bridge edge of $G$ will always be in the induced matching set for the coloring since all ribbon graphs of $G$ will have only one face incident to the bridge edge.  

We can now relate basis elements of $\widehat{\mathcal{VCH}}_n(\Gamma_\nu)$ to graph theory notions of face coloring and  to perfect matchings:

\begin{proposition} Let $\Gamma$ be a ribbon graph of a connected trivalent graph $G(V,E)$.  Harmonic colorings of $\widehat{\mathcal{VCH}}_n(\Gamma_\nu)$ are in one-to-one correspondence with  partial $n$-face colorings of the state graph $\Gamma_\nu$. Furthermore, each partial $n$-face coloring induces a matching set $M\subset E$ such that 
\begin{enumerate}
\item when every vertex of $G$ is a proper $n$-face coloring, then $M=\emptyset $, and
\item when every vertex of $G$ is a partial $n$-face coloring, then $M$ is a perfect matching.
\end{enumerate}                                            
\label{lem:harmonicsToColors}
\end{proposition}

\begin{proof} Follows from the discussion above and \Cref{prop:harmonic-vertex-smoothing-sites}.\end{proof}

The second statement of \Cref{lem:harmonicsToColors} motivates the following definition:

\begin{definition} Let $\Gamma$ be a ribbon graph of a connected trivalent graph $G(V,E)$. A harmonic coloring on a  vertex state $\Gamma_\nu$  is called a \emph{perfect matching face coloring} if there are at least two distinct colors present at every vertex smoothing site.
\label{def:PMColoring}
\end{definition}

When $n=2$, \Cref{lem:harmonicsToColors} implies that harmonic colorings {\em are} perfect matching face colorings, i.e., there is no distinction between harmonic colorings and perfect matching face colorings (other than to emphasize that perfect matchings are induced by the harmonic colorings).  For $n>2$, harmonic colorings allow for three colors to be used at a vertex.  In such a case a partial $n$-face coloring will potentially have some number of vertices at which the faces are properly colored with the remainder being partially colored.  Each of the vertices where the faces are partially colored induce a matching edge for the graph and altogether a matching set for the harmonic coloring.  However,  more complicated phenomena may occur when $n\not=2$. For example, a perfect matching face coloring may use more than two colors overall even though separately there are only two colors present at each vertex. This leads to the interesting notion of a {\em colored perfect matching set}, something we will not explore further in this paper.  

\begin{remark} In \cite{BKM2}, perfect matching $n$-colorings were defined.  These should not be confused with the perfect matching face coloring defined above.  In that paper, the coloring was an edge-coloring that leaves the edges of the perfect matching uncolored.  Here, a perfect matching face coloring is a special type of face coloring that specifies a perfect matching. 
\end{remark}

\subsection{Filtered $2$-color vertex homology and perfect matchings}
\label{sec:TwoColor}  In light of \Cref{prop:degree-zero-harmonic-colorings}, \Cref{prop:direct-sum-equals-all-harmonics}, \Cref{prop:harmonic-vertex-smoothing-sites}, \Cref{lem:harmonicsToColors}, and \Cref{def:PMColoring}, we obtain ``half'' of Theorem~\ref{thm:perfectMatchings} and a little bit more:

\begin{proposition}
\label{thm:PMClasses}
Let $\Gamma$ be a ribbon graph of a connected trivalent graph $G(V,E)$. The filtered $2$-color vertex homology, $\widehat{VCH}_2^*(\Gamma)$, has a basis given by perfect matching face colorings of the vertex states. In particular, 
$$\text{dim } \widehat{VCH}_2^{0}(\Gamma) \leq 2\cdot \#\{\text{perfect matchings of }G\}.$$
\end{proposition}

\begin{proof} Only the inequality needs further elaboration. In fact, the  question that remains is whether two different harmonic colorings, $c_I$ and $c_J$, can induce the same perfect matching. Suppose they do. 

If the color of each face of $c_I$ is the opposite of the color of $c_J$ on that face for all faces, then the induce matching edges will be the same: a perfect matching edge with two incident faces colored ``red'' in $c_I$ will become an edge where the two incident faces are colored ``blue'' in $c_J$, and vice versa.  In this case, both $c_I$ and $c_J$ induce the same perfect matching, which explains the factor of two in the inequality.

Next, assume one face incident to an induced perfect matching edge is colored with the same color for both $c_I$ and $c_J$, e.g., $i_1=j_1$. Then the opposite face along the perfect matching edge must be colored the same for both $c_I$ and $c_J$, which implies the remaining faces incident to vertices of the perfect matching edge must be colored the opposite color. Continuing in this fashion, keeping the same color across each perfect matching edge, and switching colors along non-perfect matching edges, one gets $I= J$, which contradicts that they were chosen to be $c_I\not=c_J$.
\end{proof}
 
\begin{example}
In \Cref{ex:Theta2}, the Euler characteristic of the filtered $2$-color vertex homology of the $\theta$ graph was shown to be twelve, which is twice the number of 3-edge colorings.  The color basis and \Cref{thm:PMClasses} can be used to show that there are six harmonic colorings that generate $\widehat{\mathcal{VCH}}_2^0(\Gamma_\theta)$.  \Cref{fig:thetaColors} shows three such classes, while the other three are obtained by swapping $c_0$ and $c_1$.  Each coloring defines a perfect matching $M$ by choosing the edges where the incident circles have the same color (indicated by a star in \Cref{fig:thetaColors}).  Edges where incident circles differ in color define cycles of $\Gamma_\theta \setminus M$.  The remaining six harmonic colorings lie in the all-one vertex state.

\begin{figure}[H]
\includegraphics[scale=.6]{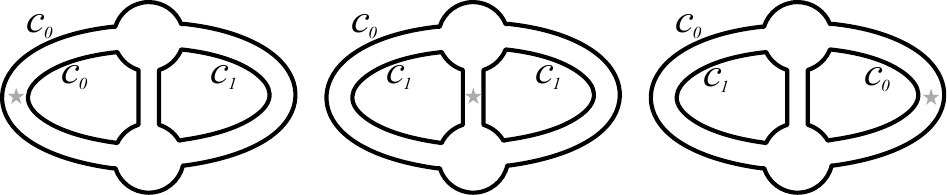}
\caption{Three colorings of the all-zero vertex smoothing.}\label{fig:thetaColors}
\end{figure}
\end{example}

\begin{example}
For the graph $\Gamma_{L3}$, it was already observed in \Cref{ex:3lolly} that the only maps in the hypercube of vertex states correspond to Case 3 and Case 7 in \Cref{fig:VertexCases}.  In particular, no circles are created or destroyed, and every state has exactly four circles just as shown in  \Cref{fig:3LollyZero}.  

\begin{figure}[H]
\includegraphics[scale=.45]{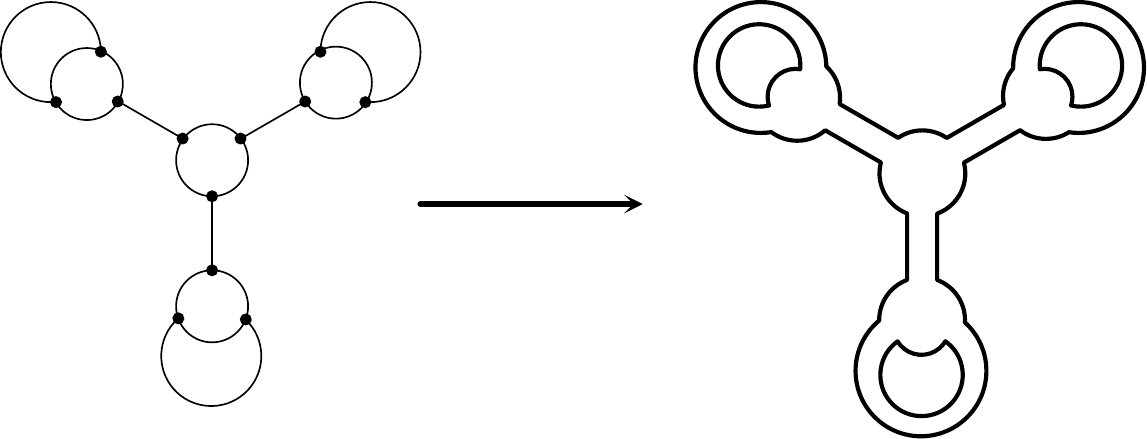}
\caption{The all-zero vertex state of $\Gamma_{L3}$.}\label{fig:3LollyZero}
\end{figure}

Note that the only states that support an element of the kernel of $\widehat{\delta}$ are states for which all outgoing edges of the hypercube of vertex states correspond to Case 3 (this happens when the central vertex has already been resolved).  However, each such element will lie in the image of an incoming Case 7 edge corresponding to changing the vertex smoothing at the central vertex.  Therefore, the filtered $2$-color vertex homology must be trivial.  Of course, considering \Cref{thm:PMClasses}, it had to be since there are no perfect matchings on $\Gamma_{L3}$.  
\end{example}

\Cref{thm:PMClasses} says that all nontrivial harmonic colorings represent a perfect matching of the graph when $n=2$, but not in a unique manner.  In fact, for a plane ribbon graph all possible perfect matchings are represented (twice) in homological grading zero, which is the content of Theorem~\ref{thm:perfectMatchings}:

\perfectMatchings*

\begin{proof}
The ``$\leq$'' part of this theorem was proven in \Cref{thm:PMClasses}.  All that is needed is to show every perfect matching corresponds to two harmonic colorings in $\widehat{\mathcal{VCH}}_2^{0}(\Gamma)$.  Choose a perfect matching $M$ for $\Gamma$ and temporarily delete the edges $M$ from $G$ to obtain a collection of cycles embedded in the plane.  The complement of those cycles is a collection of regions.  Assign colors red or blue to each region as follows:  The red (blue) regions are separated from infinity by an even (odd) number of cycles.  This coloring scheme induces a harmonic coloring on the all-zero vertex smoothing, which corresponds to the perfect matching $M$.

Finally, another harmonic coloring corresponding to the same perfect matching can be obtained by changing red to blue and vice-versa.  
\end{proof}

The proof of Theorem~\ref{thm:perfectMatchings} also showed the following:

\begin{scholium}
\label{sch:PMZero}
Every perfect matching for a plane graph is represented by a perfect matching face coloring in the zeroth filtered $2$-color vertex homology group.
\end{scholium}
 
In fact, the argument used in Theorem~\ref{thm:perfectMatchings} can be modified slightly to show that filtered $n$-color vertex homology is usually nontrivial, even for $n>2$.

\begin{figure}[H]
\includegraphics[scale=.4]{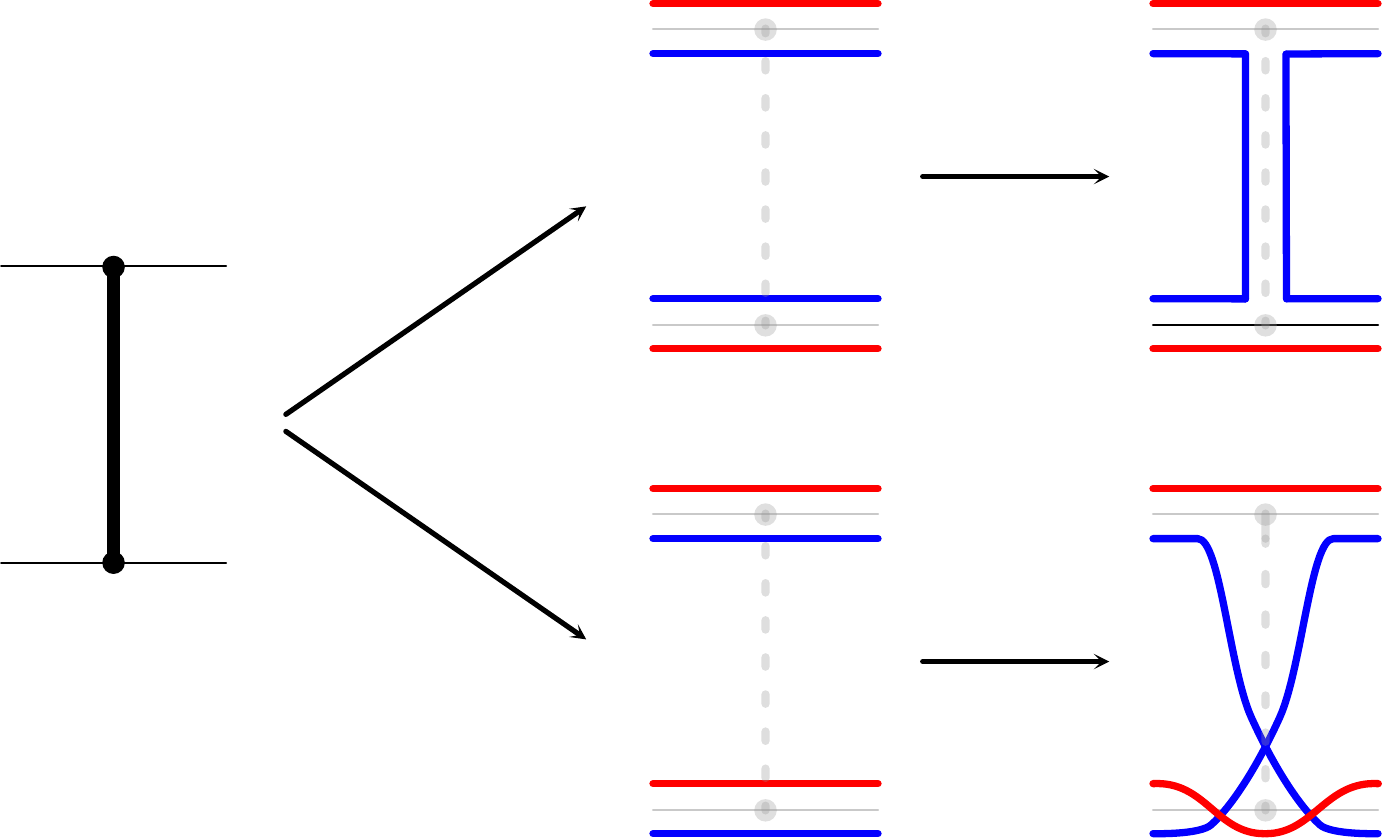}
\caption{Construction of a perfect matching face coloring for a given perfect matching.}\label{fig:PMConstruction}
\end{figure}

\begin{theorem}
Let $\Gamma$ be an oriented ribbon graph for an abstract trivalent graph $G(V,E)$ and let $M$ be a perfect matching for $G$.  Then, $M$ can be represented by a perfect matching face coloring in $\widehat{\mathcal{VCH}}_n^{i}(\Gamma)$ for some $i\in\{0,\ldots, |V|\}$ and for any $n\geq 2$.
\label{thm:nonTriv}
\end{theorem}
 
 \begin{proof}
Let $\Gamma$ be represented by an oriented ribbon diagram thought of as ribbons (edges) attached to disks (vertices). Note that $G\setminus M$ is a collection of cycles.  Delete the ribbons corresponding to the edges of $M$ from $\Gamma$.  The resulting ribbon graph, being a collection of annuli, can be properly $2$-colored.  Each ribbon corresponding to an edge of $M$ can then be reinserted to get partial $2$-face coloring at the vertices incident to the edge, after possibly inserting three half-twists around one of the two vertices, that is, a vertex $1$-smoothing (see \Cref{fig:PMConstruction}).  Let $i\in\{0,\ldots, |V|\}$ be the number of these vertex $1$-smoothings. Since this construction produces a coloring with the property that there are exactly two colors present at each vertex, the result is a harmonic coloring of a state in the hypercube of vertex states, which represents a perfect matching face coloring in $\widehat{\mathcal{VCH}}_n^{i}(\Gamma)$ for $n=2$, in fact, for any $n\geq 2$.   
 \end{proof}

Perfect matchings are not rare.  A theorem of Petersen \cite{Frink} says that any trivalent graph with at most two bridges has a perfect matching.  This condition is improved by a theorem of Tutte \cite{Tutte} that says that if every vertex subset $S\subset V$ is at least as large as the number of components of $G\setminus S$ having an odd number of vertices then the graph possesses a perfect matching.

\subsection{Symmetries of perfect matching face colorings}
Let $\Gamma$ be a ribbon graph for a graph $G(V,E)$. Suppose $\Gamma_\nu$ is a vertex state of the $\Gamma_\bullet$  hypercube that supports a perfect matching face coloring, $c_I\in \widehat{\mathcal{VCH}}_2(\Gamma_\nu)$, corresponding to a perfect matching $M\subset E$.  

\begin{figure}[H]
\includegraphics[scale=.4]{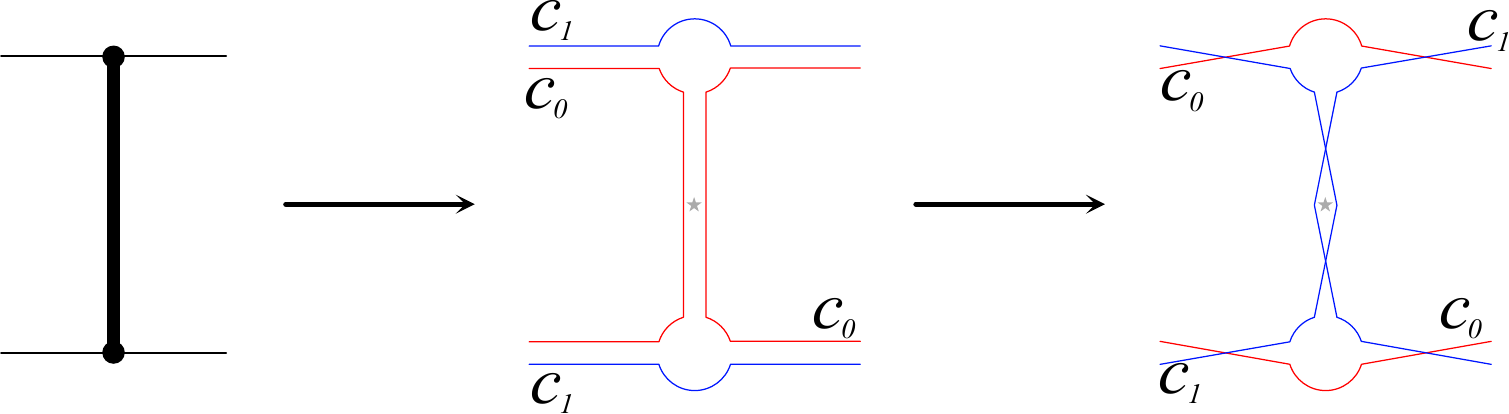}
\caption{Changing the vertex smoothings at a perfect matching edge.}\label{fig:PMSwitch}
\end{figure}

There is a symmetry that occurs in $c_I$ when vertex smoothings in $\Gamma_\nu$ are changed at each vertex of a perfect matching edge of $M$.  Switching the smoothings at each vertex (from a zero-smoothing to a one-smoothing and vice-versa) has the effect of cutting the ribbon surface at the four edges emanating from the perfect matching edge, inserting a half-twist of the ribbon on each, and then gluing it back (cf. \Cref{fig:PMSwitch}).  As shown in \Cref{fig:PMSwitch}, the resulting state can be colored in the same manner before and after the switch.  Since the switch can be performed at each perfect match edge independently, we obtain a family of $2^{\frac{1}{2}|V|}$  perfect matching face colorings in different vertex smoothings.  This number becomes the factor in the equation of Theorem~\ref{thm:oddMatchingsCancel}.

\begin{lemma}
\label{lem:MatchingSwitch}
Let $\Gamma$ be a ribbon graph for a trivalent graph $G(V,E)$ and let $M$ be a perfect matching for $G$.  Suppose $\Gamma_\nu$ is a state of the hypercube of vertex states for $\Gamma_\bullet$ that supports a perfect matching face coloring, $c_I\in \widehat{\mathcal{VCH}}_2(\Gamma_\nu)$, that corresponds to the perfect matching $M$.  Let $\Gamma_{\nu'}$ be the state obtained from $\Gamma_\nu$ by switching the vertex smoothings at each vertex of a perfect matching edge of $M$.  Then $\Gamma_{\nu'}$ also supports a perfect matching face coloring, $c'_J\in \widehat{\mathcal{VCH}}_2(\Gamma_{\nu'})$, corresponding to $M$.
\end{lemma}

The perfect matching switch of \Cref{lem:MatchingSwitch} involves locally switching the colors at the perfect matching edge from both-red to both-blue or vice-versa.  Away from the perfect matching edge, nothing else changes. Note that this switch only works for $n=2$ perfect matching face colorings.
\\

There is another operation one can do on a perfect matching face coloring of a vertex state $\Gamma_\nu$ to get a corresponding perfect matching face coloring on another vertex state $\Gamma_{\nu'}$.  Start with a perfect matching face coloring on $\Gamma_\nu$ and change the vertex smoothing at each vertex of a cycle of $G\setminus M$ (for example, see \Cref{fig:SaturnV}). The new state $\Gamma_{\nu'}$ has a perfect matching face coloring whose corresponding perfect matching is the same.  This is because  every perfect matching edge emanating from the chosen cycle had to have incident circles of the same color before the switch, and so even if the number of circles is different after the switch, we can still color them the same.  Moreover, there are two twists of the ribbon surface on every edge of the cycle, and so the remaining faces can be colored the same.   Since the switch can be performed at each cycle of $G\setminus M$ independently, we obtain a family of $2^{C}$ classes for the given perfect matching face coloring where $C$ is the number of cycles of $G\setminus M$.  Thus,

\begin{lemma}
\label{lem:cyclePairs}
Let $\Gamma$ be an  ribbon graph for a trivalent graph $G(V,E)$ and let $M$ be a perfect matching for $G$.  Suppose $\Gamma_\nu$ is a state of the hypercube of vertex states for $\Gamma_\bullet$ that supports a perfect matching face coloring, $c_I\in \widehat{\mathcal{VCH}}_2(\Gamma_\nu)$, that corresponds to the perfect matching $M$.  Let $\Gamma_{\nu'}$ be the state obtained from $\Gamma_\nu$ by switching the vertex smoothings at each vertex of a cycle of $G\setminus M$.  Then $\Gamma_{\nu'}$ supports a perfect matching face coloring, $c'_J\in \widehat{\mathcal{VCH}}_2(\Gamma_{\nu'})$, corresponding to $M$.
\end{lemma}

\begin{example}
Let $\Gamma_V$ be the plane graph shown in \Cref{fig:SaturnV} with perfect matching given by the thickened edges.  The middle picture in \Cref{fig:SaturnV} shows a perfect matching face coloring in $\widehat{\mathcal{VCH}}_2^0(\Gamma_V)$ that corresponds to an odd perfect matching with cycles of length three and five.  The picture on the right in that figure shows the result of performing a vertex 1-smoothing at every vertex of the 5-cycle.  Note that the coloring shown corresponds to the same odd perfect matching.

\begin{figure}[H]
\includegraphics[scale=.3]{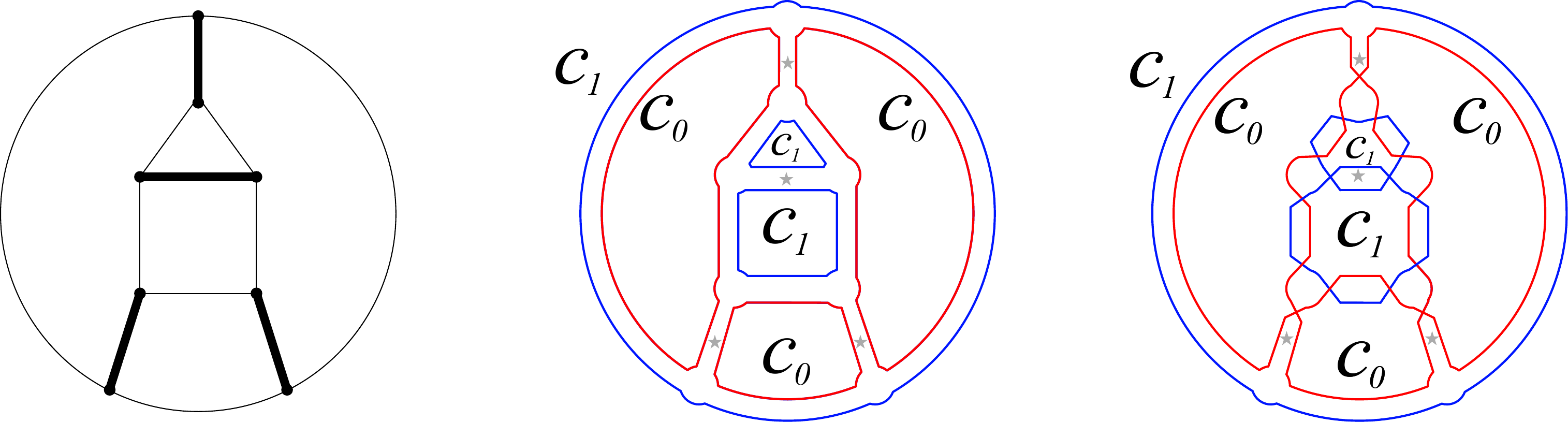}
\caption{The graph $\Gamma_V$ along with two colorings.}\label{fig:SaturnV}
\end{figure}

\end{example}

\begin{lemma}
\label{lem::oddMatchingsCancel}
Let $\Gamma$ be a ribbon graph for a trivalent graph $G(V,E)$.  Only even perfect matchings contribute to the Euler characteristic of the filtered $2$-color vertex homology $\widehat{VCH}_2^*(\Gamma)$.
\end{lemma}

\begin{proof}
By \Cref{lem:cyclePairs}, any perfect matching face coloring $c_I$ that corresponds to an odd perfect matching, $M$,   is canceled by a perfect matching face coloring $c_I'$ for the same perfect matching an odd number of columns away. This forms a sub-hypercube with $2^{C'}$  states where $C'$ is the number of odd cycles of $G\setminus M$. Therefore, the  signed sum of these perfect matching face colorings is zero. 
\end{proof}

There are non-planar examples of graphs, like  $K_{3,3}$ in \Cref{ex:K33Filtered} (see also \Cref{ex:K33Computation}), where even perfect matchings do not contribute to the Euler characteristic. But if the Euler characteristic is nonzero, then \Cref{lem::oddMatchingsCancel} implies that the count is coming from some set of even perfect matchings on $G$.

\begin{lemma}
\label{lem:EvenPMEvenDegree}
Let $\Gamma$ be a plane graph for a planar trivalent graph $G(V,E)$.  Perfect matching face colorings that correspond to even perfect matchings of $\Gamma$ only occur in even homological degrees, i.e., if $c_I \in \widehat{VCH}^i_2(\Gamma)$ corresponds to an even perfect matching, then $i\in \BZ$ is even. 
\end{lemma}

\begin{proof}
Suppose $\Gamma_\nu$ is a vertex state that supports a perfect matching face coloring that corresponds to an even perfect matching $M$, with faces of the state graph $\Gamma_\nu$ colored ``red'' and ``blue.''  Recall from Definition~3.22 in \cite{BaldCohomology} (see also Definition~1 in \cite{BLM}) that a two-factor that factors through the perfect matching is a set of cycles in $G$ in which the edges of each cycle alternate between edges of $M$ and edges of $G\setminus M$, and every edge in $M$ is contained in some cycle (see \Cref{fig:FactorsThroughM}).  
\begin{figure}[H]
\includegraphics[scale=1]{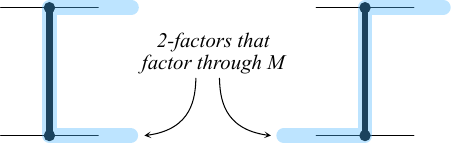}
\caption{Local pictures of a $2$-factor that factors through a perfect matching.}\label{fig:FactorsThroughM}
\end{figure}
By Proposition 3 of \cite{BLM}, an even perfect matching has $2^C$ such two-factors where $C$ is the number of cycles of $G\setminus M$. Choose one.  Construct a family of circles corresponding to the chosen two-factor by doing the following:
\begin{enumerate}
\item Beginning at the red side of an edge of $G\setminus M$ that is contained in some cycle of your chosen two-factor, follow it until you reach a perfect matching edge.
\item Follow the perfect matching edge to the red side at the other end of the perfect matching edge.
\item Continue in this way until the loop closes by following your chosen cycle in the $2$-factor.
\item Repeat for all other cycles in the two-factor.
\end{enumerate}

\begin{figure}[H]
\includegraphics[scale=1]{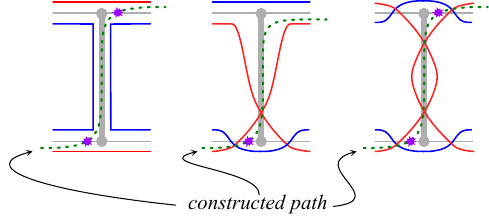}
\caption{Local picture of the path constructed above, at a perfect matching edge.  Intersections of the path with the cycles of $G\setminus M$ are shown in purple.}\label{fig:constructedPath}
\end{figure}

After doing this, count the number of intersections of the constructed family of circles with $G\setminus M$.  The parity of the number of intersections of this family of circles with the cycles of $G\setminus M$ matches the parity of the number of vertex 1-smoothings used to construct $\Gamma_\nu$.  In \Cref{fig:constructedPath}, local pictures of the three cases are drawn.  In each case, the intersection (marked by a purple star) has the same parity as the number of vertex $1$-smoothings.  By the Jordan curve theorem, the number of intersections is even, which implies that $|\nu|$ is even.
\end{proof}

We are now ready to give meaning to the filtered $2$-color vertex homology, and prove Theorem~\ref{thm:oddMatchingsCancel}, restated below, by observing that its nontrivial classes correspond to perfect matchings of the graph, and that in the Euler characteristic, only even perfect matchings are counted.  Since even perfect matchings correspond to $3$-edge colorings, the Euler characteristic provides a way to count the number of 3-edge-colorings of the graph. 

\oddMatchingsCancel*

\begin{proof}
First, observe that an even perfect matching has $2^{\mbox{\# of cycles of $G\setminus M$}}$ 3-edge-colorings that fix the color of the edges of $M$.  Also, by \Cref{lem::oddMatchingsCancel}, odd perfect matchings do not contribute to the Euler characteristic.  

By \Cref{lem:EvenPMEvenDegree}, every even perfect matching occurs only in even homological degree.  Finally,  \Cref{lem:cyclePairs,lem:MatchingSwitch} demonstrate that such matchings occur with the appropriate multiplicities.
\end{proof}

 \Cref{lem:MatchingSwitch} and \Cref{lem:cyclePairs} show that performing a perfect matching switch or cycle switch for a perfect matching face coloring on the all-zero vertex state will produce perfect matching face colorings on state graphs $\Gamma_\nu$ with $|\nu|\geq 2$.  Thus, one may conjecture that the $2$-color vertex homology vanishes in grading one for planar graphs.  The next theorem shows that this is not true.

\begin{theorem}
Let $\Gamma$ be a plane ribbon graph of a connected, trivalent graph $G(V,E)$ with $b$ bridges and $m$ perfect matchings.  Then
$$\dim \widehat{VCH}_2^1(\Gamma) = 4mb.$$
\label{thm:degreeOne}
\end{theorem}

\begin{proof} Since the theorem is true for $m=0$ by \Cref{lem:harmonicsToColors}, assume $m>0$. Recall that a connected trivalent graph $G$ with a bridge must include that bridge in every perfect matching of $G$ (cf. Lemma~2.2 in \cite{BaldCohomology}). By Theorem~\ref{thm:perfectMatchings}, $\dim \widehat{VCH}_2^0(\Gamma) = 2m$, and is generated by harmonic colorings (perfect matching face colorings) that correspond to perfect matchings---two for each perfect matching.  For each bridge, \Cref{fig:1smoothFlip} shows that, after performing a 1-smoothing on one of the two vertices incident to the bridge, the resulting state graph represents a ribbon graph with the same perfect matching face colorings as the all-zero state of $\Gamma$.  Thus, there exist two harmonic colorings in homological degree one for each harmonic coloring in homological degree zero, one for each vertex.  Finally, all harmonic colorings in homological degree one are of this type since  performing a 1-smoothing at any vertex that is not incident to a bridge edge results in a single loop at that vertex (recall that for a planar graph, such a vertex is modeled by \Cref{fig:Case1}).  

\begin{figure}[H]
\includegraphics[scale=.4]{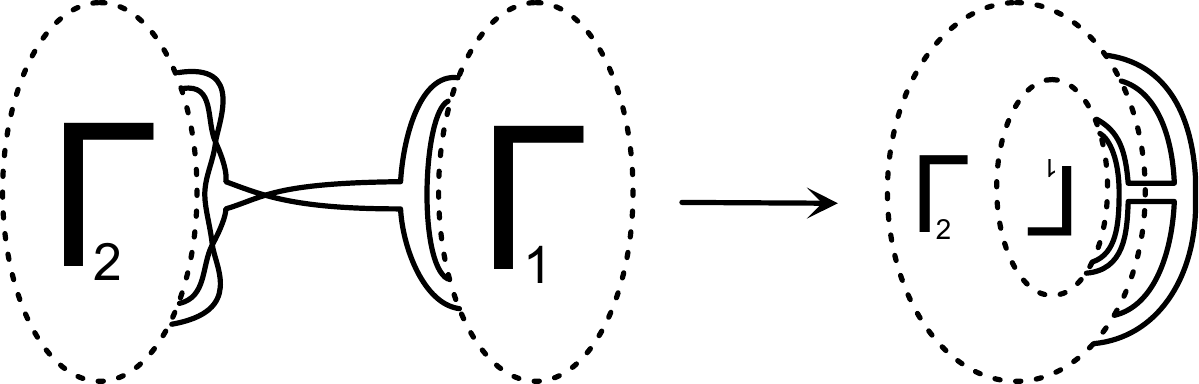}
\caption{Equivalent surfaces before and after flip moves.}\label{fig:1smoothFlip}
\end{figure}

\end{proof}

\begin{remark} One consequence of \Cref{thm:degreeOne} is that for any plane graph with no bridges, the filtered $2$-color vertex homology {\em is} trivial in homological degree one.
\end{remark}

\section{The Vertex and Total Matching Polynomials}\label{sec:TotalMatching}
In this section, we explore a few results about the vertex polynomial and define a new polynomial, called the total matching polynomial.  

\subsection{The vertex polynomial}\label{sec:vertexPoly}
Recall from \Cref{defn:VertexPenrose} that the vertex polynomial is characterized by the following vertex bracket and loop value for any $n\in\ZZ$:
\begin{eqnarray}\label{eq:vertex-poly-bracket}
V\left( \vertexbracketvertex \right) &=&  V\left( \vertexbracketzero \right) \ - \ V\left(\vertexbracketone \right), \mbox{ \ and} \\ [.2cm]
\label{eq:vertex-poly-loop-value}  V\left( \bigcirc  \right)& = & n.
\end{eqnarray}

The vertex polynomial and filtered $n$-color vertex homology are eminently encodable on a computer (see \Cref{app:computations} for  code of the vertex polynomial). In \Cref{ex:dodecahedron}, the  polynomial  of the dodecahedron is computed, $V(Dodec,n) = 2 (n+1) n^2 (n-1) (240 - 116 n^2 + 114 n^4 + 11 n^6 + n^8)$, which together with Theorem~\ref{thm:oddMatchingsCancel} implies that $$\chi\left(\widehat{VCH}^*_2(Dodec)\right) = V(Dodec,2) = 2^{10} \cdot 60= 2^{10}\cdot \#\{\mbox{Tait colorings of the dodecahedron}\}.$$ Compare this calculation to  papers on the computation of (or bounds on) the number of Tait colorings of the dodecahedron using  Kronheimer and Mrowka's  instanton homology (cf. \cite{KM3, Boozer, KhoRob}). 

The vertex polynomial has a particularly interesting ``history that never was,'' but a history nonetheless. It is a polynomial that could have been defined fifty years ago and studied as intensely as the Penrose polynomial. We briefly explain why. In Roger Penrose's seminal 1971 paper \cite{Penrose}, he defined a number $K$ for a plane graph and showed it was equal to the number of $3$-edge colorings (Tait colorings) of that graph.  This number was based upon computations in a graphical calculus for an abstract tensor system of dimension three, i.e., a calculus with loop value~$3$. As a side note, TQFTs, graphical calculuses in representation theory, and quantum invariants of manifolds (plus many other fields) can find much of their modern origins in Penrose's paper. 

Penrose then related the computation of $K$ to a negative dimensional abstract tensor system, which he called a {\em binor system}.  This system was for a ``vector space'' with dimension $-2$, i.e., a graphical calculus with loop value $-2$. He labeled such negative dimensional tensor systems as ``unreasonable.''  Later mathematicians would find these unreasonable systems completely reasonable! For example, Louis Kauffman's original formulation of the Kauffman bracket for the Jones polynomial has loop polynomial $\langle \bigcirc\rangle= -A^2-A^{-2}$ with a loop value of $-2$.  More generally, it gives rise to the study of Jones-Wenzl projectors and Chebyshev polynomials (defined recursively by $\Delta_0=1$, $\Delta_1=d$, $\Delta_{n+1}= d\cdot \Delta_n-\Delta_{n-1}$)  used in Temperley-Lieb algebras and  quantum invariants of $3$-manifolds (cf. \cite{Kauffman-Lins}). 

In this binor system, the computation of $K$ uses a similar graphical calculus, but replaces the loop value of $3$ with $-2$. Also, the computation is carried out on the blowup of the plane graph $\Gamma$ instead. This lead Penrose to generalize these computations to the {\em Penrose polynomial}, $P(\Gamma,n)$, which allows for the  loop value to be any $n\in \ZZ$. He then went on to show that  
$$P(\Gamma,3)=K=\left(-\frac{1}{4}\right)^{\frac12 |V|} P(\Gamma,-2)$$ 
for a trivalent plane graph $\Gamma$. He saw that the evaluation of $P(\Gamma,n)$ for $n$ had meaning in that it counted colorings of collections of circuits of $\Gamma$ with $n$ colors when $n$ was a positive integer.  The introduction of the Penrose polynomial kicked off 50 years of research into the polynomial (see \cite{Jaeger,Aigner,EMM,EMMKM,Moffat2013} for a few examples) that recently culminated in a description of the Penrose polynomial using proper $n$-face colorings of CW complexes of smooth surfaces in \cite{ColorHomology}. 

Penrose went further in his paper in relating $K$ to evaluations of other graphical calculuses: Using the binor system, he  established a formula for calculating $P(\Gamma,-2)$ using the vertex bracket of Equation~\ref{eq:vertex-poly-bracket}) with  loop value of $-2$ through a binor relation described at the bottom of page 239 of \cite{Penrose}.  In this setup, he showed that, if $K'$ was the evaluation of the vertex bracket for loop value $n=-2$, then $K = (-\frac12)^{\frac12|V|}K'$. Furthermore, if the  loop value was changed to $n=2$ instead (page 240), he showed that the resulting evaluation $K''$ would satisfy $K'=\pm K''$, with the signed determined by the number of vertices of $\Gamma$.

At this point in 1971, the equality up to sign of the evaluations of $K'$ and $K''$ should have triggered the math community to  generalize these computations to a polynomial for any $n\in \ZZ$, just like what happened for the Penrose polynomial.  This generalization, in 2024, is the vertex polynomial $V(\Gamma,n)$ defined in \Cref{defn:VertexPenrose}.  We suspect that the reasons for this oversight are three-fold: (1) Penrose did not {\em explicitly} define the vertex polynomial like he did for the Penrose polynomial, probably because (2) it was not clear what evaluating the vertex polynomial for $n>2$ meant,  and (3) maybe it was assumed that, since 
\begin{equation}\label{eq:equalities-of-vertex-and-Penrose} 
2^{\frac12 |V|} \cdot K \ = \ V(\Gamma, 2) \ = \  \left(-1\right)^{\frac12|V|} V(\Gamma,-2) \ = \ \left(-\frac12\right)^{\frac12 |V|}P(\Gamma, -2)
\end{equation}
for plane graphs $\Gamma$, the vertex polynomial contained essentially the same information as the Penrose polynomial. 

The theorems and propositions in this paper reveal that the third reason could not be farther from the truth.  Furthermore, the next theorem establishes what the vertex polynomial counts when evaluated at $n>2$:

\begin{theorem}\label{theorem:partial-coloring-ribbon-states}  Let $\Gamma$ be a ribbon graph of a connected trivalent graph $G(V,E)$. Let $\Gamma_\nu$ be the state graph that corresponds to the vertex state $\nu\in\{0,1\}^{|V|}$ in the $\Gamma_\bullet$ hypercube of $\Gamma$. Then the vertex polynomial of $\Gamma$, evaluated for any $n\in\NN$, is
$$V(\Gamma,n)  = \sum_{\mbox{$|\nu|$ even}}\#\left\{\mbox{partial $n$-face colorings of $\Gamma_\nu$}\right\} -  \sum_{\mbox{$|\nu|$ odd}}\#\left\{\mbox{partial $n$-face colorings of $\Gamma_\nu$}\right\}.$$
\end{theorem}

\begin{proof}
This theorem follows from a chain of results in this paper:  \Cref{lem:harmonicsToColors} establishes that the partial $n$-face colorings of that state are in one-to-one correspondence with the harmonic colorings of a state. \Cref{prop:direct-sum-equals-all-harmonics} and the isomorphism following it imply that the Euler characteristic of the filtered $n$-color vertex homology is equal to the alternating sum of harmonic colorings of states based upon the state's parity.  Next, the Euler characteristic of the filtered $n$-color vertex homology is equal to the graded Euler characteristic of the bigraded $n$-color vertex homology evaluated at $q=1$ by \Cref{theorem:spectral-sequence-converges}. This Euler characteristic is equal to the $n$-color polynomial evaluated at $q=1$ by Theorem~\ref{thm:gradedEuler}. Finally, the vertex polynomial evaluated at $n$ is equal to this Euler characteristic by \Cref{def:vertexPoly} and \Cref{defn:VertexPenrose}.
\end{proof}

We are now ready to prove Theorem~\ref{thm:evenOddRibbonGraphs} as an application of \Cref{theorem:partial-coloring-ribbon-states}. 

\begin{proof}[Proof of Theorem~\ref{thm:evenOddRibbonGraphs}] Let $\Gamma$ be a ribbon diagram of an oriented ribbon graph of a connected trivalent graph $G(V,E)$ with trivial automorphism group.  By Theorem~6.17 of \cite{ColorHomology}, when the automorphism group of $G$ is trivial, the hypercube of states of $\Gamma^\flat$ contains {\em all} ribbon graphs of $G$; each corresponding state graph in this hypercube is {\em distinct}.  Furthermore, from the proof of part (1) of Theorem~6.17 of \cite{ColorHomology}, all oriented ribbon graphs appear in the hypercube of states of $\Gamma^\flat$ as a collection of vertex $1$-smoothings on the all-zero state $\Gamma_{\vec{0}}$ of $\Gamma$ (\Cref{def:all-zero-state}). See Figure 12 in that paper, for example. Hence, the set of distinct oriented ribbon graphs of $G$ is a subset $S$ of the hypercube of states of $\Gamma^\flat$, each determined by performing vertex 1-smoothings to a set of vertices of $\Gamma$ on $\Gamma_{\vec{0}}$. 

There is a $2$-to-$1$ map from the hypercube of vertex states $\Gamma_\bullet$ to this set $S$. Given a  state graph $\Gamma_\nu$ in the $\Gamma_\bullet$ hypercube, there is an equivalent state graph $\Gamma_{\nu'}$ of $\Gamma_\bullet$ found by switching {\em all} vertex $0$-smoothings of $\Gamma_\nu$ to vertex $1$-smoothings, and switching all vertex $1$-smoothings to vertex $0$-smoothings. This new state graph $\Gamma_{\nu'}$ is equivalent to the original state graph $\Gamma_\nu$ as ribbon graphs. For example, if one applies this operation to the vertex all-zero state, one gets the vertex all-one state, which is the same ribbon graph as the vertex all-zero state but with two half-twists put into each band (see \Cref{fig:vertex-state-of-theta} for example). 

The distinctness of the ribbon graphs in $S$ together with the $2$-to-$1$ map explains the distinctness and factor of two in the righthand side of the equation in Theorem~\ref{thm:evenOddRibbonGraphs}.  Finally, since performing a vertex 1-smoothing places a half-twist into each of the three bands at a vertex, the parity of a vertex state $\Gamma_\nu$ with respect to $\Gamma$ is the same as the parity of $|\nu|=\nu_1+\nu_2+\dots+\nu_{|V|}$.

Finally, the proof that the vertex polynomial is an abstract graph invariant follows from the lemma and definition below.
\end{proof}

Let $G(V,E)$ be a connected trivalent ribbon graph and $\Gamma$ and $\Gamma'$ be two ribbon graphs of it. Define the ribbon graph $\Gamma$ to be {\em even with respect to $\Gamma'$} if the number of half-twists that need to be inserted into the bands of $\Gamma$ to make it equivalent to $\Gamma'$ as ribbon graphs is even. This defines an equivalence class by part (1) of Theorem~6.17 of \cite{ColorHomology}. For a discussion of this concept, see Definition~7.1 in \cite{ColorHomology} and the text that follows it.  

\begin{lemma}
Let $G(V,E)$ be a connected trivalent graph and let $\Gamma_1, \Gamma_2$ be any two oriented ribbon diagrams of it.  The vertex polynomial is an invariant of oriented ribbon structures up to sign, that is,
$$V(\Gamma_1,n) = \pm V(\Gamma_2,n)$$
with the sign is positive if they are even with respect to each other and negative if they are odd.
\label{thm:VInvt-of-oriented-ribbon-graph}
\end{lemma}

The orientability condition is important. An orientable ribbon graph and a nonorientable ribbon graph for the same graph may have different vertex polynomials. For instance, compare \Cref{ex:theta} versus \Cref{ex:VPDThetaNeg}. 
 
\begin{proof}
This lemma follows from ideas presented in the proof of Theorem~\ref{thm:evenOddRibbonGraphs}.  Observe that, even if the automorphism group of $G$ is not trivial, the proof of part (1) of Theorem~6.17 of \cite{ColorHomology} still implies that all oriented ribbon graphs of $G$ are contained within the hypercube of vertex states of $\Gamma_1$.  Hence, $\Gamma_2$ is some vertex state $\Gamma_\nu$ of that hypercube.  \Cref{theorem:partial-coloring-ribbon-states} shows that   ${V(\Gamma_1,n) = (-1)^{|\nu|} V(\Gamma_\nu,n)}$.
\end{proof}

Therefore, the vertex polynomial can be defined as an abstract graph invariant if an oriented ribbon graph is always chosen to define it. There is still a sign issue in the polynomial itself, but this can be resolved by forcing the vertex polynomial to be positive when evaluated at large values of $n\in\NN$.  This is done so that the vertex polynomial of bridgeless planar graphs is positive when evaluated at $n=2$. An oriented ribbon graph of a connected trivalent graph is called {\em positive}, {\em negative}, or {\em zero} if the leading coefficient of its vertex polynomial is positive, negative, or zero.  

\begin{definition} Let $G(V,E)$ be a connected trivalent graph.  The {\em vertex polynomial of $G$} is defined to be $$V(G,n) := V(\Gamma,n)$$ 
for any  nonnegative oriented ribbon graph $\Gamma$ of $G$. \label{definition:abstract-vertex-poly}
\end{definition}

\Cref{thm:VInvt-of-oriented-ribbon-graph} together with \Cref{definition:abstract-vertex-poly} completes the proof of  Theorem~\ref{thm:evenOddRibbonGraphs}.

The remainder of this subsection describes some of the properties of the vertex polynomial. The first property is that the vertex polynomial is an even or odd function for oriented ribbon graphs. Not only does this property generalize Penrose's equality $V(\Gamma,2) = (-1)^{\frac12|V|}V(\Gamma,2)$  in \Cref{eq:equalities-of-vertex-and-Penrose} to all $n\in\ZZ$, but it opens the door to categorifications of Penrose's binor systems (negative dimensional tensor systems) by working with positive dimensional systems instead. One can think of Penrose's transformation from a negative dimensional tensor system to a positive one as similar to the transformation of the Jones polynomial defined with Kauffman's loop value of $d=-A^2-A^{-2}$ to the Jones polynomial defined with loop value $d=q+q^{-1}$, which was useful in discovering Khovanov homology.  In fact, the symmetry in $n$ of $V(G,n)$ played a large part in our discovery of the homology theories of this paper.

\begin{theorem}\label{thm:VertexParity}
Let $\Gamma$ be an orientable ribbon graph for a connected trivalent graph $G(V,E)$.  Then the vertex polynomial $V(\Gamma,n)$ is either an even or odd function in the variable $n$.  Furthermore, it is even precisely when $\frac12 |V|$ is even, and it is odd when $\frac12 |V|$ is odd.
\end{theorem}

\begin{proof}
For an orientable ribbon graph the only possible maps (as shown in \Cref{fig:VertexCases}) in the vertex hypercube $\Gamma_\bullet$ are Configurations 1, 3, 6, and 7.  Thus, if $k_\nu$ is the number of circles in the state $\Gamma_\nu$, then  $k_\nu$ has the same parity as the number of circles in the all-zero state for every state $\Gamma_\nu$.  Since the vertex polynomial is, by definition, a sum of terms of the form $n^{k_\nu}$, it must be even or odd.

For the second claim, observe that the number of edges in the graph is $\frac32 |V|$.  Hence, if the all-zero state corresponds to a genus $g$ surface, then we can write the number of faces, $F$, as
$$F = 2 - 2g - |V| + 3\left(\frac12 |V|\right).$$
Since $|V|$ is even, the parity of the number of faces, which correspond to the number of circles in the all-zero state, is determined by $\frac12 |V|$. 
\end{proof}

\begin{remark}
Application of Theorem~\ref{thm:oddMatchingsCancel} requires the graph to be planar.  However, one should note that the parity result for the vertex polynomial described in \Cref{thm:VertexParity} holds for any orientable ribbon graph.  For non-orientable ribbon graphs $\Gamma$, $V(\Gamma,n)$ may not be an even or odd function (cf. \Cref{ex:VPDThetaNeg}).
\end{remark}

The vertex polynomial is well-behaved when blowing up the graph at a vertex, as the following theorem shows.

\begin{theorem}\label{thm:VertexPolyBlowUp}
Let $\Gamma_1$ be a ribbon graph for a trivalent graph $G(V,E)$, and $\Gamma_2$ be the ribbon graph obtained by blowing up $\Gamma_1$ at a single vertex.  Then 
$$V(\Gamma_2,n) = n\cdot V(\Gamma_1,n).$$
\end{theorem}

\begin{proof}
For each state graph of $\Gamma_1$ there is a corresponding state graph for $\Gamma_2$ with one extra circle at the blown up vertex,  as shown in \Cref{fig:Switch2BlowUp}. This extra circle results in each of those states contributing an extra factor of $n$.
\begin{figure}[H]
\includegraphics[scale=.4]{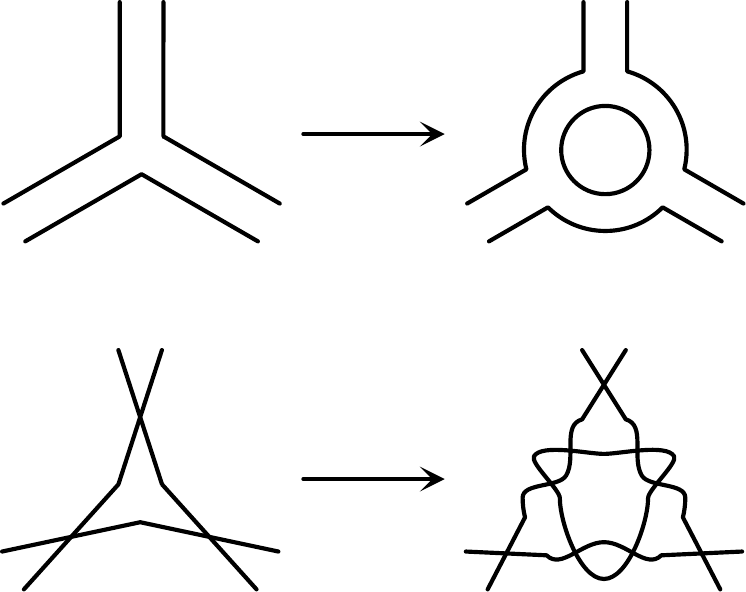}
\caption{The effect of blowing up at a vertex.}\label{fig:Switch2BlowUp}
\end{figure}
The cube of vertex resolutions for $\Gamma_2$, however, contains states which have a single vertex 1-smoothing at one of the vertices of the blowup cycle, as pictured on the left of \Cref{fig:SwitchAtBlowUp}.  By \Cref{lem:cyclePairs}, performing a cycle switch along  the blowup cycle  results in the configuration shown on the right side of \Cref{fig:SwitchAtBlowUp}.  Since the number of circles on the left side of \Cref{fig:SwitchAtBlowUp} is equal to the number of circles on the right, and because the two states corresponding to them differ in the parity of the number of vertex 1-smoothings, the terms corresponding to these states  do not contribute to the vertex polynomial.  

\begin{figure}[H]
\includegraphics[scale=.4]{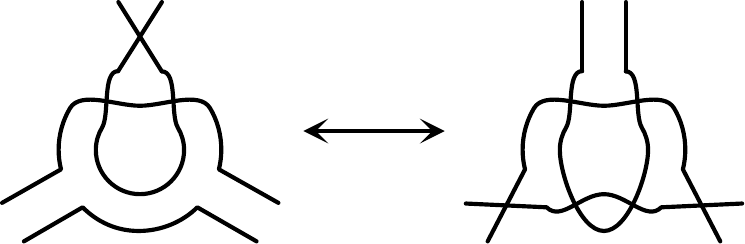}
\caption{Cycle switch on a blowup cycle.}\label{fig:SwitchAtBlowUp}
\end{figure}

\end{proof}

\Cref{thm:VertexPolyBlowUp} implies that to calculate the vertex polynomial of any ribbon graph, one may first blow-down any triangular face (i.e., collapse the $3$-cycle boundary of the face to a single vertex).  The vertex polynomial of the result, together with the number of triangular faces blown down determines the vertex polynomial of the original.  Thus:

\begin{corollary}\label{cor:VertexPolyBlowUp}
Let $\Gamma$ be a ribbon graph for a trivalent graph $G(V,E)$.  Suppose $\Gamma$ has $\tau$ triangular faces and $\Gamma'$ is the ribbon graph obtained from $\Gamma$ by blowing down every triangular face.  Then,
$$V(\Gamma',n) = n^\tau V(\Gamma,n).$$
\end{corollary}

Examples to illustrate \Cref{thm:VertexPolyBlowUp} and \Cref{cor:VertexPolyBlowUp} are included in \Cref{app:computations} along with Mathematica code that may be used for computation of the polynomials.  

To illustrate one advantage that the filtered $n$-color vertex homology has over the vertex polynomial, we include the following example.

\Cref{thm:VertexPolyBlowUp} allows us to generalize the vertex polynomial to {\em any} abstract graph $G$ with vertices of different valences.

\begin{definition} \label{definition:any-valence-abstract-vertex-poly} Let $G(V,E)$ be a connected graph and $\Gamma$ an oriented ribbon graph of it such that the blowup, $\Gamma^\flat$, is nonnegative in the sense that $V(\Gamma^\flat,n) \geq 0$ for large $n\in\NN$. Define the {\em vertex polynomial of $G$} to be
$$V(G,n):= \left(\frac{1}{n}\right)^{|V|} \cdot V(\Gamma^\flat,n).$$
\end{definition}

Note that when $G$ is trivalent, this definition will produce the same polynomial as in \Cref{definition:abstract-vertex-poly}. It is not clear how or if this polynomial for non-trivalent graphs can be categorified. However, in \Cref{rem:generalize}, we briefly described how to modify the definition of bigraded $n$-color vertex homology to generalize it to $r$-regular graphs, which {\em can} be thought of as a categorification of a different generalization of the vertex polynomial. \Cref{sec:4valent} on $4$-regular graphs shows that this polynomial is different than $V(G,n)$.

Due to the factor in the definition, there is an interesting question that is yet to be resolved:

\begin{question}
Is the vertex polynomial of an abstract graph always a polynomial, or can it be a Laurent polynomial?
\end{question}

All examples we have tried it on (using the Mathematica code in \Cref{app:computations}) suggest that it is always a polynomial. 

Finally, we present the vertex polynomial of a nonplanar graph and analyze its relationship to the filtered $n$-color vertex homology.

\begin{example}
\label{ex:K33Filtered}
For the graph $K_{3,3}$ shown in \Cref{fig:K33}, the filtered $n$-color vertex homology can be calculated by hand.  Given the symmetry in the hypercube of vertex states, only the number of partial $n$-face colorings for $32$ vertex states need to be calculated (the other $32$ states will have corresponding numbers to those states).  For example, in homological degree one there are six vertex states to consider.  Two of these states consist of a single circle; hence they do not support partial $n$-face colorings.  The remaining four, up to symmetry, are represented by the state shown in \Cref{fig:K33}.  To obtain a partial $n$-face coloring on such a state we observe that once a color is chosen for the blue circle shown ($n$ choices), the red circle must be colored differently to avoid a single-color vertex ($n-1$ choices).  The final circle, shown in black, may be colored red, but not blue ($n-1$ choices).  Thus, the number of partial $n$-face colorings on such a state is $n(n-1)^2$.  Performing a similar calculation for the entire cube of vertex resolutions, the following ranks for the filtered $2$-color vertex homology are given by $\{2,8,22,32,22,8,2\}$. In general, they are
$$\bigg\{n(n-1)^2, 4n(n-1)^2,n(7-16n+9n^2)), 4n(2-5n+3n^2),n(7-16n+9n^2),4n(n-1)^2,n(n-1)^2\bigg\}.$$  
For every $n$, observe that the Euler characteristic of this homology is zero, and in fact, $V(K_{3,3},n) = 0$ for all $n$ (cf. \Cref{app:computations} for Mathematica code that can also be used to verify this computation).  Hence, while the vertex polynomial is zero, the homology theories of this paper are not. In particular, the vertex polynomial evaluated at $n=2$ does not count the number of $3$-edge colorings of $K_{3,3}$. This is because $K_{3,3}$ is nonplanar for which Theorem~\ref{thm:oddMatchingsCancel} does not apply. However, there is still valuable information in the ranks of the filtered $2$-color vertex homology.

\begin{figure}[H]
\includegraphics[scale=.3]{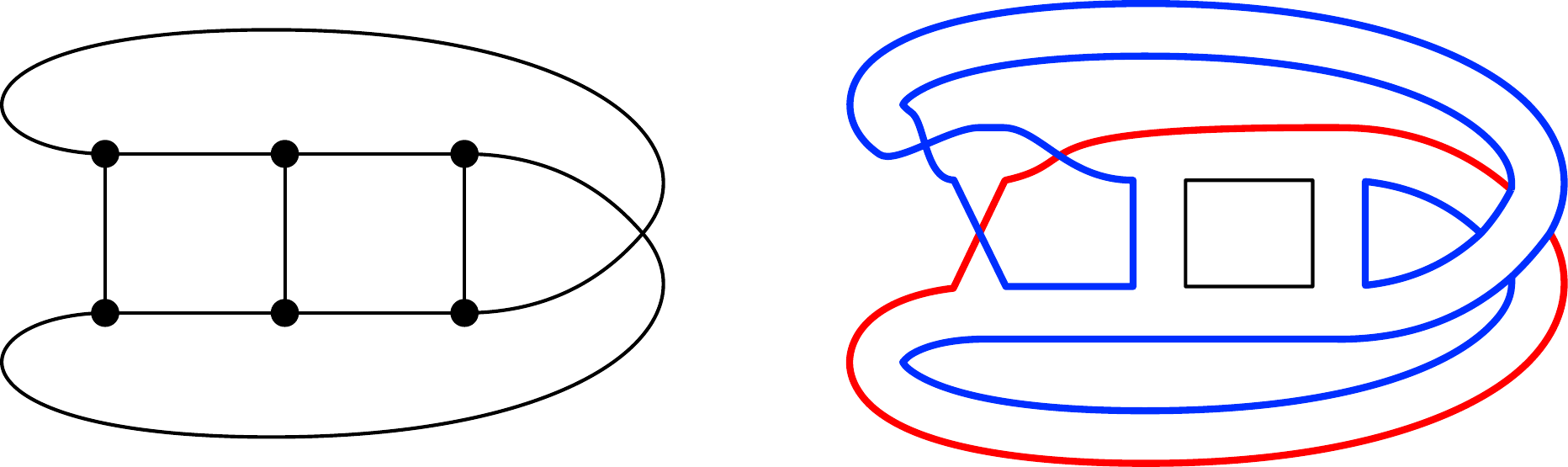}
\caption{The graph $K_{3,3}$ and a state with a single vertex 1-smoothing.}\label{fig:K33}
\end{figure}
\end{example}

If one sums the ranks of the homology groups $\{2,8,22,32,22,8,2\}$, the result is $96$.  This is equal to $2^{\frac12 |V|} \cdot 12$.  Since there are twelve $3$-edge colorings of $K_{3,3}$, one might guess that the Poincar\'{e} polynomial of the filtered $2$-color homology counts $3$-edge colorings.  This conjecture does not hold in general if the graph possesses odd perfect matchings ($K_{3,3}$ has only even perfect matchings).  In the general case, the sum of the ranks of the filtered $2$-color homology will be greater than or equal to $2^{\frac12 |V|} \cdot \# \{ \text{3-edge-colorings}\}$, with equality if the graph possesses only even perfect matchings. We explore this idea a bit more below.\\

\subsection{The total matching polynomial} 
In addition to the vertex polynomial, which is the Euler characteristic of the filtered $n$-color vertex homology, we obtain an abstract graph invariant by taking the Poincar\'{e} polynomial instead.  

\begin{definition} Let $G(V,E)$ be a connected trivalent graph and let $\Gamma$ be any ribbon diagram of it. The Poincar\'{e} polynomials of the filtered $n$-color vertex homologies generate the {\em $2$-variable total matching polynomial}, which is characterized by
$$TM(\Gamma,n,t) := \sum_{|\nu|=i} t^i \dim \widehat{\mathcal{VCH}}_n(\Gamma_\nu).$$
The {\em total matching polynomial of $\Gamma$} is $TM(\Gamma,n):=TM(\Gamma,n,1)$.  \label{definition:totalfacecolorpolynomial}
\end{definition}

\Cref{lem:harmonicsToColors} says that the filtered $n$-color vertex homology is generated by partial $n$-face colorings.  Theorem~6.17 and Remark~7.5, both from \cite{ColorHomology},  imply that the total matching polynomial does not depend on the oriented ribbon structure.

\begin{theorem}
Let $G(V,E)$ be a connected trivalent graph and let $\Gamma_1, \Gamma_2$ be any two oriented ribbon diagrams of it.  The total matching polynomial is invariant of the oriented ribbon structures, that is,
$$TM(\Gamma_1,n) = TM(\Gamma_2,n).$$
Consequently, define $TM(G,n) := TM(\Gamma,n)$ for any oriented ribbon graph $\Gamma$.
\label{thm:TMInvt}
\end{theorem}

The proof to this theorem follows the same ideas as presented in the proof of Theorem~\ref{thm:evenOddRibbonGraphs} and \Cref{thm:VInvt-of-oriented-ribbon-graph}. It is left as an exercise to the reader. 

\remark \Cref{thm:TMInvt} says that the total matching polynomial,  $TM(G,n)$, is  an abstract graph invariant as defined. It is important that an orientable ribbon graph $\Gamma$ be chosen to define the invariant: a nonorientable ribbon graph may lead to a different polynomial (cf. \Cref{ex:VPDThetaNeg}). Furthermore, the $2$-variable total matching polynomial still depends on the ribbon structure chosen. Picking a different ribbon graph for $G$ may change the homological degrees of the nontrivial classes in ways that lead to a different polynomial. Finally, note that there is no corresponding theorem for the bigraded vertex homology.\\

For $n=2$, recall that the homology is generated by perfect matching face colorings (cf. \Cref{thm:PMClasses}).  Thus, the total matching polynomial evaluated at $n=2$ may be expressed in terms of perfect matchings.

\begin{theorem}
Let $G(V,E)$ be a connected trivalent graph.  Suppose that $G$  has $m$ perfect matchings $\{M_1,\ldots, M_m\}$. If $m>0$, then
$$TM(G,2) = 2^{\frac{1}{2}|V|} \sum_{i=1}^m 2^{\#\{\text{cycles of }G\setminus M_i\}}.$$ 
Otherwise, if $m=0$, then $TM(G,2) = 0$.
\label{thm:TMPMCount}
\end{theorem}

\begin{proof}
\Cref{thm:nonTriv} implies that every perfect matching is represented by a perfect matching face coloring.  \Cref{lem:MatchingSwitch} and \Cref{lem:cyclePairs} guarantee that the harmonic colorings corresponding to this matching occur with the correct multiplicity. 
\end{proof}

\Cref{thm:TMPMCount} may be seen as an analog of Theorem~\ref{thm:oddMatchingsCancel} for non-planar graphs (compare to \cite{Jaeger2}).  When the graph is planar, only the even perfect matchings contribute to the Euler characteristic, and each even perfect matching yields two $3$-edge colorings obtained by coloring the matching edges one color (e.g. purple) and alternating red and blue on the cycles of the complement of the perfect matching.  However, even for non-planar graphs, it is sometimes the case that all perfect matchings are even. For example, see the $K_{3,3}$ graph in \Cref{ex:K33Filtered}.  In that case, the same coloring scheme, combined with \Cref{thm:TMPMCount}, implies:

\begin{corollary}
\label{cor:TMPMCount}
Let $G(V,E)$ be a connected trivalent graph.  Then
$$TM(G,2) \geq 2^{\frac12 |V|} \cdot \# \{ \text{3-edge-colorings of }\Gamma\}$$
with equality if all the perfect matchings of $G$ are even.
 \end{corollary}

Finally, we observe that in the planar case, Theorem~\ref{thm:perfectMatchings} and \Cref{sch:PMZero} show that all perfect matchings are represented in homological degree zero, and \Cref{lem:MatchingSwitch,lem:cyclePairs} show that, in fact, every perfect matching face coloring, in any homological degree, is obtained from a corresponding one in homological degree zero.  Thus, all information needed to compute the evaluation of the total matching polynomial at $n=2$, for a planar graph, is contained in the degree zero filtered $2$-color vertex homology.  However, if we begin with a nonplanar ribbon graph or choose $n>2$, then it is not necessarily true that all relevant information is contained in homological degree zero.  Even in the planar case, though, the rest of the homology provides a useful obstruction to the non-existence of a perfect matching.  The entire homology, and consequently, the total matching polynomial, must vanish for a connected trivalent graph without a perfect matching.  Thus, we obtain the following corollary.

\begin{corollary}
Let $G(V,E)$ be a connected trivalent graph and $\Gamma$ an oriented ribbon graph of it.  Then the following are equivalent:
\begin{itemize}
\item G has a perfect matching.
\item $TM(G,2) > 0$.
\item $\text{dim } \widehat{VCH}_2^i(\Gamma) >0$ for some $i$.
\end{itemize}\label{cor:TFAE-perfect-matching}
\end{corollary}

\section{$4$-regular graphs}
\label{sec:4valent}
In \Cref{rem:generalize}, it was noted that the $n$-color vertex homology (bigraded or filtered) can be generalized from trivalent graphs to $r$-regular graphs.  In this section we discuss this generalization for $r=4$ and give an example of a result that follows from this generalization.

Begin with a $4$-regular graph $G(V,E)$, and a ribbon graph $\Gamma$ for $G$.  Construct the hypercube of vertex states of $\Gamma$, which exists as a subset of the states of the cube of resolutions for the bubbled blowup $\Gamma^B$.  One may then proceed to construct the vertex polynomial by adapting \Cref{eq:vertex-bracket} for a $4$-valent vertex in the obvious way, changing $q^{3m}$ to $q^{4m}$ (with similar shifts in the quantum gradings).  Following the constructions of \Cref{sec:vertexHomology} and \Cref{sec:VertexLee}, we obtain both a bigraded homology theory, denoted $VCH_n^*(\Gamma)$, and a filtered homology theory, denoted $\widehat{VCH}_n^*(\Gamma)$.  The definitions of the differentials all proceed just as in the trivalent case.  The fact that the maps are well-defined and commute follows from the TQFT in Section 9 of \cite{ColorHomology}.  

While the vertex polynomial and homology are well-defined for any $n>1$, a number of interesting properties become apparent when we study the $n=2$ case for the filtered homology specifically.  Just as in the trivalent case, the filtered differential is nonzero precisely when only a single color appears at some $0$-smoothed vertex.  Similarly, the presence of a vertex $1$-smoothing at which only one color appears indicates that the coloring is in the image of some local differential.  Studying the cases of \Cref{fig:4valentCases},  each partial $2$-face coloring in $\widehat{VCH}^*_2(\Gamma)$ specifies two subgraphs:
\begin{enumerate}
\item The \emph{color-match subgraph} is a collection of cycles such that for each edge of the cycle, the face-colors match on each side (shown in green in the second and fourth pictures of \Cref{fig:4valentCases}).
\item The \emph{properly colored subgraph} is the subgraph obtained by deleting the edges of the color-match subgraph (shown in purple in the last three pictures of \Cref{fig:4valentCases}).
\end{enumerate}

\begin{figure}[H]
\includegraphics[scale=1]{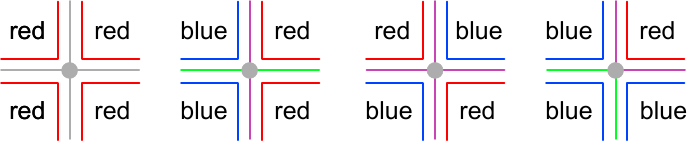}
\caption{The four possible ways of coloring the faces at a $4$-valent vertex (up to symmetry and color-switching).  A partial $2$-face coloring in $\widehat{VCH}_2^*(\Gamma)$ will have no vertices with adjacent faces colored all-red (or all-blue).}\label{fig:4valentCases}
\end{figure}

\Cref{lem:MatchingSwitch} and \Cref{lem:cyclePairs} may be generalized to this context.  Given a state for a $4$-regular graph that supports a partial $2$-face coloring, one may switch the vertex resolutions at each vertex of any component of the properly colored subgraph or any component of the color-match subgraph to obtain a state that supports an equivalent partial $2$-face coloring (equivalent in the sense that the properly colored subgraphs and the color-match subgraphs are identical in the two partial $2$-face colorings).  Moreover, one additional type of switch is available in this case (in the vein of \Cref{lem:MatchingSwitch} and \Cref{lem:cyclePairs}), in that any single degree four vertex in the properly colored subgraph (the third picture in \Cref{fig:4valentCases}) may be switched independently.  Consequently, the only partial $2$-face colorings of $4$-regular graphs that can contribute to the Euler characteristic of the filtered vertex homology are those for which the color-match subgraph and the properly colored subgraph is a family of cycles (second and fourth pictures in \Cref{fig:4valentCases}).  Each of these decompositions defines a family of $4$-edge colorings by labeling the cycles of the properly colored subgraph with two alternating colors and labeling the cycles of the color-match subgraph with a different pair of alternating colors.  In a subsequent paper we will use this approach to prove that the filtered $2$-color vertex homology can be used to count $4$-edge colorings of $4$-regular graphs, which is a generalization of Theorem~\ref{thm:oddMatchingsCancel}.

\section{Conclusion}

Computations using a graphical calculus are usually treated as computations defined only by the (labeled) edges and vertices of the graph used in the calculation. For example,  Penrose's formulas for counting the $3$-edge colorings of a trivalent plane graph describes a property of the edge set of that graph. In calculating with Feynman diagrams, Kauffman-like brackets, spin networks, Penrose graphical notations, $(0+1)$-dimensional TQFTs, and more generally category theory and representation theory, we as researchers tend to concentrate only on the meaning of the edges drawn in our planar diagrams. Rarely do we step back and look globally at the diagrams and see the higher genus closed surfaces, i.e., the ribbon graphs, that our diagrams correspond to.

We hope that the results of this paper and of \cite{ColorHomology} show how valuable it is to study these surfaces and their faces, and that the original diagrams under study (wherever they occur in mathematics)   can be enriched by including this extra information.  For example, the $3$-edge color information about the edges of $G$ in Theorem~\ref{thm:oddMatchingsCancel} becomes more enlightening when one realizes that this information is actually derived from the partial $2$-face colorings on all the possible oriented ribbon graphs of $G$. Armed with this additional knowledge, it becomes possible to generalize Theorem~\ref{thm:oddMatchingsCancel}  from $n=2$ to all natural numbers as in Theorem~\ref{thm:evenOddRibbonGraphs}, and in the process, create and give meaning to a new abstract graph invariant. 

This kind of thinking may lead to new insights. For example, it is well known that  quantum spin networks \cite{KirResh, TV, Kauffman-Lins, T1}, based upon spin networks (cf. \cite{Penrose2, Penrose}),  are closely related to quantum invariants of $3$-manifolds and the Jones polynomial. A spin network is a trivalent ribbon graph where the edges are ``colored'' by natural numbers.  (In this paper, the blowup of a ribbon graph $\Gamma$ can be thought of as a spin network where the edges of the original graph are labeled by ``2''  in the blowup and the remaining edges are labeled by ``1.'')  Spin networks are a diagrammatic description of abstract tensor systems where the color $k$ on an edge indicates the $k+1$ dimensional irreducible representation of $SU(2)$. The evaluation of the network is a contraction of those tensors.  Quantum spin networks are similar but evaluate to rational functions of a variable $q$ using expressions involving  quantum numbers $[n]$. (See the discussion on Chebyshev polynomials and quantum numbers following \Cref{def:vertexPoly}.) These quantum spin networks, in particular the theta and tetrahedral graphs, form the building blocks for topological invariants of closed $3$-manifolds \cite{TV}. We speculate that the homology theories of this paper and in \cite{ColorHomology} are the beginnings of new quantum invariants of manifolds that take into consideration the face colorings of the quantum spin networks. For instance, there is already evidence that the $n$-color polynomials in \cite{ColorHomology} satisfy the pentagon identity like quantum $6j$ symbols do (see the Biedenharn-Elliot Identity in \cite{Kauffman-Lins}), which is needed to prove invariance using Pachner moves for $3$-manifold invariants. Since the $n$-color homology theories have nice cobordism properties, they are likely to satisfy the pentagon identity also. Future research will be in this direction.

\appendix
\section{The seven types of local vertex differentials}\label{app:differentials}

\begin{figure}[H]
$$\begin{tikzpicture}[scale = 0.5]

\begin{scope}
\draw[dashed] (0,-0.18) circle (1 cm);
\draw (0.15,0.866-0.05) to [out = -90, in = 135] (0.8 + 0.08,-0.8 + 0.15);
\draw (-0.15,0.866-0.05) to [out = -90, in = 45] (-0.8 - 0.08,-0.8 + 0.15);
\draw (-0.8 + 0.095,-0.8 - 0.1) to [out = 45, in = 135] (0.8 - 0.095,-0.8 - 0.1);

\draw (0.15,0.866-0.05) to [out = 90, in = 135] (1.2,0.7) to [out = -45, in = -45] (0.8 + 0.08,-0.8 + 0.15);
\draw (-0.15,0.866-0.05) to [out = 90, in = 45] (-1.20, 0.7) to [out = 225, in = 225] (-0.8 - 0.08,-0.8 + 0.15);
\draw (0.8 - 0.095,-0.8 - 0.1) to [out = -45, in = 0] (0, -1.6) to [out = 180, in = 225]  (-0.8 + 0.095,-0.8 - 0.1);
\end{scope}

\begin{scope}[xshift = 1 cm, yshift = -0.2 cm]
\draw[->] (0.5,0.5) --(2,2);
\draw[->] (0.5,0) --(2,0);
\draw[->] (0.5,-0.5) --(2,-2);
\end{scope}

\begin{scope}[xshift = 5.5 cm, yshift = -0.2 cm]
\draw[->] (0.5,0.5) --(2,2);
\draw[->] (0.5,-0.5) --(2,-2);
\draw[->] (0.5,3) --(2,3);
\draw[->] (0.5,-3) --(2,-3);
\draw[->] (0.5,-2.5) --(2,-1);
\draw[->] (0.5,2.5) --(2,1);
\end{scope}

\begin{scope}[xshift = 8.5 cm, yshift = -0.2 cm]
\draw[->] (2,2) -- (3.5,0.5);
\draw[->] (2,0) -- (3.5,0);
\draw[->] (2,-2) -- (3.5,-0.5);
\end{scope}

\begin{scope}[xshift = 4.5 cm, yshift = 3 cm]
\draw[dashed] (0,-0.18) circle (1 cm);
\draw (-0.15,0.866-0.05) to [out = -90, in = 135] (0.8 + 0.08,-0.8 + 0.15);
\draw (0.15,0.866-0.05) to [out = -90, in = 45] (-0.8 - 0.08,-0.8 + 0.15);
\draw (-0.8 + 0.095,-0.8 - 0.1) to [out = 45, in = 135] (0.8 - 0.095,-0.8 - 0.1);

\draw (0.15,0.866-0.05) to [out = 90, in = 135] (1.2,0.7) to [out = -45, in = -45] (0.8 + 0.08,-0.8 + 0.15);
\draw (-0.15,0.866-0.05) to [out = 90, in = 45] (-1.20, 0.7) to [out = 225, in = 225] (-0.8 - 0.08,-0.8 + 0.15);
\draw (0.8 - 0.095,-0.8 - 0.1) to [out = -45, in = 0] (0, -1.6) to [out = 180, in = 225]  (-0.8 + 0.095,-0.8 - 0.1);
\end{scope}

\begin{scope}[xshift = 4.5 cm, yshift = 0 cm]
\draw[dashed] (0,-0.18) circle (1 cm);
\draw (0.15,0.866-0.05) to [out = -90, in = 135] (0.8 + 0.08,-0.8 + 0.15);
\draw (-0.15,0.866-0.05) to [out = -90, in = 45] (-0.8 + 0.095,-0.8 - 0.1);
\draw (-0.8 - 0.08,-0.8 + 0.15) to [out = 45, in = 135] (0.8 - 0.095,-0.8 - 0.1);
\draw (0.15,0.866-0.05) to [out = 90, in = 135] (1.2,0.7) to [out = -45, in = -45] (0.8 + 0.08,-0.8 + 0.15);
\draw (-0.15,0.866-0.05) to [out = 90, in = 45] (-1.20, 0.7) to [out = 225, in = 225] (-0.8 - 0.08,-0.8 + 0.15);
\draw (0.8 - 0.095,-0.8 - 0.1) to [out = -45, in = 0] (0, -1.6) to [out = 180, in = 225]  (-0.8 + 0.095,-0.8 - 0.1);
\end{scope}

\begin{scope}[xshift = 4.5 cm, yshift = -3 cm]
\draw[dashed] (0,-0.18) circle (1 cm);
\draw (0.15,0.866-0.05) to [out = -90, in = 135] (0.8 - 0.095,-0.8 - 0.1);
\draw (-0.15,0.866-0.05) to [out = -90, in = 45] (-0.8 - 0.08,-0.8 + 0.15);
\draw (-0.8 + 0.095,-0.8 - 0.1) to [out = 45, in = 135] (0.8 + 0.08,-0.8 + 0.15);
\draw (0.15,0.866-0.05) to [out = 90, in = 135] (1.2,0.7) to [out = -45, in = -45] (0.8 + 0.08,-0.8 + 0.15);
\draw (-0.15,0.866-0.05) to [out = 90, in = 45] (-1.20, 0.7) to [out = 225, in = 225] (-0.8 - 0.08,-0.8 + 0.15);
\draw (0.8 - 0.095,-0.8 - 0.1) to [out = -45, in = 0] (0, -1.6) to [out = 180, in = 225]  (-0.8 + 0.095,-0.8 - 0.1);
\end{scope}

\begin{scope}[xshift = 4.5 cm]
\begin{scope}[xshift = 4.5 cm, yshift = 3 cm]
\draw[dashed] (0,-0.18) circle (1 cm);
\draw (-0.15,0.866-0.05) to [out = -90, in = 135] (0.8 + 0.08,-0.8 + 0.15);
\draw (0.15,0.866-0.05) to [out = -90,  in = 60] (-0.2, 0) to [out = 240, in = 45] (-0.8 + 0.095,-0.8 - 0.1);
\draw (-0.8 - 0.08,-0.8 + 0.15) to [out = 45, in = 135] (0.8 - 0.095,-0.8 - 0.1);
\draw (0.15,0.866-0.05) to [out = 90, in = 135] (1.2,0.7) to [out = -45, in = -45] (0.8 + 0.08,-0.8 + 0.15);
\draw (-0.15,0.866-0.05) to [out = 90, in = 45] (-1.20, 0.7) to [out = 225, in = 225] (-0.8 - 0.08,-0.8 + 0.15);
\draw (0.8 - 0.095,-0.8 - 0.1) to [out = -45, in = 0] (0, -1.6) to [out = 180, in = 225]  (-0.8 + 0.095,-0.8 - 0.1);
\end{scope}

\begin{scope}[xshift = 4.5 cm, yshift = 0 cm]
\draw[dashed] (0,-0.18) circle (1 cm);
\draw (-0.15,0.866-0.05) to [out = -90, in = 120] (0.2, 0) to [out = -60, in = 135] (0.8 - 0.095,-0.8 - 0.1);
\draw (0.15,0.866-0.05) to [out = -90, in = 45] (-0.8 - 0.08,-0.8 + 0.15);
\draw (-0.8 + 0.095,-0.8 - 0.1) to [out = 45, in = 135] (0.8 + 0.08,-0.8 + 0.15);
\draw (0.15,0.866-0.05) to [out = 90, in = 135] (1.2,0.7) to [out = -45, in = -45] (0.8 + 0.08,-0.8 + 0.15);
\draw (-0.15,0.866-0.05) to [out = 90, in = 45] (-1.20, 0.7) to [out = 225, in = 225] (-0.8 - 0.08,-0.8 + 0.15);
\draw (0.8 - 0.095,-0.8 - 0.1) to [out = -45, in = 0] (0, -1.6) to [out = 180, in = 225]  (-0.8 + 0.095,-0.8 - 0.1);
\end{scope}

\begin{scope}[xshift = 4.5 cm, yshift = -3 cm]
\draw[dashed] (0,-0.18) circle (1 cm);
\draw (0.15,0.866-0.05) to [out = -90, in = 135] (0.8 - 0.095,-0.8 - 0.1);
\draw (-0.15,0.866-0.05) to [out = -90, in = 45] (-0.8 + 0.095,-0.8 - 0.1);
\draw (-0.8 - 0.08,-0.8 + 0.15) to [out = 45, in = 180] (0, -0.5) to [out = 0, in = 135] (0.8 + 0.08,-0.8 + 0.15);
\draw (0.15,0.866-0.05) to [out = 90, in = 135] (1.2,0.7) to [out = -45, in = -45] (0.8 + 0.08,-0.8 + 0.15);
\draw (-0.15,0.866-0.05) to [out = 90, in = 45] (-1.20, 0.7) to [out = 225, in = 225] (-0.8 - 0.08,-0.8 + 0.15);
\draw (0.8 - 0.095,-0.8 - 0.1) to [out = -45, in = 0] (0, -1.6) to [out = 180, in = 225]  (-0.8 + 0.095,-0.8 - 0.1);
\end{scope}

\end{scope}

\begin{scope}[xshift = 13.5 cm]
\begin{scope}
\draw[dashed] (0,-0.18) circle (1 cm);
\draw (-0.15,0.866-0.05) to [out = -90, in = 120] (0.2, 0) to [out = -60, in = 135] (0.8 - 0.095,-0.8 - 0.1);
\draw (0.15,0.866-0.05) to [out = -90,  in = 60] (-0.2, 0) to [out = 240, in = 45] (-0.8 + 0.095,-0.8 - 0.1);
\draw (-0.8 - 0.08,-0.8 + 0.15) to [out = 45, in = 180] (0, -0.5) to [out = 0, in = 135] (0.8 + 0.08,-0.8 + 0.15);
\draw (0.15,0.866-0.05) to [out = 90, in = 135] (1.2,0.7) to [out = -45, in = -45] (0.8 + 0.08,-0.8 + 0.15);
\draw (-0.15,0.866-0.05) to [out = 90, in = 45] (-1.20, 0.7) to [out = 225, in = 225] (-0.8 - 0.08,-0.8 + 0.15);
\draw (0.8 - 0.095,-0.8 - 0.1) to [out = -45, in = 0] (0, -1.6) to [out = 180, in = 225]  (-0.8 + 0.095,-0.8 - 0.1);
\end{scope}

\end{scope}

\end{tikzpicture}$$
\caption{Case 1:  Three circles to one circle.}
\label{fig:Case1}
\end{figure}
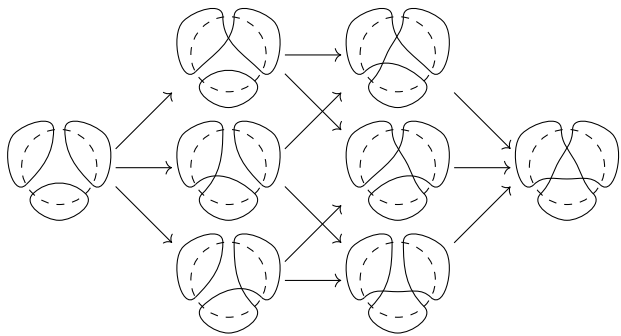

\begin{figure}[H]
$$\begin{tikzpicture}[scale = 0.5]

\begin{scope}
\draw[dashed] (0,-0.18) circle (1 cm);
\draw (0.15,0.866-0.05) to [out = -90, in = 135] (0.8 + 0.08,-0.8 + 0.15);
\draw (-0.15,0.866-0.05) to [out = -90, in = 45] (-0.8 - 0.08,-0.8 + 0.15);
\draw (-0.8 + 0.095,-0.8 - 0.1) to [out = 45, in = 135] (0.8 - 0.095,-0.8 - 0.1);

\draw (0.15,0.866-0.05) to [out = 90, in = 135] (1.2,0.7) to [out = -45, in = -45] (0.8 - 0.095,-0.8 - 0.1);
\draw (-0.15,0.866-0.05) to [out = 90, in = 45] (-1.20, 0.7) to [out = 225, in = 225] (-0.8 - 0.08,-0.8 + 0.15);
\draw (0.8 + 0.08,-0.8 + 0.15) to [out = -45, in = 0] (0, -1.6) to [out = 180, in = 225]  (-0.8 + 0.095,-0.8 - 0.1);
\end{scope}

\begin{scope}[xshift = 1 cm, yshift = -0.2 cm]
\draw[->] (0.5,0.5) --(2,2);
\draw[->] (0.5,0) --(2,0);
\draw[->] (0.5,-0.5) --(2,-2);
\end{scope}

\begin{scope}[xshift = 5.5 cm, yshift = -0.2 cm]
\draw[->] (0.5,0.5) --(2,2);
\draw[->] (0.5,-0.5) --(2,-2);
\draw[->] (0.5,3) --(2,3);
\draw[->] (0.5,-3) --(2,-3);
\draw[->] (0.5,-2.5) --(2,-1);
\draw[->] (0.5,2.5) --(2,1);
\end{scope}

\begin{scope}[xshift = 8.5 cm, yshift = -0.2 cm]
\draw[->] (2,2) -- (3.5,0.5);
\draw[->] (2,0) -- (3.5,0);
\draw[->] (2,-2) -- (3.5,-0.5);
\end{scope}

\begin{scope}[xshift = 4.5 cm, yshift = 3 cm]
\draw[dashed] (0,-0.18) circle (1 cm);
\draw (-0.15,0.866-0.05) to [out = -90, in = 135] (0.8 + 0.08,-0.8 + 0.15);
\draw (0.15,0.866-0.05) to [out = -90, in = 45] (-0.8 - 0.08,-0.8 + 0.15);
\draw (-0.8 + 0.095,-0.8 - 0.1) to [out = 45, in = 135] (0.8 - 0.095,-0.8 - 0.1);

\draw (0.15,0.866-0.05) to [out = 90, in = 135] (1.2,0.7) to [out = -45, in = -45] (0.8 - 0.095,-0.8 - 0.1);
\draw (-0.15,0.866-0.05) to [out = 90, in = 45] (-1.20, 0.7) to [out = 225, in = 225] (-0.8 - 0.08,-0.8 + 0.15);
\draw (0.8 + 0.08,-0.8 + 0.15) to [out = -45, in = 0] (0, -1.6) to [out = 180, in = 225]  (-0.8 + 0.095,-0.8 - 0.1);
\end{scope}

\begin{scope}[xshift = 4.5 cm, yshift = 0 cm]
\draw[dashed] (0,-0.18) circle (1 cm);
\draw (0.15,0.866-0.05) to [out = -90, in = 135] (0.8 + 0.08,-0.8 + 0.15);
\draw (-0.15,0.866-0.05) to [out = -90, in = 45] (-0.8 + 0.095,-0.8 - 0.1);
\draw (-0.8 - 0.08,-0.8 + 0.15) to [out = 45, in = 135] (0.8 - 0.095,-0.8 - 0.1);

\draw (0.15,0.866-0.05) to [out = 90, in = 135] (1.2,0.7) to [out = -45, in = -45] (0.8 - 0.095,-0.8 - 0.1);
\draw (-0.15,0.866-0.05) to [out = 90, in = 45] (-1.20, 0.7) to [out = 225, in = 225] (-0.8 - 0.08,-0.8 + 0.15);
\draw (0.8 + 0.08,-0.8 + 0.15) to [out = -45, in = 0] (0, -1.6) to [out = 180, in = 225]  (-0.8 + 0.095,-0.8 - 0.1);
\end{scope}

\begin{scope}[xshift = 4.5 cm, yshift = -3 cm]
\draw[dashed] (0,-0.18) circle (1 cm);
\draw (0.15,0.866-0.05) to [out = -90, in = 135] (0.8 - 0.095,-0.8 - 0.1);
\draw (-0.15,0.866-0.05) to [out = -90, in = 45] (-0.8 - 0.08,-0.8 + 0.15);
\draw (-0.8 + 0.095,-0.8 - 0.1) to [out = 45, in = 135] (0.8 + 0.08,-0.8 + 0.15);
\draw (0.15,0.866-0.05) to [out = 90, in = 135] (1.2,0.7) to [out = -45, in = -45] (0.8 - 0.095,-0.8 - 0.1);
\draw (-0.15,0.866-0.05) to [out = 90, in = 45] (-1.20, 0.7) to [out = 225, in = 225] (-0.8 - 0.08,-0.8 + 0.15);
\draw (0.8 + 0.08,-0.8 + 0.15) to [out = -45, in = 0] (0, -1.6) to [out = 180, in = 225]  (-0.8 + 0.095,-0.8 - 0.1);
\end{scope}

\begin{scope}[xshift = 4.5 cm]
\begin{scope}[xshift = 4.5 cm, yshift = 3 cm]
\draw[dashed] (0,-0.18) circle (1 cm);
\draw (-0.15,0.866-0.05) to [out = -90, in = 135] (0.8 + 0.08,-0.8 + 0.15);
\draw (0.15,0.866-0.05) to [out = -90,  in = 60] (-0.2, 0) to [out = 240, in = 45] (-0.8 + 0.095,-0.8 - 0.1);
\draw (-0.8 - 0.08,-0.8 + 0.15) to [out = 45, in = 135] (0.8 - 0.095,-0.8 - 0.1);
\draw (0.15,0.866-0.05) to [out = 90, in = 135] (1.2,0.7) to [out = -45, in = -45] (0.8 - 0.095,-0.8 - 0.1);
\draw (-0.15,0.866-0.05) to [out = 90, in = 45] (-1.20, 0.7) to [out = 225, in = 225] (-0.8 - 0.08,-0.8 + 0.15);
\draw (0.8 + 0.08,-0.8 + 0.15) to [out = -45, in = 0] (0, -1.6) to [out = 180, in = 225]  (-0.8 + 0.095,-0.8 - 0.1);
\end{scope}

\begin{scope}[xshift = 4.5 cm, yshift = 0 cm]
\draw[dashed] (0,-0.18) circle (1 cm);
\draw (-0.15,0.866-0.05) to [out = -90, in = 120] (0.2, 0) to [out = -60, in = 135] (0.8 - 0.095,-0.8 - 0.1);
\draw (0.15,0.866-0.05) to [out = -90, in = 45] (-0.8 - 0.08,-0.8 + 0.15);
\draw (-0.8 + 0.095,-0.8 - 0.1) to [out = 45, in = 135] (0.8 + 0.08,-0.8 + 0.15);
\draw (0.15,0.866-0.05) to [out = 90, in = 135] (1.2,0.7) to [out = -45, in = -45] (0.8 - 0.095,-0.8 - 0.1);
\draw (-0.15,0.866-0.05) to [out = 90, in = 45] (-1.20, 0.7) to [out = 225, in = 225] (-0.8 - 0.08,-0.8 + 0.15);
\draw (0.8 + 0.08,-0.8 + 0.15) to [out = -45, in = 0] (0, -1.6) to [out = 180, in = 225]  (-0.8 + 0.095,-0.8 - 0.1);
\end{scope}

\begin{scope}[xshift = 4.5 cm, yshift = -3 cm]
\draw[dashed] (0,-0.18) circle (1 cm);
\draw (0.15,0.866-0.05) to [out = -90, in = 135] (0.8 - 0.095,-0.8 - 0.1);
\draw (-0.15,0.866-0.05) to [out = -90, in = 45] (-0.8 + 0.095,-0.8 - 0.1);
\draw (-0.8 - 0.08,-0.8 + 0.15) to [out = 45, in = 180] (0, -0.5) to [out = 0, in = 135] (0.8 + 0.08,-0.8 + 0.15);
\draw (0.15,0.866-0.05) to [out = 90, in = 135] (1.2,0.7) to [out = -45, in = -45] (0.8 - 0.095,-0.8 - 0.1);
\draw (-0.15,0.866-0.05) to [out = 90, in = 45] (-1.20, 0.7) to [out = 225, in = 225] (-0.8 - 0.08,-0.8 + 0.15);
\draw (0.8 + 0.08,-0.8 + 0.15) to [out = -45, in = 0] (0, -1.6) to [out = 180, in = 225]  (-0.8 + 0.095,-0.8 - 0.1);
\end{scope}

\end{scope}

\begin{scope}[xshift = 13.5 cm]
\begin{scope}
\draw[dashed] (0,-0.18) circle (1 cm);
\draw (-0.15,0.866-0.05) to [out = -90, in = 120] (0.2, 0) to [out = -60, in = 135] (0.8 - 0.095,-0.8 - 0.1);
\draw (0.15,0.866-0.05) to [out = -90,  in = 60] (-0.2, 0) to [out = 240, in = 45] (-0.8 + 0.095,-0.8 - 0.1);
\draw (-0.8 - 0.08,-0.8 + 0.15) to [out = 45, in = 180] (0, -0.5) to [out = 0, in = 135] (0.8 + 0.08,-0.8 + 0.15);
\draw (0.15,0.866-0.05) to [out = 90, in = 135] (1.2,0.7) to [out = -45, in = -45] (0.8 - 0.095,-0.8 - 0.1);
\draw (-0.15,0.866-0.05) to [out = 90, in = 45] (-1.20, 0.7) to [out = 225, in = 225] (-0.8 - 0.08,-0.8 + 0.15);
\draw (0.8 + 0.08,-0.8 + 0.15) to [out = -45, in = 0] (0, -1.6) to [out = 180, in = 225]  (-0.8 + 0.095,-0.8 - 0.1);
\end{scope}

\end{scope}

\end{tikzpicture}$$
\caption{Case 2:  Two circles to one circle.}
\label{fig:Case2}
\end{figure}
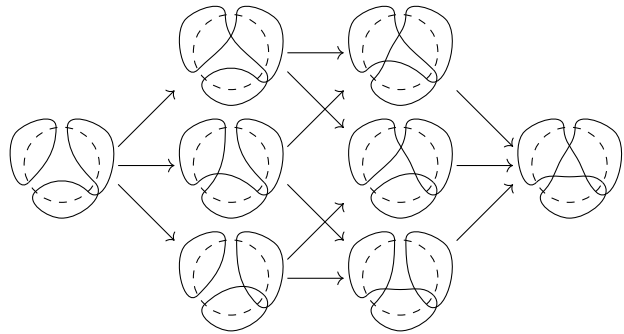

\begin{figure}[H]
$$\begin{tikzpicture}[scale = 0.5]

\begin{scope}
\draw[dashed] (0,-0.18) circle (1 cm);
\draw (0.15,0.866-0.05) to [out = -90, in = 135] (0.8 + 0.08,-0.8 + 0.15);
\draw (-0.15,0.866-0.05) to [out = -90, in = 45] (-0.8 - 0.08,-0.8 + 0.15);
\draw (-0.8 + 0.095,-0.8 - 0.1) to [out = 45, in = 135] (0.8 - 0.095,-0.8 - 0.1);

\draw (0.15,0.866-0.05) to [out = 90, in = 0] (0, 1.2) to [out = 180, in = 90] (-0.15,0.866-0.05);
\draw (0.8 + 0.08,-0.8 + 0.15) to [out = -45, in = 0] (0, -1.75) to [out = 180, in = 225] (-0.8 - 0.08,-0.8 + 0.15) ;
\draw (0.8 - 0.095,-0.8 - 0.1) to [out = -45, in = 0] (0, -1.5) to [out = 180, in = 225] (-0.8 + 0.095,-0.8 - 0.1);
\end{scope}

\begin{scope}[xshift = 1 cm, yshift = -0.2 cm]

\draw[->] (0.5,0.5) --(2,2);
\draw[->] (0.5,0) --(2,0);
\draw[->] (0.5,-0.5) --(2,-2);
\end{scope}

\begin{scope}[xshift = 5.5 cm, yshift = -0.2 cm]

\draw[->] (0.5,0.5) --(2,2);
\draw[->] (0.5,-0.5) --(2,-2);
\draw[->] (0.5,3) --(2,3);
\draw[->] (0.5,-3) --(2,-3);
\draw[->] (0.5,-2.5) --(2,-1);
\draw[->] (0.5,2.5) --(2,1);
\end{scope}

\begin{scope}[xshift = 8.5 cm, yshift = -0.2 cm]

\draw[->] (2,2) -- (3.5,0.5);
\draw[->] (2,0) -- (3.5,0);
\draw[->] (2,-2) -- (3.5,-0.5);
\end{scope}

\begin{scope}[xshift = 4.5 cm, yshift = 3 cm]
\draw[dashed] (0,-0.18) circle (1 cm);
\draw (-0.15,0.866-0.05) to [out = -90, in = 135] (0.8 + 0.08,-0.8 + 0.15);
\draw (0.15,0.866-0.05) to [out = -90, in = 45] (-0.8 - 0.08,-0.8 + 0.15);
\draw (-0.8 + 0.095,-0.8 - 0.1) to [out = 45, in = 135] (0.8 - 0.095,-0.8 - 0.1);

\draw (0.15,0.866-0.05) to [out = 90, in = 0] (0, 1.2) to [out = 180, in = 90] (-0.15,0.866-0.05);
\draw (0.8 + 0.08,-0.8 + 0.15) to [out = -45, in = 0] (0, -1.75) to [out = 180, in = 225] (-0.8 - 0.08,-0.8 + 0.15) ;
\draw (0.8 - 0.095,-0.8 - 0.1) to [out = -45, in = 0] (0, -1.5) to [out = 180, in = 225] (-0.8 + 0.095,-0.8 - 0.1);
\end{scope}

\begin{scope}[xshift = 4.5 cm, yshift = 0 cm]
\draw[dashed] (0,-0.18) circle (1 cm);
\draw (0.15,0.866-0.05) to [out = -90, in = 135] (0.8 + 0.08,-0.8 + 0.15);
\draw (-0.15,0.866-0.05) to [out = -90, in = 45] (-0.8 + 0.095,-0.8 - 0.1);
\draw (-0.8 - 0.08,-0.8 + 0.15) to [out = 45, in = 135] (0.8 - 0.095,-0.8 - 0.1);

\draw (0.15,0.866-0.05) to [out = 90, in = 0] (0, 1.2) to [out = 180, in = 90] (-0.15,0.866-0.05);
\draw (0.8 + 0.08,-0.8 + 0.15) to [out = -45, in = 0] (0, -1.75) to [out = 180, in = 225] (-0.8 - 0.08,-0.8 + 0.15) ;
\draw (0.8 - 0.095,-0.8 - 0.1) to [out = -45, in = 0] (0, -1.5) to [out = 180, in = 225] (-0.8 + 0.095,-0.8 - 0.1);
\end{scope}

\begin{scope}[xshift = 4.5 cm, yshift = -3 cm]
\draw[dashed] (0,-0.18) circle (1 cm);
\draw (0.15,0.866-0.05) to [out = -90, in = 135] (0.8 - 0.095,-0.8 - 0.1);
\draw (-0.15,0.866-0.05) to [out = -90, in = 45] (-0.8 - 0.08,-0.8 + 0.15);
\draw (-0.8 + 0.095,-0.8 - 0.1) to [out = 45, in = 135] (0.8 + 0.08,-0.8 + 0.15);

\draw (0.15,0.866-0.05) to [out = 90, in = 0] (0, 1.2) to [out = 180, in = 90] (-0.15,0.866-0.05);
\draw (0.8 + 0.08,-0.8 + 0.15) to [out = -45, in = 0] (0, -1.75) to [out = 180, in = 225] (-0.8 - 0.08,-0.8 + 0.15) ;
\draw (0.8 - 0.095,-0.8 - 0.1) to [out = -45, in = 0] (0, -1.5) to [out = 180, in = 225] (-0.8 + 0.095,-0.8 - 0.1);
\end{scope}

\begin{scope}[xshift = 4.5 cm]
\begin{scope}[xshift = 4.5 cm, yshift = 3 cm]
\draw[dashed] (0,-0.18) circle (1 cm);
\draw (-0.15,0.866-0.05) to [out = -90, in = 135] (0.8 + 0.08,-0.8 + 0.15);
\draw (0.15,0.866-0.05) to [out = -90,  in = 60] (-0.2, 0) to [out = 240, in = 45] (-0.8 + 0.095,-0.8 - 0.1);
\draw (-0.8 - 0.08,-0.8 + 0.15) to [out = 45, in = 135] (0.8 - 0.095,-0.8 - 0.1);

\draw (0.15,0.866-0.05) to [out = 90, in = 0] (0, 1.2) to [out = 180, in = 90] (-0.15,0.866-0.05);
\draw (0.8 + 0.08,-0.8 + 0.15) to [out = -45, in = 0] (0, -1.75) to [out = 180, in = 225] (-0.8 - 0.08,-0.8 + 0.15) ;
\draw (0.8 - 0.095,-0.8 - 0.1) to [out = -45, in = 0] (0, -1.5) to [out = 180, in = 225] (-0.8 + 0.095,-0.8 - 0.1);
\end{scope}

\begin{scope}[xshift = 4.5 cm, yshift = 0 cm]
\draw[dashed] (0,-0.18) circle (1 cm);
\draw (-0.15,0.866-0.05) to [out = -90, in = 120] (0.2, 0) to [out = -60, in = 135] (0.8 - 0.095,-0.8 - 0.1);
\draw (0.15,0.866-0.05) to [out = -90, in = 45] (-0.8 - 0.08,-0.8 + 0.15);
\draw (-0.8 + 0.095,-0.8 - 0.1) to [out = 45, in = 135] (0.8 + 0.08,-0.8 + 0.15);

\draw (0.15,0.866-0.05) to [out = 90, in = 0] (0, 1.2) to [out = 180, in = 90] (-0.15,0.866-0.05);
\draw (0.8 + 0.08,-0.8 + 0.15) to [out = -45, in = 0] (0, -1.75) to [out = 180, in = 225] (-0.8 - 0.08,-0.8 + 0.15) ;
\draw (0.8 - 0.095,-0.8 - 0.1) to [out = -45, in = 0] (0, -1.5) to [out = 180, in = 225] (-0.8 + 0.095,-0.8 - 0.1);
\end{scope}

\begin{scope}[xshift = 4.5 cm, yshift = -3 cm]
\draw[dashed] (0,-0.18) circle (1 cm);
\draw (0.15,0.866-0.05) to [out = -90, in = 135] (0.8 - 0.095,-0.8 - 0.1);
\draw (-0.15,0.866-0.05) to [out = -90, in = 45] (-0.8 + 0.095,-0.8 - 0.1);
\draw (-0.8 - 0.08,-0.8 + 0.15) to [out = 45, in = 180] (0, -0.5) to [out = 0, in = 135] (0.8 + 0.08,-0.8 + 0.15);

\draw (0.15,0.866-0.05) to [out = 90, in = 0] (0, 1.2) to [out = 180, in = 90] (-0.15,0.866-0.05);
\draw (0.8 + 0.08,-0.8 + 0.15) to [out = -45, in = 0] (0, -1.75) to [out = 180, in = 225] (-0.8 - 0.08,-0.8 + 0.15) ;
\draw (0.8 - 0.095,-0.8 - 0.1) to [out = -45, in = 0] (0, -1.5) to [out = 180, in = 225] (-0.8 + 0.095,-0.8 - 0.1);
\end{scope}

\end{scope}

\begin{scope}[xshift = 13.5 cm]
\begin{scope}
\draw[dashed] (0,-0.18) circle (1 cm);
\draw (-0.15,0.866-0.05) to [out = -90, in = 120] (0.2, 0) to [out = -60, in = 135] (0.8 - 0.095,-0.8 - 0.1);
\draw (0.15,0.866-0.05) to [out = -90,  in = 60] (-0.2, 0) to [out = 240, in = 45] (-0.8 + 0.095,-0.8 - 0.1);
\draw (-0.8 - 0.08,-0.8 + 0.15) to [out = 45, in = 180] (0, -0.5) to [out = 0, in = 135] (0.8 + 0.08,-0.8 + 0.15);

\draw (0.15,0.866-0.05) to [out = 90, in = 0] (0, 1.2) to [out = 180, in = 90] (-0.15,0.866-0.05);
\draw (0.8 + 0.08,-0.8 + 0.15) to [out = -45, in = 0] (0, -1.75) to [out = 180, in = 225] (-0.8 - 0.08,-0.8 + 0.15) ;
\draw (0.8 - 0.095,-0.8 - 0.1) to [out = -45, in = 0] (0, -1.5) to [out = 180, in = 225] (-0.8 + 0.095,-0.8 - 0.1);
\end{scope}

\end{scope}

\end{tikzpicture}$$
\caption{Case 3:  Two circles to two circles.}
\label{fig:Case3App}
\end{figure}
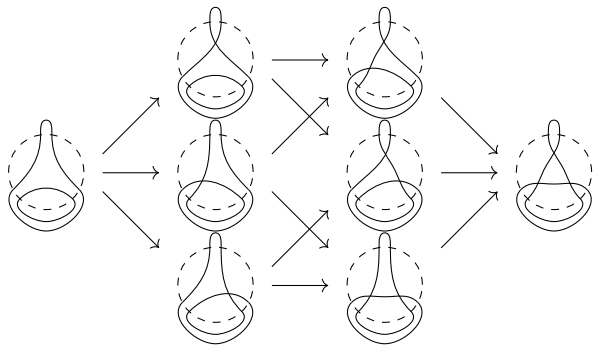

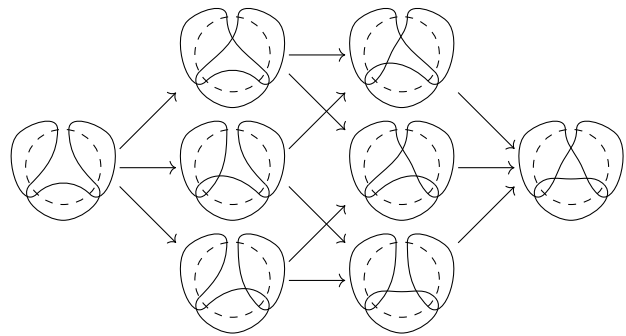
\begin{figure}[H]
$$\begin{tikzpicture}[scale = 0.5]

\begin{scope}
\draw[dashed] (0,-0.18) circle (1 cm);
\draw (0.15,0.866-0.05) to [out = -90, in = 135] (0.8 + 0.08,-0.8 + 0.15);
\draw (-0.15,0.866-0.05) to [out = -90, in = 45] (-0.8 - 0.08,-0.8 + 0.15);
\draw (-0.8 + 0.095,-0.8 - 0.1) to [out = 45, in = 135] (0.8 - 0.095,-0.8 - 0.1);
\draw (0.15,0.866-0.05) to [out = 90, in = 135] (1.2,0.7) to [out = -45, in = -45] (0.8 - 0.095,-0.8 - 0.1);
\draw (-0.15,0.866-0.05) to [out = 90, in = 45] (-1.20, 0.7) to [out = 225, in = 225]  (-0.8 + 0.095,-0.8 - 0.1);
\draw (0.8 + 0.08,-0.8 + 0.15) to [out = -45, in = 0] (0, -1.6) to [out = 180, in = 225] (-0.8 - 0.08,-0.8 + 0.15);
\end{scope}

\begin{scope}[xshift = 1 cm, yshift = -0.2 cm]
\draw[->] (0.5,0.5) --(2,2);
\draw[->] (0.5,0) --(2,0);
\draw[->] (0.5,-0.5) --(2,-2);
\end{scope}

\begin{scope}[xshift = 5.5 cm, yshift = -0.2 cm]
\draw[->] (0.5,0.5) --(2,2);
\draw[->] (0.5,-0.5) --(2,-2);
\draw[->] (0.5,3) --(2,3);
\draw[->] (0.5,-3) --(2,-3);
\draw[->] (0.5,-2.5) --(2,-1);
\draw[->] (0.5,2.5) --(2,1);
\end{scope}

\begin{scope}[xshift = 8.5 cm, yshift = -0.2 cm]
\draw[->] (2,2) -- (3.5,0.5);
\draw[->] (2,0) -- (3.5,0);
\draw[->] (2,-2) -- (3.5,-0.5);
\end{scope}

\begin{scope}[xshift = 4.5 cm, yshift = 3 cm]
\draw[dashed] (0,-0.18) circle (1 cm);
\draw (-0.15,0.866-0.05) to [out = -90, in = 135] (0.8 + 0.08,-0.8 + 0.15);
\draw (0.15,0.866-0.05) to [out = -90, in = 45] (-0.8 - 0.08,-0.8 + 0.15);
\draw (-0.8 + 0.095,-0.8 - 0.1) to [out = 45, in = 135] (0.8 - 0.095,-0.8 - 0.1);
\draw (0.15,0.866-0.05) to [out = 90, in = 135] (1.2,0.7) to [out = -45, in = -45] (0.8 - 0.095,-0.8 - 0.1);
\draw (-0.15,0.866-0.05) to [out = 90, in = 45] (-1.20, 0.7) to [out = 225, in = 225]  (-0.8 + 0.095,-0.8 - 0.1);
\draw (0.8 + 0.08,-0.8 + 0.15) to [out = -45, in = 0] (0, -1.6) to [out = 180, in = 225] (-0.8 - 0.08,-0.8 + 0.15);
\end{scope}

\begin{scope}[xshift = 4.5 cm, yshift = 0 cm]
\draw[dashed] (0,-0.18) circle (1 cm);
\draw (0.15,0.866-0.05) to [out = -90, in = 135] (0.8 + 0.08,-0.8 + 0.15);
\draw (-0.15,0.866-0.05) to [out = -90, in = 45] (-0.8 + 0.095,-0.8 - 0.1);
\draw (-0.8 - 0.08,-0.8 + 0.15) to [out = 45, in = 135] (0.8 - 0.095,-0.8 - 0.1);

\draw (0.15,0.866-0.05) to [out = 90, in = 135] (1.2,0.7) to [out = -45, in = -45] (0.8 - 0.095,-0.8 - 0.1);
\draw (-0.15,0.866-0.05) to [out = 90, in = 45] (-1.20, 0.7) to [out = 225, in = 225]  (-0.8 + 0.095,-0.8 - 0.1);
\draw (0.8 + 0.08,-0.8 + 0.15) to [out = -45, in = 0] (0, -1.6) to [out = 180, in = 225] (-0.8 - 0.08,-0.8 + 0.15);
\end{scope}

\begin{scope}[xshift = 4.5 cm, yshift = -3 cm]
\draw[dashed] (0,-0.18) circle (1 cm);
\draw (0.15,0.866-0.05) to [out = -90, in = 135] (0.8 - 0.095,-0.8 - 0.1);
\draw (-0.15,0.866-0.05) to [out = -90, in = 45] (-0.8 - 0.08,-0.8 + 0.15);
\draw (-0.8 + 0.095,-0.8 - 0.1) to [out = 45, in = 135] (0.8 + 0.08,-0.8 + 0.15);
\draw (0.15,0.866-0.05) to [out = 90, in = 135] (1.2,0.7) to [out = -45, in = -45] (0.8 - 0.095,-0.8 - 0.1);
\draw (-0.15,0.866-0.05) to [out = 90, in = 45] (-1.20, 0.7) to [out = 225, in = 225]  (-0.8 + 0.095,-0.8 - 0.1);
\draw (0.8 + 0.08,-0.8 + 0.15) to [out = -45, in = 0] (0, -1.6) to [out = 180, in = 225] (-0.8 - 0.08,-0.8 + 0.15);
\end{scope}

\begin{scope}[xshift = 4.5 cm]
\begin{scope}[xshift = 4.5 cm, yshift = 3 cm]
\draw[dashed] (0,-0.18) circle (1 cm);
\draw (-0.15,0.866-0.05) to [out = -90, in = 135] (0.8 + 0.08,-0.8 + 0.15);
\draw (0.15,0.866-0.05) to [out = -90,  in = 60] (-0.2, 0) to [out = 240, in = 45] (-0.8 + 0.095,-0.8 - 0.1);
\draw (-0.8 - 0.08,-0.8 + 0.15) to [out = 45, in = 135] (0.8 - 0.095,-0.8 - 0.1);
\draw (0.15,0.866-0.05) to [out = 90, in = 135] (1.2,0.7) to [out = -45, in = -45] (0.8 - 0.095,-0.8 - 0.1);
\draw (-0.15,0.866-0.05) to [out = 90, in = 45] (-1.20, 0.7) to [out = 225, in = 225]  (-0.8 + 0.095,-0.8 - 0.1);
\draw (0.8 + 0.08,-0.8 + 0.15) to [out = -45, in = 0] (0, -1.6) to [out = 180, in = 225] (-0.8 - 0.08,-0.8 + 0.15);
\end{scope}

\begin{scope}[xshift = 4.5 cm, yshift = 0 cm]
\draw[dashed] (0,-0.18) circle (1 cm);
\draw (-0.15,0.866-0.05) to [out = -90, in = 120] (0.2, 0) to [out = -60, in = 135] (0.8 - 0.095,-0.8 - 0.1);
\draw (0.15,0.866-0.05) to [out = -90, in = 45] (-0.8 - 0.08,-0.8 + 0.15);
\draw (-0.8 + 0.095,-0.8 - 0.1) to [out = 45, in = 135] (0.8 + 0.08,-0.8 + 0.15);
\draw (0.15,0.866-0.05) to [out = 90, in = 135] (1.2,0.7) to [out = -45, in = -45] (0.8 - 0.095,-0.8 - 0.1);
\draw (-0.15,0.866-0.05) to [out = 90, in = 45] (-1.20, 0.7) to [out = 225, in = 225]  (-0.8 + 0.095,-0.8 - 0.1);
\draw (0.8 + 0.08,-0.8 + 0.15) to [out = -45, in = 0] (0, -1.6) to [out = 180, in = 225] (-0.8 - 0.08,-0.8 + 0.15);
\end{scope}

\begin{scope}[xshift = 4.5 cm, yshift = -3 cm]
\draw[dashed] (0,-0.18) circle (1 cm);
\draw (0.15,0.866-0.05) to [out = -90, in = 135] (0.8 - 0.095,-0.8 - 0.1);
\draw (-0.15,0.866-0.05) to [out = -90, in = 45] (-0.8 + 0.095,-0.8 - 0.1);
\draw (-0.8 - 0.08,-0.8 + 0.15) to [out = 45, in = 180] (0, -0.5) to [out = 0, in = 135] (0.8 + 0.08,-0.8 + 0.15);
\draw (0.15,0.866-0.05) to [out = 90, in = 135] (1.2,0.7) to [out = -45, in = -45] (0.8 - 0.095,-0.8 - 0.1);
\draw (-0.15,0.866-0.05) to [out = 90, in = 45] (-1.20, 0.7) to [out = 225, in = 225]  (-0.8 + 0.095,-0.8 - 0.1);
\draw (0.8 + 0.08,-0.8 + 0.15) to [out = -45, in = 0] (0, -1.6) to [out = 180, in = 225] (-0.8 - 0.08,-0.8 + 0.15);
\end{scope}

\end{scope}

\begin{scope}[xshift = 13.5 cm]
\begin{scope}
\draw[dashed] (0,-0.18) circle (1 cm);
\draw (-0.15,0.866-0.05) to [out = -90, in = 120] (0.2, 0) to [out = -60, in = 135] (0.8 - 0.095,-0.8 - 0.1);
\draw (0.15,0.866-0.05) to [out = -90,  in = 60] (-0.2, 0) to [out = 240, in = 45] (-0.8 + 0.095,-0.8 - 0.1);
\draw (-0.8 - 0.08,-0.8 + 0.15) to [out = 45, in = 180] (0, -0.5) to [out = 0, in = 135] (0.8 + 0.08,-0.8 + 0.15);

\draw (0.15,0.866-0.05) to [out = 90, in = 135] (1.2,0.7) to [out = -45, in = -45] (0.8 - 0.095,-0.8 - 0.1);
\draw (-0.15,0.866-0.05) to [out = 90, in = 45] (-1.20, 0.7) to [out = 225, in = 225]  (-0.8 + 0.095,-0.8 - 0.1);
\draw (0.8 + 0.08,-0.8 + 0.15) to [out = -45, in = 0] (0, -1.6) to [out = 180, in = 225] (-0.8 - 0.08,-0.8 + 0.15);
\end{scope}

\end{scope}

\end{tikzpicture}$$
\caption{Case 4:  One circle to two circles.}
\label{fig:Case4}
\end{figure}

\begin{figure}[H]
$$\begin{tikzpicture}[scale = 0.5]

\begin{scope}
\draw[dashed] (0,-0.18) circle (1 cm);
\draw (0.15,0.866-0.05) to [out = -90, in = 135] (0.8 + 0.08,-0.8 + 0.15);
\draw (-0.15,0.866-0.05) to [out = -90, in = 45] (-0.8 - 0.08,-0.8 + 0.15);
\draw (-0.8 + 0.095,-0.8 - 0.1) to [out = 45, in = 135] (0.8 - 0.095,-0.8 - 0.1);

\draw (0.15,0.866-0.05) to [out = 90, in = 0] (0, 1.2) to [out = 180, in = 90] (-0.15,0.866-0.05);
\draw (0.8 + 0.08,-0.8 + 0.15) to [out = -45, in = -20] (0, -1.5) to [out = 160, in = 225] (-0.8 + 0.095,-0.8 - 0.1);
\draw (0.8 - 0.095,-0.8 - 0.1) to [out = -45, in = 20] (0, -1.5) to [out = 200, in = 225]  (-0.8 - 0.08,-0.8 + 0.15);
\end{scope}

\begin{scope}[xshift = 1 cm, yshift = -0.2 cm]
\draw[->] (0.5,0.5) --(2,2);
\draw[->] (0.5,0) --(2,0);
\draw[->] (0.5,-0.5) --(2,-2);
\end{scope}

\begin{scope}[xshift = 5.5 cm, yshift = -0.2 cm]
\draw[->] (0.5,0.5) --(2,2);
\draw[->] (0.5,-0.5) --(2,-2);
\draw[->] (0.5,3) --(2,3);
\draw[->] (0.5,-3) --(2,-3);
\draw[->] (0.5,-2.5) --(2,-1);
\draw[->] (0.5,2.5) --(2,1);
\end{scope}

\begin{scope}[xshift = 8.5 cm, yshift = -0.2 cm]
\draw[->] (2,2) -- (3.5,0.5);
\draw[->] (2,0) -- (3.5,0);
\draw[->] (2,-2) -- (3.5,-0.5);
\end{scope}

\begin{scope}[xshift = 4.5 cm, yshift = 3 cm]
\draw[dashed] (0,-0.18) circle (1 cm);
\draw (-0.15,0.866-0.05) to [out = -90, in = 135] (0.8 + 0.08,-0.8 + 0.15);
\draw (0.15,0.866-0.05) to [out = -90, in = 45] (-0.8 - 0.08,-0.8 + 0.15);
\draw (-0.8 + 0.095,-0.8 - 0.1) to [out = 45, in = 135] (0.8 - 0.095,-0.8 - 0.1);
\draw (0.15,0.866-0.05) to [out = 90, in = 0] (0, 1.2) to [out = 180, in = 90] (-0.15,0.866-0.05);
\draw (0.8 + 0.08,-0.8 + 0.15) to [out = -45, in = -20] (0, -1.5) to [out = 160, in = 225] (-0.8 + 0.095,-0.8 - 0.1);
\draw (0.8 - 0.095,-0.8 - 0.1) to [out = -45, in = 20] (0, -1.5) to [out = 200, in = 225]  (-0.8 - 0.08,-0.8 + 0.15);
\end{scope}

\begin{scope}[xshift = 4.5 cm, yshift = 0 cm]
\draw[dashed] (0,-0.18) circle (1 cm);
\draw (0.15,0.866-0.05) to [out = -90, in = 135] (0.8 + 0.08,-0.8 + 0.15);
\draw (-0.15,0.866-0.05) to [out = -90, in = 45] (-0.8 + 0.095,-0.8 - 0.1);
\draw (-0.8 - 0.08,-0.8 + 0.15) to [out = 45, in = 135] (0.8 - 0.095,-0.8 - 0.1);
\draw (0.15,0.866-0.05) to [out = 90, in = 0] (0, 1.2) to [out = 180, in = 90] (-0.15,0.866-0.05);
\draw (0.8 + 0.08,-0.8 + 0.15) to [out = -45, in = -20] (0, -1.5) to [out = 160, in = 225] (-0.8 + 0.095,-0.8 - 0.1);
\draw (0.8 - 0.095,-0.8 - 0.1) to [out = -45, in = 20] (0, -1.5) to [out = 200, in = 225]  (-0.8 - 0.08,-0.8 + 0.15);
\end{scope}

\begin{scope}[xshift = 4.5 cm, yshift = -3 cm]
\draw[dashed] (0,-0.18) circle (1 cm);
\draw (0.15,0.866-0.05) to [out = -90, in = 135] (0.8 - 0.095,-0.8 - 0.1);
\draw (-0.15,0.866-0.05) to [out = -90, in = 45] (-0.8 - 0.08,-0.8 + 0.15);
\draw (-0.8 + 0.095,-0.8 - 0.1) to [out = 45, in = 135] (0.8 + 0.08,-0.8 + 0.15);
\draw (0.15,0.866-0.05) to [out = 90, in = 0] (0, 1.2) to [out = 180, in = 90] (-0.15,0.866-0.05);
\draw (0.8 + 0.08,-0.8 + 0.15) to [out = -45, in = -20] (0, -1.5) to [out = 160, in = 225] (-0.8 + 0.095,-0.8 - 0.1);
\draw (0.8 - 0.095,-0.8 - 0.1) to [out = -45, in = 20] (0, -1.5) to [out = 200, in = 225]  (-0.8 - 0.08,-0.8 + 0.15);
\end{scope}

\begin{scope}[xshift = 4.5 cm]
\begin{scope}[xshift = 4.5 cm, yshift = 3 cm]
\draw[dashed] (0,-0.18) circle (1 cm);
\draw (-0.15,0.866-0.05) to [out = -90, in = 135] (0.8 + 0.08,-0.8 + 0.15);
\draw (0.15,0.866-0.05) to [out = -90,  in = 60] (-0.2, 0) to [out = 240, in = 45] (-0.8 + 0.095,-0.8 - 0.1);
\draw (-0.8 - 0.08,-0.8 + 0.15) to [out = 45, in = 135] (0.8 - 0.095,-0.8 - 0.1);
\draw (0.15,0.866-0.05) to [out = 90, in = 0] (0, 1.2) to [out = 180, in = 90] (-0.15,0.866-0.05);
\draw (0.8 + 0.08,-0.8 + 0.15) to [out = -45, in = -20] (0, -1.5) to [out = 160, in = 225] (-0.8 + 0.095,-0.8 - 0.1);
\draw (0.8 - 0.095,-0.8 - 0.1) to [out = -45, in = 20] (0, -1.5) to [out = 200, in = 225]  (-0.8 - 0.08,-0.8 + 0.15);
\end{scope}

\begin{scope}[xshift = 4.5 cm, yshift = 0 cm]
\draw[dashed] (0,-0.18) circle (1 cm);
\draw (-0.15,0.866-0.05) to [out = -90, in = 120] (0.2, 0) to [out = -60, in = 135] (0.8 - 0.095,-0.8 - 0.1);
\draw (0.15,0.866-0.05) to [out = -90, in = 45] (-0.8 - 0.08,-0.8 + 0.15);
\draw (-0.8 + 0.095,-0.8 - 0.1) to [out = 45, in = 135] (0.8 + 0.08,-0.8 + 0.15);
\draw (0.15,0.866-0.05) to [out = 90, in = 0] (0, 1.2) to [out = 180, in = 90] (-0.15,0.866-0.05);
\draw (0.8 + 0.08,-0.8 + 0.15) to [out = -45, in = -20] (0, -1.5) to [out = 160, in = 225] (-0.8 + 0.095,-0.8 - 0.1);
\draw (0.8 - 0.095,-0.8 - 0.1) to [out = -45, in = 20] (0, -1.5) to [out = 200, in = 225]  (-0.8 - 0.08,-0.8 + 0.15);
\end{scope}

\begin{scope}[xshift = 4.5 cm, yshift = -3 cm]
\draw[dashed] (0,-0.18) circle (1 cm);
\draw (0.15,0.866-0.05) to [out = -90, in = 135] (0.8 - 0.095,-0.8 - 0.1);
\draw (-0.15,0.866-0.05) to [out = -90, in = 45] (-0.8 + 0.095,-0.8 - 0.1);
\draw (-0.8 - 0.08,-0.8 + 0.15) to [out = 45, in = 180] (0, -0.5) to [out = 0, in = 135] (0.8 + 0.08,-0.8 + 0.15);
\draw (0.15,0.866-0.05) to [out = 90, in = 0] (0, 1.2) to [out = 180, in = 90] (-0.15,0.866-0.05);
\draw (0.8 + 0.08,-0.8 + 0.15) to [out = -45, in = -20] (0, -1.5) to [out = 160, in = 225] (-0.8 + 0.095,-0.8 - 0.1);
\draw (0.8 - 0.095,-0.8 - 0.1) to [out = -45, in = 20] (0, -1.5) to [out = 200, in = 225]  (-0.8 - 0.08,-0.8 + 0.15);
\end{scope}

\end{scope}

\begin{scope}[xshift = 13.5 cm]
\begin{scope}
\draw[dashed] (0,-0.18) circle (1 cm);
\draw (-0.15,0.866-0.05) to [out = -90, in = 120] (0.2, 0) to [out = -60, in = 135] (0.8 - 0.095,-0.8 - 0.1);
\draw (0.15,0.866-0.05) to [out = -90,  in = 60] (-0.2, 0) to [out = 240, in = 45] (-0.8 + 0.095,-0.8 - 0.1);
\draw (-0.8 - 0.08,-0.8 + 0.15) to [out = 45, in = 180] (0, -0.5) to [out = 0, in = 135] (0.8 + 0.08,-0.8 + 0.15);
\draw (0.15,0.866-0.05) to [out = 90, in = 0] (0, 1.2) to [out = 180, in = 90] (-0.15,0.866-0.05);
\draw (0.8 + 0.08,-0.8 + 0.15) to [out = -45, in = -20] (0, -1.5) to [out = 160, in = 225] (-0.8 + 0.095,-0.8 - 0.1);
\draw (0.8 - 0.095,-0.8 - 0.1) to [out = -45, in = 20] (0, -1.5) to [out = 200, in = 225]  (-0.8 - 0.08,-0.8 + 0.15);
\end{scope}

\end{scope}

\end{tikzpicture}$$
\caption{Case 5:  One circle to one circle.}
\label{fig:Case5}
\end{figure}
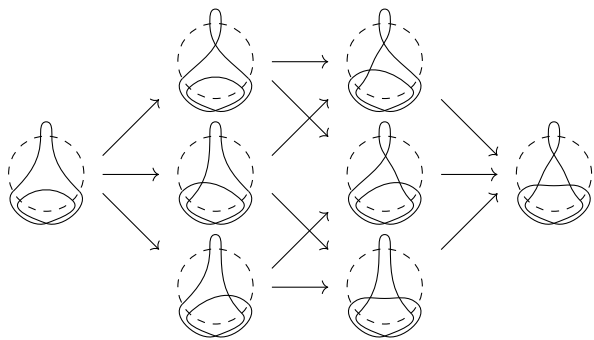

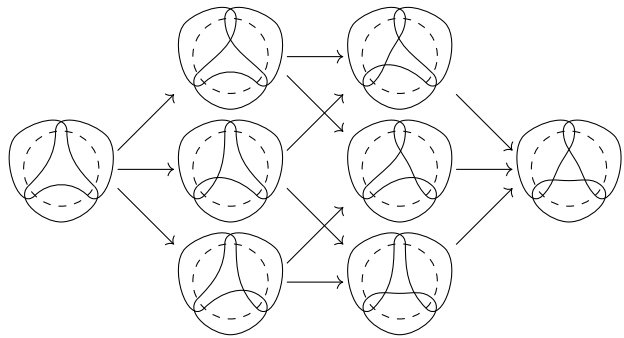
\begin{figure}[H]
$$\begin{tikzpicture}[scale = 0.5]

\begin{scope}
\draw[dashed] (0,-0.18) circle (1 cm);
\draw (0.15,0.866-0.05) to [out = -90, in = 135] (0.8 + 0.08,-0.8 + 0.15);
\draw (-0.15,0.866-0.05) to [out = -90, in = 45] (-0.8 - 0.08,-0.8 + 0.15);
\draw (-0.8 + 0.095,-0.8 - 0.1) to [out = 45, in = 135] (0.8 - 0.095,-0.8 - 0.1);
\draw (-0.15,0.866-0.05) to [out = 90, in = 135] (1.2,0.7) to [out = -45, in = -45] (0.8 - 0.095,-0.8 - 0.1);
\draw (0.15,0.866-0.05) to [out = 90, in = 45] (-1.20, 0.7) to [out = 225, in = 225]  (-0.8 + 0.095,-0.8 - 0.1);
\draw (0.8 + 0.08,-0.8 + 0.15) to [out = -45, in = 0] (0, -1.6) to [out = 180, in = 225] (-0.8 - 0.08,-0.8 + 0.15);
\end{scope}

\begin{scope}[xshift = 1 cm, yshift = -0.2 cm]
\draw[->] (0.5,0.5) --(2,2);
\draw[->] (0.5,0) --(2,0);
\draw[->] (0.5,-0.5) --(2,-2);
\end{scope}

\begin{scope}[xshift = 5.5 cm, yshift = -0.2 cm]
\draw[->] (0.5,0.5) --(2,2);
\draw[->] (0.5,-0.5) --(2,-2);
\draw[->] (0.5,3) --(2,3);
\draw[->] (0.5,-3) --(2,-3);
\draw[->] (0.5,-2.5) --(2,-1);
\draw[->] (0.5,2.5) --(2,1);
\end{scope}

\begin{scope}[xshift = 8.5 cm, yshift = -0.2 cm]
\draw[->] (2,2) -- (3.5,0.5);
\draw[->] (2,0) -- (3.5,0);
\draw[->] (2,-2) -- (3.5,-0.5);
\end{scope}

\begin{scope}[xshift = 4.5 cm, yshift = 3 cm]
\draw[dashed] (0,-0.18) circle (1 cm);
\draw (-0.15,0.866-0.05) to [out = -90, in = 135] (0.8 + 0.08,-0.8 + 0.15);
\draw (0.15,0.866-0.05) to [out = -90, in = 45] (-0.8 - 0.08,-0.8 + 0.15);
\draw (-0.8 + 0.095,-0.8 - 0.1) to [out = 45, in = 135] (0.8 - 0.095,-0.8 - 0.1);
\draw (-0.15,0.866-0.05) to [out = 90, in = 135] (1.2,0.7) to [out = -45, in = -45] (0.8 - 0.095,-0.8 - 0.1);
\draw (0.15,0.866-0.05) to [out = 90, in = 45] (-1.20, 0.7) to [out = 225, in = 225]  (-0.8 + 0.095,-0.8 - 0.1);
\draw (0.8 + 0.08,-0.8 + 0.15) to [out = -45, in = 0] (0, -1.6) to [out = 180, in = 225] (-0.8 - 0.08,-0.8 + 0.15);
\end{scope}

\begin{scope}[xshift = 4.5 cm, yshift = 0 cm]
\draw[dashed] (0,-0.18) circle (1 cm);
\draw (0.15,0.866-0.05) to [out = -90, in = 135] (0.8 + 0.08,-0.8 + 0.15);
\draw (-0.15,0.866-0.05) to [out = -90, in = 45] (-0.8 + 0.095,-0.8 - 0.1);
\draw (-0.8 - 0.08,-0.8 + 0.15) to [out = 45, in = 135] (0.8 - 0.095,-0.8 - 0.1);
\draw (-0.15,0.866-0.05) to [out = 90, in = 135] (1.2,0.7) to [out = -45, in = -45] (0.8 - 0.095,-0.8 - 0.1);
\draw (0.15,0.866-0.05) to [out = 90, in = 45] (-1.20, 0.7) to [out = 225, in = 225]  (-0.8 + 0.095,-0.8 - 0.1);
\draw (0.8 + 0.08,-0.8 + 0.15) to [out = -45, in = 0] (0, -1.6) to [out = 180, in = 225] (-0.8 - 0.08,-0.8 + 0.15);
\end{scope}

\begin{scope}[xshift = 4.5 cm, yshift = -3 cm]
\draw[dashed] (0,-0.18) circle (1 cm);
\draw (0.15,0.866-0.05) to [out = -90, in = 135] (0.8 - 0.095,-0.8 - 0.1);
\draw (-0.15,0.866-0.05) to [out = -90, in = 45] (-0.8 - 0.08,-0.8 + 0.15);
\draw (-0.8 + 0.095,-0.8 - 0.1) to [out = 45, in = 135] (0.8 + 0.08,-0.8 + 0.15);
\draw (-0.15,0.866-0.05) to [out = 90, in = 135] (1.2,0.7) to [out = -45, in = -45] (0.8 - 0.095,-0.8 - 0.1);
\draw (0.15,0.866-0.05) to [out = 90, in = 45] (-1.20, 0.7) to [out = 225, in = 225]  (-0.8 + 0.095,-0.8 - 0.1);
\draw (0.8 + 0.08,-0.8 + 0.15) to [out = -45, in = 0] (0, -1.6) to [out = 180, in = 225] (-0.8 - 0.08,-0.8 + 0.15);
\end{scope}

\begin{scope}[xshift = 4.5 cm]
\begin{scope}[xshift = 4.5 cm, yshift = 3 cm]
\draw[dashed] (0,-0.18) circle (1 cm);
\draw (-0.15,0.866-0.05) to [out = -90, in = 135] (0.8 + 0.08,-0.8 + 0.15);
\draw (0.15,0.866-0.05) to [out = -90,  in = 60] (-0.2, 0) to [out = 240, in = 45] (-0.8 + 0.095,-0.8 - 0.1);
\draw (-0.8 - 0.08,-0.8 + 0.15) to [out = 45, in = 135] (0.8 - 0.095,-0.8 - 0.1);
\draw (-0.15,0.866-0.05) to [out = 90, in = 135] (1.2,0.7) to [out = -45, in = -45] (0.8 - 0.095,-0.8 - 0.1);
\draw (0.15,0.866-0.05) to [out = 90, in = 45] (-1.20, 0.7) to [out = 225, in = 225]  (-0.8 + 0.095,-0.8 - 0.1);
\draw (0.8 + 0.08,-0.8 + 0.15) to [out = -45, in = 0] (0, -1.6) to [out = 180, in = 225] (-0.8 - 0.08,-0.8 + 0.15);
\end{scope}

\begin{scope}[xshift = 4.5 cm, yshift = 0 cm]
\draw[dashed] (0,-0.18) circle (1 cm);
\draw (-0.15,0.866-0.05) to [out = -90, in = 120] (0.2, 0) to [out = -60, in = 135] (0.8 - 0.095,-0.8 - 0.1);
\draw (0.15,0.866-0.05) to [out = -90, in = 45] (-0.8 - 0.08,-0.8 + 0.15);
\draw (-0.8 + 0.095,-0.8 - 0.1) to [out = 45, in = 135] (0.8 + 0.08,-0.8 + 0.15);
\draw (-0.15,0.866-0.05) to [out = 90, in = 135] (1.2,0.7) to [out = -45, in = -45] (0.8 - 0.095,-0.8 - 0.1);
\draw (0.15,0.866-0.05) to [out = 90, in = 45] (-1.20, 0.7) to [out = 225, in = 225]  (-0.8 + 0.095,-0.8 - 0.1);
\draw (0.8 + 0.08,-0.8 + 0.15) to [out = -45, in = 0] (0, -1.6) to [out = 180, in = 225] (-0.8 - 0.08,-0.8 + 0.15);
\end{scope}

\begin{scope}[xshift = 4.5 cm, yshift = -3 cm]
\draw[dashed] (0,-0.18) circle (1 cm);
\draw (0.15,0.866-0.05) to [out = -90, in = 135] (0.8 - 0.095,-0.8 - 0.1);
\draw (-0.15,0.866-0.05) to [out = -90, in = 45] (-0.8 + 0.095,-0.8 - 0.1);
\draw (-0.8 - 0.08,-0.8 + 0.15) to [out = 45, in = 180] (0, -0.5) to [out = 0, in = 135] (0.8 + 0.08,-0.8 + 0.15);
\draw (-0.15,0.866-0.05) to [out = 90, in = 135] (1.2,0.7) to [out = -45, in = -45] (0.8 - 0.095,-0.8 - 0.1);
\draw (0.15,0.866-0.05) to [out = 90, in = 45] (-1.20, 0.7) to [out = 225, in = 225]  (-0.8 + 0.095,-0.8 - 0.1);
\draw (0.8 + 0.08,-0.8 + 0.15) to [out = -45, in = 0] (0, -1.6) to [out = 180, in = 225] (-0.8 - 0.08,-0.8 + 0.15);
\end{scope}

\end{scope}

\begin{scope}[xshift = 13.5 cm]
\begin{scope}
\draw[dashed] (0,-0.18) circle (1 cm);
\draw (-0.15,0.866-0.05) to [out = -90, in = 120] (0.2, 0) to [out = -60, in = 135] (0.8 - 0.095,-0.8 - 0.1);
\draw (0.15,0.866-0.05) to [out = -90,  in = 60] (-0.2, 0) to [out = 240, in = 45] (-0.8 + 0.095,-0.8 - 0.1);
\draw (-0.8 - 0.08,-0.8 + 0.15) to [out = 45, in = 180] (0, -0.5) to [out = 0, in = 135] (0.8 + 0.08,-0.8 + 0.15);
\draw (-0.15,0.866-0.05) to [out = 90, in = 135] (1.2,0.7) to [out = -45, in = -45] (0.8 - 0.095,-0.8 - 0.1);
\draw (0.15,0.866-0.05) to [out = 90, in = 45] (-1.20, 0.7) to [out = 225, in = 225]  (-0.8 + 0.095,-0.8 - 0.1);
\draw (0.8 + 0.08,-0.8 + 0.15) to [out = -45, in = 0] (0, -1.6) to [out = 180, in = 225] (-0.8 - 0.08,-0.8 + 0.15);
\end{scope}

\end{scope}

\end{tikzpicture}$$
\caption{Case 6:  One circle to three circles.}
\label{fig:Case6}
\end{figure}

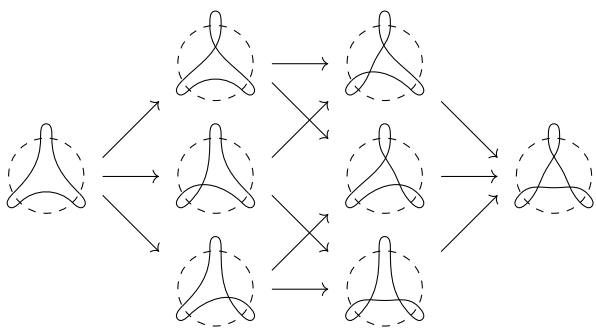
\begin{figure}[H]
$$\begin{tikzpicture}[scale = 0.5]

\begin{scope}
\draw[dashed] (0,-0.18) circle (1 cm);
\draw (0.15,0.866-0.05) to [out = -90, in = 135] (0.8 + 0.08,-0.8 + 0.15);
\draw (-0.15,0.866-0.05) to [out = -90, in = 45] (-0.8 - 0.08,-0.8 + 0.15);
\draw (-0.8 + 0.095,-0.8 - 0.1) to [out = 45, in = 135] (0.8 - 0.095,-0.8 - 0.1);

\draw (0.15,0.866-0.05) to [out = 90, in = 0] (0, 1.2) to [out = 180, in = 90] (-0.15,0.866-0.05);
\draw (0.8 + 0.08,-0.8 + 0.15) to [out = -45, in = 45]  (0.8 + 0.2,-0.8 - 0.2) to [out = 225, in = -45]  (0.8 - 0.095,-0.8 - 0.1);
\draw (-0.8 - 0.08,-0.8 + 0.15) to [out = 225, in = 135] (-0.8 - 0.2,-0.8 - 0.2) to [out = -45, in = 225] (-0.8 + 0.095,-0.8 - 0.1);
\end{scope}

\begin{scope}[xshift = 1 cm, yshift = -0.2 cm]
\draw[->] (0.5,0.5) --(2,2);
\draw[->] (0.5,0) --(2,0);
\draw[->] (0.5,-0.5) --(2,-2);
\end{scope}

\begin{scope}[xshift = 5.5 cm, yshift = -0.2 cm]
\draw[->] (0.5,0.5) --(2,2);
\draw[->] (0.5,-0.5) --(2,-2);
\draw[->] (0.5,3) --(2,3);
\draw[->] (0.5,-3) --(2,-3);
\draw[->] (0.5,-2.5) --(2,-1);
\draw[->] (0.5,2.5) --(2,1);
\end{scope}

\begin{scope}[xshift = 8.5 cm, yshift = -0.2 cm]
\draw[->] (2,2) -- (3.5,0.5);
\draw[->] (2,0) -- (3.5,0);
\draw[->] (2,-2) -- (3.5,-0.5);
\end{scope}

\begin{scope}[xshift = 4.5 cm, yshift = 3 cm]
\draw[dashed] (0,-0.18) circle (1 cm);
\draw (-0.15,0.866-0.05) to [out = -90, in = 135] (0.8 + 0.08,-0.8 + 0.15);
\draw (0.15,0.866-0.05) to [out = -90, in = 45] (-0.8 - 0.08,-0.8 + 0.15);
\draw (-0.8 + 0.095,-0.8 - 0.1) to [out = 45, in = 135] (0.8 - 0.095,-0.8 - 0.1);

\draw (0.15,0.866-0.05) to [out = 90, in = 0] (0, 1.2) to [out = 180, in = 90] (-0.15,0.866-0.05);
\draw (0.8 + 0.08,-0.8 + 0.15) to [out = -45, in = 45]  (0.8 + 0.2,-0.8 - 0.2) to [out = 225, in = -45]  (0.8 - 0.095,-0.8 - 0.1);
\draw (-0.8 - 0.08,-0.8 + 0.15) to [out = 225, in = 135] (-0.8 - 0.2,-0.8 - 0.2) to [out = -45, in = 225] (-0.8 + 0.095,-0.8 - 0.1);
\end{scope}

\begin{scope}[xshift = 4.5 cm, yshift = 0 cm]
\draw[dashed] (0,-0.18) circle (1 cm);
\draw (0.15,0.866-0.05) to [out = -90, in = 135] (0.8 + 0.08,-0.8 + 0.15);
\draw (-0.15,0.866-0.05) to [out = -90, in = 45] (-0.8 + 0.095,-0.8 - 0.1);
\draw (-0.8 - 0.08,-0.8 + 0.15) to [out = 45, in = 135] (0.8 - 0.095,-0.8 - 0.1);

\draw (0.15,0.866-0.05) to [out = 90, in = 0] (0, 1.2) to [out = 180, in = 90] (-0.15,0.866-0.05);
\draw (0.8 + 0.08,-0.8 + 0.15) to [out = -45, in = 45]  (0.8 + 0.2,-0.8 - 0.2) to [out = 225, in = -45]  (0.8 - 0.095,-0.8 - 0.1);
\draw (-0.8 - 0.08,-0.8 + 0.15) to [out = 225, in = 135] (-0.8 - 0.2,-0.8 - 0.2) to [out = -45, in = 225] (-0.8 + 0.095,-0.8 - 0.1);
\end{scope}

\begin{scope}[xshift = 4.5 cm, yshift = -3 cm]
\draw[dashed] (0,-0.18) circle (1 cm);
\draw (0.15,0.866-0.05) to [out = -90, in = 135] (0.8 - 0.095,-0.8 - 0.1);
\draw (-0.15,0.866-0.05) to [out = -90, in = 45] (-0.8 - 0.08,-0.8 + 0.15);
\draw (-0.8 + 0.095,-0.8 - 0.1) to [out = 45, in = 135] (0.8 + 0.08,-0.8 + 0.15);

\draw (0.15,0.866-0.05) to [out = 90, in = 0] (0, 1.2) to [out = 180, in = 90] (-0.15,0.866-0.05);
\draw (0.8 + 0.08,-0.8 + 0.15) to [out = -45, in = 45]  (0.8 + 0.2,-0.8 - 0.2) to [out = 225, in = -45]  (0.8 - 0.095,-0.8 - 0.1);
\draw (-0.8 - 0.08,-0.8 + 0.15) to [out = 225, in = 135] (-0.8 - 0.2,-0.8 - 0.2) to [out = -45, in = 225] (-0.8 + 0.095,-0.8 - 0.1);
\end{scope}

\begin{scope}[xshift = 4.5 cm]
\begin{scope}[xshift = 4.5 cm, yshift = 3 cm]
\draw[dashed] (0,-0.18) circle (1 cm);
\draw (-0.15,0.866-0.05) to [out = -90, in = 135] (0.8 + 0.08,-0.8 + 0.15);
\draw (0.15,0.866-0.05) to [out = -90,  in = 60] (-0.2, 0) to [out = 240, in = 45] (-0.8 + 0.095,-0.8 - 0.1);
\draw (-0.8 - 0.08,-0.8 + 0.15) to [out = 45, in = 135] (0.8 - 0.095,-0.8 - 0.1);

\draw (0.15,0.866-0.05) to [out = 90, in = 0] (0, 1.2) to [out = 180, in = 90] (-0.15,0.866-0.05);
\draw (0.8 + 0.08,-0.8 + 0.15) to [out = -45, in = 45]  (0.8 + 0.2,-0.8 - 0.2) to [out = 225, in = -45]  (0.8 - 0.095,-0.8 - 0.1);
\draw (-0.8 - 0.08,-0.8 + 0.15) to [out = 225, in = 135] (-0.8 - 0.2,-0.8 - 0.2) to [out = -45, in = 225] (-0.8 + 0.095,-0.8 - 0.1);
\end{scope}

\begin{scope}[xshift = 4.5 cm, yshift = 0 cm]
\draw[dashed] (0,-0.18) circle (1 cm);
\draw (-0.15,0.866-0.05) to [out = -90, in = 120] (0.2, 0) to [out = -60, in = 135] (0.8 - 0.095,-0.8 - 0.1);
\draw (0.15,0.866-0.05) to [out = -90, in = 45] (-0.8 - 0.08,-0.8 + 0.15);
\draw (-0.8 + 0.095,-0.8 - 0.1) to [out = 45, in = 135] (0.8 + 0.08,-0.8 + 0.15);

\draw (0.15,0.866-0.05) to [out = 90, in = 0] (0, 1.2) to [out = 180, in = 90] (-0.15,0.866-0.05);
\draw (0.8 + 0.08,-0.8 + 0.15) to [out = -45, in = 45]  (0.8 + 0.2,-0.8 - 0.2) to [out = 225, in = -45]  (0.8 - 0.095,-0.8 - 0.1);
\draw (-0.8 - 0.08,-0.8 + 0.15) to [out = 225, in = 135] (-0.8 - 0.2,-0.8 - 0.2) to [out = -45, in = 225] (-0.8 + 0.095,-0.8 - 0.1);
\end{scope}

\begin{scope}[xshift = 4.5 cm, yshift = -3 cm]
\draw[dashed] (0,-0.18) circle (1 cm);
\draw (0.15,0.866-0.05) to [out = -90, in = 135] (0.8 - 0.095,-0.8 - 0.1);
\draw (-0.15,0.866-0.05) to [out = -90, in = 45] (-0.8 + 0.095,-0.8 - 0.1);
\draw (-0.8 - 0.08,-0.8 + 0.15) to [out = 45, in = 180] (0, -0.5) to [out = 0, in = 135] (0.8 + 0.08,-0.8 + 0.15);

\draw (0.15,0.866-0.05) to [out = 90, in = 0] (0, 1.2) to [out = 180, in = 90] (-0.15,0.866-0.05);
\draw (0.8 + 0.08,-0.8 + 0.15) to [out = -45, in = 45]  (0.8 + 0.2,-0.8 - 0.2) to [out = 225, in = -45]  (0.8 - 0.095,-0.8 - 0.1);
\draw (-0.8 - 0.08,-0.8 + 0.15) to [out = 225, in = 135] (-0.8 - 0.2,-0.8 - 0.2) to [out = -45, in = 225] (-0.8 + 0.095,-0.8 - 0.1);
\end{scope}

\end{scope}

\begin{scope}[xshift = 13.5 cm]
\begin{scope}
\draw[dashed] (0,-0.18) circle (1 cm);
\draw (-0.15,0.866-0.05) to [out = -90, in = 120] (0.2, 0) to [out = -60, in = 135] (0.8 - 0.095,-0.8 - 0.1);
\draw (0.15,0.866-0.05) to [out = -90,  in = 60] (-0.2, 0) to [out = 240, in = 45] (-0.8 + 0.095,-0.8 - 0.1);
\draw (-0.8 - 0.08,-0.8 + 0.15) to [out = 45, in = 180] (0, -0.5) to [out = 0, in = 135] (0.8 + 0.08,-0.8 + 0.15);

\draw (0.15,0.866-0.05) to [out = 90, in = 0] (0, 1.2) to [out = 180, in = 90] (-0.15,0.866-0.05);
\draw (0.8 + 0.08,-0.8 + 0.15) to [out = -45, in = 45]  (0.8 + 0.2,-0.8 - 0.2) to [out = 225, in = -45]  (0.8 - 0.095,-0.8 - 0.1);
\draw (-0.8 - 0.08,-0.8 + 0.15) to [out = 225, in = 135] (-0.8 - 0.2,-0.8 - 0.2) to [out = -45, in = 225] (-0.8 + 0.095,-0.8 - 0.1);
\end{scope}

\end{scope}

\end{tikzpicture}$$
\caption{Case 7:  One circle to one circle.}
\label{fig:Case7}
\end{figure}

\section{Computations of the Vertex Polynomial}\label{app:computations}
In this section we provide Mathematica code for the computation of the various vertex polynomials for (signed) ribbon graphs.  We begin with a description of a vertex planar diagram, or VPD notation. 

Let $\Gamma$ be a ribbon diagram for a trivalent graph $G(V,E)$.  For simplicity, we assume all edges are positive edges (negative edges will be described later).  Enumerate the edges, $E = \{e_1,\ldots, e_k\}$, and split each edge into two half-edges.  We then number the half-edges associated to $e_i$ with the labels $2i-1$ and $2i$ (see \Cref{fig:VPD}).  Note that there are two ways to do this for each half-edge, but the choice is arbitrary, and equivalent to the choice of an orientation of each edge that runs from the lesser to the greater label.

\begin{figure}[H]
\includegraphics[scale=1]{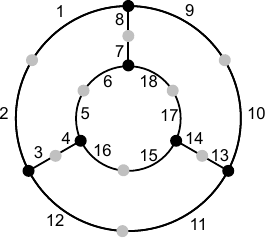}
\caption{Labeling for the vertex planar diagram or VPD notation.  The original vertices of the graph are shown in black, and the vertices used to split the half-edges are shown in grey.  }\label{fig:VPD}
\end{figure}

We record the VPD notation using a series of $3$-tuples to specify the cyclic order of the half-edges at each of the original vertices of the graph, recorded counterclockwise around each vertex.  For the $P3$ shown in \Cref{fig:VPD} we obtain the code:
$$P3 := G[V[1,8,9],V[6,18,7],V[4,16,5],V[14,17,15],V[2,12,3],V[10,13,11]]$$
This encoding scheme is equivalent to the choice of a \emph{signed rotation system}, which is a pair of permutations  $(\sigma, \theta)$, each acting on the set of integers $i = 1, \ldots 2k$, where $\theta$ is a fixed-point free involution, and the group generated by $\sigma$ and $\theta$ acts transitively on the integers $i = 1,\ldots 2k$.  For the VPD code, $\theta$ is assumed to be a product of transpositions, $\Pi_{i=1}^{k} (2i-1 \ \  2i)$, and therefore suppressed from the VPD notation.  The VPD notation itself records a triple for each vertex, which is the $3$-cycle specifying the cyclic orientation of the edges at the vertex.

After copying and pasting the code below to a Mathematica notebook, along with the VPD notation above, one may run the following command to obtain the $2$-color vertex polynomial: \verb$TwoColorVert[P3]$.
$$(1 + q) (-6 q^3 (1 + q)^2 - 6 q^{15} (1 + q)^2 + (1 + q)^4 + q^{18} (1 + q)^4 - 4 q^9 (3 + 2 (1 + q)^2) + 3 q^6 (2 + 3 (1 + q)^2) + 3 q^{12} (2 + 3 (1 + q)^2))$$
Expanding the polynomial, we see that it is equivalent to the calculation in \Cref{ex:P3Calc}.

\begin{example}\label{ex:theta}
The $\theta$ graph shown in \Cref{fig:VPDTheta} has VPD code
$$theta :=G[V[1,6,4],V[2,3,5]].$$
The vertex polynomial is calculated with the command \verb$Vert[theta]$ and produces the output $2n(n^2 - 1)$.

\begin{figure}[H]
\includegraphics[scale=1]{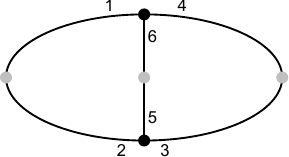}
\caption{VPD notation labels for the $\theta$ graph. }\label{fig:VPDTheta}
\end{figure}
\end{example}

To encode a negative edge, which allows one to encode VPD notation for non-orientable ribbon graphs, we use a minus sign on the half-edge with an odd label.  Otherwise, the code is generated in the same manner as described above.

\begin{example} \label{ex:VPDThetaNeg}
The $\theta$ graph shown in \Cref{fig:VPDThetaNeg} has VPD code
$$thetaNeg:=G[V[1,6,4],V[2,3,-5]].$$
The vertex polynomial is calculated with the command \verb$Vert[thetaNeg]$ and produces the output $2n(n-1)$.

\begin{figure}[H]
\includegraphics[scale=1]{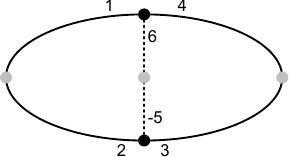}
\caption{VPD notation labels for the $\theta$ graph with a negative edge. }\label{fig:VPDThetaNeg}
\end{figure}
This example shows that \Cref{thm:VertexParity} does not hold for non-orientable graphs.  Compare this to the next example, where we obtain an even polynomial.
\end{example}

\begin{example}\label{ex:K4}
Blowing up the $\theta$ graph at a single vertex produces the $K_4$ graph shown in \Cref{fig:VPDK4}.  The VPD code is given by
$$K4 := G[V[1, 10, 6], V[9, 12, 8], V[4, 5, 7], V[2, 3, 11]].$$
The vertex polynomial is calculated with the command \verb$Vert[K4]$, and produces the output $2n^2(n^2-1)$.

\begin{figure}[H]
\includegraphics[scale=.9]{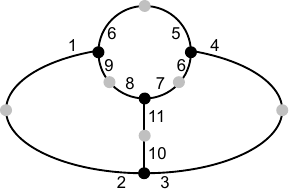}
\caption{Blowing up a $\theta$ graph at one vertex produces a $K_4$. }\label{fig:VPDK4}
\end{figure}

\end{example}

\begin{example}\label{ex:P3}
Blowing up the $\theta$ graph at both vertices produces the $3$-prism, i.e., $\theta^\flat= P3$, shown in \Cref{fig:VPDP3}.  The VPD code is given by
$$thetab:=G[V[1,10,6],V[8,9,18],V[4,5,7],V[2,12,13],V[14,15,17],V[3,16,11]].$$
The vertex polynomial is calculated with the command \verb$Vert[thetab]$, and produces the output $2n^3(n^2-1)$.

\begin{figure}[H]
\includegraphics[scale=.85]{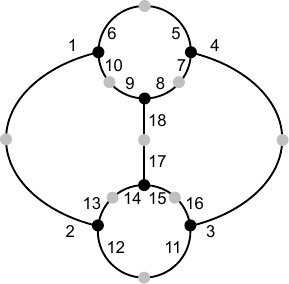}
\caption{Blowing up a $\theta$ graph at both vertices produces a $P3$. }\label{fig:VPDP3}
\end{figure}

\end{example}

It should be noted that \Cref{ex:theta}, \Cref{ex:K4}, and \Cref{ex:P3} are all related via blowing up at a single vertex at a time.  The vertex polynomial of each is related to the vertex polynomial of the previous one by a factor of $n$, which illustrates \Cref{thm:VertexPolyBlowUp}. 

\begin{example}\label{ex:K33Computation}
The $K_{3,3}$ of \Cref{ex:K33Filtered} is shown again in \Cref{fig:VPDK33}, and has VPD code:
$$K33 := G[V[1,14,12],V[11,16,10],V[9,18,8],V[7,6,13],V[5,4,15],V[3,2,17]].$$
\begin{figure}[H]
\includegraphics[scale=.37]{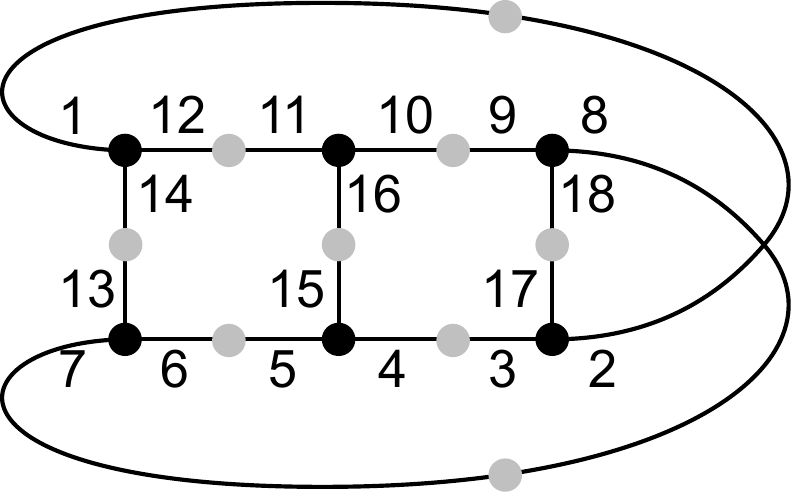}
\caption{The $K_{3,3}$ graph. }\label{fig:VPDK33}
\end{figure}
\end{example}
Running the command $Vert[K33]$ we obtain that the vertex polynomial is $0$.  However, running the command $TwoColorVert[K33]$ we obtain
$$(1 + q) (-1 + q^3)^2 ((1 + q)^2 + q^{12} (1 + q)^2 - 
   2 q^3 (1 + (1 + q)^2) - 2 q^9 (1 + (1 + q)^2) + 
   2 q^6 (1 + 2 (1 + q)^2)).$$
Since the $2$-color vertex polynomial is nonzero, this indicates that the filtered $2$-color vertex homology is non-trivial, as observed in \Cref{ex:K33Filtered}.  Of course, to determine the ranks of the homology groups requires more work but is feasible to do by hand.

\begin{example}\label{ex:dodecahedron}
The dodecahedron of \Cref{fig:DodecVDP} has VPD code:
\begin{multline*}
Dodec := G[V[2, 3, 13], V[1, 11, 10], V[8, 9, 19], V[6, 7, 17], V[4, 5, 15], V[12, 21, 40],V[22, 23, 41],\\  
V[24, 14, 25], V[26, 27, 43], V[16, 29, 28],  V[30, 31, 45], V[18, 33, 32],  V[34, 35, 47],V[36, 20, 37],\\ 
V[38, 39, 49],V[60, 50, 51], V[52, 42, 53], V[54, 44, 55],V[46, 57, 56], V[58, 48, 59]].
\end{multline*}

\begin{figure}[H]
\includegraphics[scale=.37]{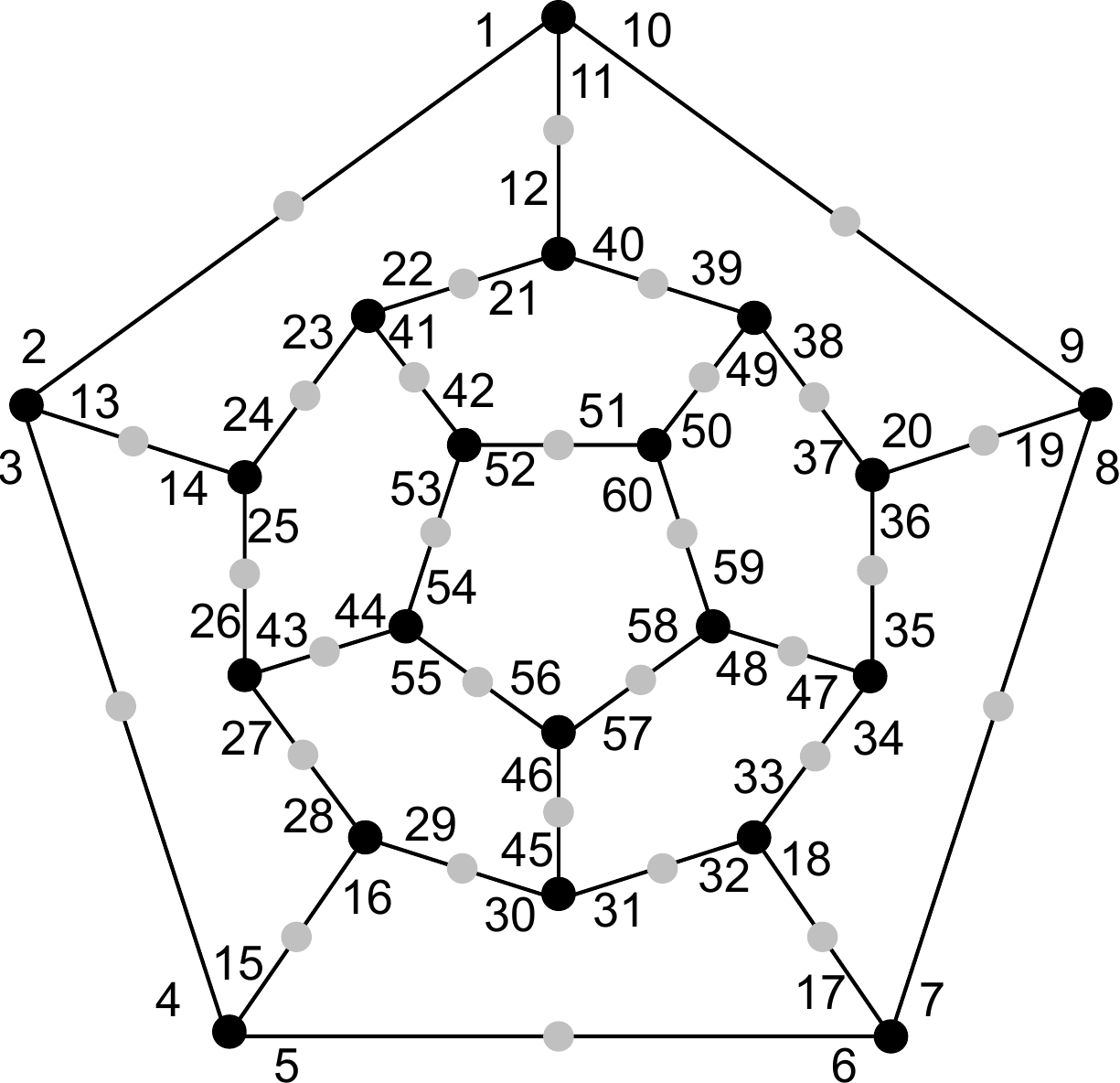}
\caption{The dodecahedron graph. }\label{fig:DodecVDP}
\end{figure}
\end{example}
Running the command $Vert[Dodec]$ we obtain that the vertex polynomial is 
$$2 (n+1) n^2 (n-1) (240 - 116 n^2 + 114 n^4 + 11 n^6 + n^8).$$

This example illustrates the difference between computations of filtered vertex homology and, at least, initial strategies employed to estimate the ranks of Kronheimer and Mrowka's instanton homology (see \cite{KM3, KR, Boozer}). Briefly, the strategy obtains (lower) bounds on the rank of instanton homology of a planar graph (a web $K$) by computing the rank of a bilinear form restricted to a subspace by evaluating it on clever choices of closed foams that contain the web.  Even though each closed foam can be computed algorithmically, the procedure to test the forms may not  be since there is potentially an infinite number of foams that might need to be tested. Boozer in \cite{Boozer} used these techniques to further restrict the bounds on the number of Tait colorings of the dodecahedron to $$58\leq \dim J^\flat(Dodec) \leq 60 \leq \dim J^\sharp(Dodec) \leq 68,$$ where $J^\flat$ and $J^\sharp$ are functors from from the category of webs and foams to the category of vector spaces. (The dimension of $J^\sharp(K)$ was conjectured to equal the number of Tait colorings in \cite{KM3}.)

The computations of our homology theories, on the other hand, are completely algorithmic and give exact values. For example, Theorem~\ref{thm:oddMatchingsCancel} shows that the number of Tait colorings of the dodecahedron graph is $60$ by evaluating $Vert[Dodec]$ at $n=2$.

Note that the calculation of $Vert[Dodec]$ will take about a day to run.  Nevertheless, this example also shows that the polynomial can be calculated in a reasonable amount of time even for relatively ``large" graphs. 
   
\subsection{Mathematica Code} The code below may be copied and pasted into a Mathematica notebook. It uses the VPD notation described above to encode graphs.  The function \verb$Vert$ computes the vertex polynomial, while the functions \verb$TwoColorVert$, \verb$ThreeColorVert$, and \verb$FourColorVert$ may be used to compute the $n$-color vertex polynomials for $n=2,3,4$.

\begin{remark}[An important pointer to running the code correctly] When pasting in the code, an error may be encountered if one does not paste the code from each successive page on a new line (not a new input cell) in Mathematica.
\end{remark}

\noindent \verb$ToProduct[graph_G] := ($\\
\verb$  Product[graph[[i]], {i, 1, Length[graph]}])$\\
\verb$rule1 = {V[a_, b_, c_] :> $\\
\verb$    A arc[{a, 2}, {b, 1}] arc[{b, 2}, {c, 1}] arc[{c, 2}, {a, 1}] + $\\
\verb$      B arc[{a, 2}, {c, 1}] arc[{b, 2}, {a, 1}] arc[{c, 2}, {b, 1}] /;$\\
\verb$      OddQ[a] && OddQ[b] && OddQ[c],$\\
\verb$   V[a_, b_, c_] :> $\\
\verb$    A arc[{a, 1}, {b, 1}] arc[{b, 2}, {c, 1}] arc[{c, 2}, {a, 2}] + $\\
\verb$      B arc[{a, 1}, {c, 1}] arc[{b, 2}, {a, 2}] arc[{c, 2}, {b, 1}] /;$\\
\verb$      EvenQ[a] && OddQ[b] && OddQ[c],$\\
\verb$   V[a_, b_, c_] :> $\\
\verb$    A arc[{a, 2}, {b, 2}] arc[{b, 1}, {c, 1}] arc[{c, 2}, {a, 1}] + $\\
\verb$      B arc[{a, 2}, {c, 1}] arc[{b, 1}, {a, 1}] arc[{c, 2}, {b, 2}] /;$\\
\verb$      OddQ[a] && EvenQ[b] && OddQ[c],$\\
\verb$   V[a_, b_, c_] :> $\\
\verb$    A arc[{a, 2}, {b, 1}] arc[{b, 2}, {c, 2}] arc[{c, 1}, {a, 1}] + $\\
\verb$      B arc[{a, 2}, {c, 2}] arc[{b, 2}, {a, 1}] arc[{c, 1}, {b, 1}] /;$\\
\verb$      OddQ[a] && OddQ[b] && EvenQ[c],$\\
\verb$   V[a_, b_, c_] :> $\\
\verb$    A arc[{a, 1}, {b, 2}] arc[{b, 1}, {c, 1}] arc[{c, 2}, {a, 2}] + $\\
\verb$      B arc[{a, 1}, {c, 1}] arc[{b, 1}, {a, 2}] arc[{c, 2}, {b, 2}] /;$\\
\verb$      EvenQ[a] && EvenQ[b] && OddQ[c],$\\
\verb$   V[a_, b_, c_] :> $\\
\verb$    A arc[{a, 1}, {b, 1}] arc[{b, 2}, {c, 2}] arc[{c, 1}, {a, 2}] + $\\
\verb$      B arc[{a, 1}, {c, 2}] arc[{b, 2}, {a, 2}] arc[{c, 1}, {b, 1}] /;$\\
\verb$      EvenQ[a] && OddQ[b] && EvenQ[c],$\\
\verb$   V[a_, b_, c_] :> $\\
\verb$    A arc[{a, 2}, {b, 2}] arc[{b, 1}, {c, 2}] arc[{c, 1}, {a, 1}] + $\\
\verb$      B arc[{a, 2}, {c, 2}] arc[{b, 1}, {a, 1}] arc[{c, 1}, {b, 2}] /;$\\
\verb$      OddQ[a] && EvenQ[b] && EvenQ[c],$\\
\verb$   V[a_, b_, c_] :> $\\
\verb$    A arc[{a, 1}, {b, 2}] arc[{b, 1}, {c, 2}] arc[{c, 1}, {a, 2}] + $\\
\verb$      B arc[{a, 1}, {c, 2}] arc[{b, 1}, {a, 2}] arc[{c, 1}, {b, 2}] /;$\\
\verb$      EvenQ[a] && EvenQ[b] && EvenQ[c]};$\\
\verb$rule2 = {arc[{a_, x_}, y__] arc[{b_, x_}, z__] :> $\\
\verb$    arc[y, z] /; OddQ[a] && a > 0 && b == a + 1, $\\
\verb$   arc[{a_, x_}, y__] arc[z__, {b_, x_}] :> $\\
\verb$    arc[y, z] /; OddQ[a] && a > 0 && b == a + 1,$\\
\verb$   arc[y__, {a_, x_}] arc[{b_, x_}, z__] :> $\\
\verb$    arc[y, z] /; OddQ[a] && a > 0 && b == a + 1,$\\
\verb$   arc[y__, {a_, x_}] arc[z__, {b_, x_}] :> $\\
\verb$    arc[y, z] /; OddQ[a] && a > 0 && b == a + 1,$\\
   
\verb$   arc[{a_, 1}, y__] arc[{b_, 2}, z__] :> $\\
\verb$    arc[y, z] /; OddQ[a] && a < 0 && b == Abs[a] + 1, $\\
\verb$   arc[{a_, 1}, y__] arc[z__, {b_, 2}] :> $\\
\verb$    arc[y, z] /; OddQ[a] && a < 0 && b == Abs[a] + 1,$\\
\verb$   arc[{a_, 2}, y__] arc[{b_, 1}, z__] :> $\\
\verb$    arc[y, z] /; OddQ[a] && a < 0 && b == Abs[a] + 1, $\\
\verb$   arc[{a_, 2}, y__] arc[z__, {b_, 1}] :> $\\
\verb$    arc[y, z] /; OddQ[a] && a < 0 && b == Abs[a] + 1,$\\
   
\verb$   arc[y__, {a_, 1}] arc[{b_, 2}, z__] :> $\\
\verb$    arc[y, z] /; OddQ[a] && a < 0 && b == Abs[a] + 1,$\\
\verb$   arc[y__, {a_, 1}] arc[z__, {b_, 2}] :> $\\
\verb$    arc[y, z] /; OddQ[a] && a < 0 && b == Abs[a] + 1,$\\
\verb$   arc[y__, {a_, 2}] arc[{b_, 1}, z__] :> $\\
\verb$    arc[y, z] /; OddQ[a] && a < 0 && b == Abs[a] + 1,$\\
\verb$   arc[y__, {a_, 2}] arc[z__, {b_, 1}] :> $\\
\verb$    arc[y, z] /; OddQ[a] && a < 0 && b == Abs[a] + 1};$\\
\verb$rule3 = {arc[x__, y__] :> L};$\\
\verb$RawBracket[t_] := Simplify[(t /. rule1 // Expand) //. rule2 //. rule3]$\\
\verb$rule4 = {A :> 1, B :> -q^3, L :> 1 + q};$\\
\verb$TwoColorVert[t_] := Simplify[RawBracket[ToProduct[t]] /. rule4]$\\
\verb$rule5 = {A :> 1, B :> -q^3, L :> q + 1 + q^(-1)};$\\
\verb$ThreeColorVert[t_] := Simplify[RawBracket[ToProduct[t]] /. rule5]$\\
\verb$rule6 = {A :> 1, B :> -q^6, L :> q^2 + q + 1 + q^(-1)};$\\
\verb$FourColorVert[t_] := Simplify[RawBracket[ToProduct[t]] /. rule6]$\\
\verb$rule7 = {A :> 1, B :> -1, L :> n};$\\
\verb$Vert[t_] := Simplify[RawBracket[ToProduct[t]] /. rule7]$\\

\end{document}